\documentclass[11pt]{article}
    \usepackage{color}
    \usepackage[centertags]{amsmath}
    \usepackage{amsfonts} 
    \usepackage{fullpage} 
    \usepackage{mathrsfs}
    \usepackage{amsthm}
    \usepackage{amssymb}
    \usepackage{graphicx}
    \usepackage{hyperref}
    \usepackage{epsfig}
    \usepackage[toc,page]{appendix}
    \usepackage{setspace}

\newtheorem{thm}{Theorem}[section]

\newtheorem{lem}[thm]{Lemma}
\newtheorem{prop}[thm]{Proposition}
\theoremstyle{definition}
\newtheorem{defn}[thm]{Definition}
\theoremstyle{remark}
\newtheorem{rem}[thm]{Remark}
\newcommand{\R}{\mathbb R} 
\newcommand{\C}{\mathbb C} 

\newcommand{\E}{\mathbb E}

\newcommand{\tr}{\text{tr}}
\newcommand{\Tr}{\text{Tr}}

\renewcommand\Im{\operatorname{\mathfrak{Im}}}

\numberwithin{equation}{section}
\begin{document}
\title{Central Limit Theorem for Linear Eigenvalue Statistics for Submatrices of Wigner
Random Matrices}
\author{Lingyun Li
\thanks{Beijing Technology And Business University,  No. 11/33 Fangshan Road, Haidian District, Beijing, 100048, China, 
lilingyun@btbu.edu.cn}
\and
Matthew Reed
\thanks{ Department of Mathematics, University of California, Davis, One Shields Avenue, Davis, CA 95616-8633, USA, currently at Zumper, San Francisco, 
USA, matthewcreed86@gmail.com}
\and
Alexander Soshnikov
\thanks{Department of Mathematics, University of California, Davis, One Shields Avenue, Davis, CA 95616-8633, USA soshniko@math.ucdavis.edu;
research has been supported in part by the Simons Foundation award \#312391}}

\maketitle

\begin{abstract}
We prove the Central Limit Theorem for finite-dimensional vectors of linear eigenvalue statistics
of submatrices of Wigner random matrices under the assumption that test functions are sufficiently smooth.
We connect the asymptotic covariance to a family of correlated Gaussian Free Fields.
\end{abstract}
\section{Introduction}

Wigner random matrices were introduced by E. Wigner in the 1950s (see e.g. \cite{Wig1},\cite{Wig2}, \cite{Meh}) to study energy levels of heavy nuclei.
Let $ \{W_{jj}\}_{j=1}^{n}$ and $ \{W_{jk}\}_{1\leq j<k \leq n}$ be two independent families of
independent and identically distributed  real-valued random variables satisfying:
\begin{equation}
\E W_{jk } = 0, \ \ \E |W_{jk}|^{2} = 1 \ \  \text{for} \ \ j < k, \ \ \text{and} \ \
 \E[ W_{jj}^{2}] = \sigma^{2}.
 \label{condentries}
\end{equation}
Set $ W = (W_{jk})_{j,k=1}^{n}$ with $W_{jk} = W_{kj}$. The Wigner Ensemble of normalized
real symmetric $ n \times n$ matrices consists of matrices $M$ of the form
\begin{equation}
M = \frac{1}{\sqrt{n}}W.
\end{equation}

The archetypal example of a Wigner real symmetric random matrix is the Gaussian Orthogonal Ensemble (GOE) defined as (\cite{Meh})
\begin{equation}
A=\frac{1}{2}(B+B^t),
\end{equation}
where the entries of $B$ are i.i.d. real Gaussian random variables with zero mean and variance $1/2.$

Wigner Hermitian random matrices are defined in a similar fashion. Specifically, we assume that 
$ \{W_{jj}\}_{j=1}^{n}$ and $ \{W_{jk}\}_{1\leq j<k \leq n}$ are two independent families of
independent and identically distributed  real, correspondingly complex random variables satisfying (\ref{condentries}).
The archetypal example of a Wigner Hermitian random matrix is the Gaussian Unitary Ensemble (GUE) 
\begin{equation}
A=\frac{1}{2}(B+B^*),
\end{equation}
where  the entries of $B$ are i.i.d. complex standard Gaussian random variables (\cite{Meh}).

Over the last sixty years, Random Matrix Theory has developed many exciting connections to Quantum Chaos \cite{BGS}, Quantum Gravity \cite{FK}, 
Mesoscopic Physics \cite{SB}, Numerical Analysis \cite{ER}, Theoretical Neuroscience \cite{Som}, Optimal Control \cite{Chow}, Number Theory \cite{Keat}, 
Integrable Systems \cite{Har}, Combinatorics \cite{Rom}, Random Growth Models \cite{Joh}, Multivariate Statistics \cite{John}, 
and many other fields of Science and Engineering.

For a real symmetric (Hermitian) matrix $M$ of order $n,$ its  
empirical distribution of the eigenvalues is defined as 
$\mu_M = \frac{1}{n} \sum_{i=1}^{n} \delta_{\lambda_i},$ where  $\lambda_1 \leq \ldots \leq \lambda_n$ are the (ordered) eigenvalues of $M.$
The Wigner semicircle law states that for any bounded continuous test function $\varphi: \R \to \R,$ the linear statistic 
\begin{equation}
 \frac{1}{n} \sum_{i=1}^n \varphi(\lambda_i) = \frac{1}{n} \*\Tr(\varphi(M))=:\tr_n(\varphi(M)) 
\end{equation}
converges to $\int \varphi(x) \* d \mu_{sc}(dx) $ in probability, where $\mu_{sc}$ is determined by its density

\begin{equation}
\label{polukrug}
\frac{d \mu_{sc}}{dx}(x) = \frac{1}{2\pi} \sqrt{ 4 - x^2} \mathbf{1}_{[-2, 2]}(x),
\end{equation}
see e.g. \cite{Wig2},  \cite{BG}, \cite{cite_key1}.

The Gaussian fluctuation for linear statistics $\sum_{i=1}^n \varphi(\lambda_i)$ has been extensively studied since the pioneering paper by Jonsson 
\cite{Jon}. We refer the reader to \cite{Johansson1}, \cite{SinSo}, \cite{B.spectral}, \cite{cite_key3}, \cite{cite_key4}, \cite{Bandmodel}, \cite{LS},
\cite{LodS}, and references therein.
The goal of this paper is to prove the central limit theorem for the joint distribution of 
linear eigenvalue statistics for submatrices of Wigner random matrices. 

The rest of the paper is organized as follows. We formulate our results in Section 2. Theorem 2.1 is proved in Section 3. Theorem 2.2 is proved in 
Section 4. Auxiliary results are discussed in the Appendices.

Research of the last author has been partially supported by the Simons Foundation Collaboration Grant for Mathematicians \# 312391.

\section{Statement of Main Results}
This section is devoted to formulation of the main results of the paper.

For a generic random variable $\xi$, in what follows denote by $\xi^{\circ} := \xi -\E[\xi]$. For a finite set $B \subset  \{1,2,\ldots,n\}$ denote 
by $M(B)$ the submatrix of $M$ formed by the 
entries corresponding to intersections of rows and columns of $M$  marked by the indices in $B$, which 
inherits  the ordering. For example, 
\begin{equation}
M(\{1,3\}) = \begin{pmatrix} M_{11} & M_{13} \\ M_{31} & M_{33} \end{pmatrix}.  
\end{equation}

Let $ \mathcal{B}_{1}, \cdots, \mathcal{B}_{d}$ be infinite subsets  of $\mathbb{N}$ such that $\mathcal{B}_{i}, \ 1\leq i\leq d, $ and their pairwise 
intersections have positive densities.  Denote
\begin{eqnarray}
\label{u1}
& & B_{i}^n=\mathcal{B}_i \cap \{1,2,\ldots, n\}, \  1\leq i\leq d, \\
\label{u2}
& &  n_{i} = |B_{i}^n|,  \  1\leq i\leq d, \\
\label{u3}
& & n_{lm} = |B_{l}^n \cap B_{m}^n|,  \  1 \leq l \leq m \leq d.
\end{eqnarray}
We assume that the following limits exist:
\begin{equation}
\label{gamma_limits}
\gamma_{l} := \lim_{n \rightarrow \infty} \frac{n_{l}}{n} > 0,  
\ \ \gamma_{lm} := \lim_{n \rightarrow \infty} \frac{n_{lm}}{n}, \ \  1 \leq l \leq m \leq d.
\end{equation}
If it does not lead to ambiguity, we will omit the superindex $n$ in the notation for $B_{i}^n, \ 1\leq i\leq d.$
For an $n \times n$ matrix $M$ and $ B \subset \{1,2,\ldots, n\}$, 
consider a spectral linear statistic 
$ \sum_{l = 1}^{|B|} \varphi(\lambda_{l}), \ $
where $ \{\lambda_{l} \}_{l = 1}^{|B|} $ are the eigenvalues of the submatrix $ M(B)$. We are going to study the joint fluctuations of 
linear statistics of the eigenvalues. It will be beneficial later to view the submatrices from a different 
perspective. Consider the matrix $P^{B} = \text{diag}(P_{jj}^{B})$, which projects onto the subspace corresponding to indices in $B$, i.e.
\begin{equation}
\label{def-proj}
P_{jj}^{B} = \mathbf{1}_{\{j \in B\}}, \ 1 \leq j \leq n.
\end{equation}
Define
\begin{eqnarray}
& &  M^{B}_{}: = P^{B}MP^{B}, \\
& & \mathcal{N}_{B}[\varphi]:=  \  \sum_{l = 1}^{n} \varphi(\lambda_{l}^{B})  = Tr(\varphi(M^{B})),
\end{eqnarray}
where $ \{ \lambda_{l}^{B} \}_{l = 1}^{n}$ are the eigenvalues of $ M^{B}$. 
Note that the spectra of $M^{B}$ and $M(B)$ differ only by a zero eigenvalue of multiplicity $n-|B|.$ 
As a result, when we consider the linear statistics of their eigenvalues the extra terms $ (n-|B|)\*\varphi(0)$ cancel
once we center these random variables.  In general, when considering multiple sequences $B_l$, in order to simplify the notation we will write 
\begin{equation}
M^{(l)} := M^{B_l}, \ \ \ \ P^{(l)}:=P^{B_l},\ \ \ \ \mathcal{N}_n^{(l)}[\varphi] := \mathcal{N}_{B_l}[\varphi], 
\ \ \ \ \mathcal{N}_n^{(l)\circ}[\varphi]=\mathcal{N}_n^{(l)}[\varphi]-\E\{\mathcal{N}_n^{(l)}[\varphi]\}.
\end{equation}
Also, denote by $P^{(l,r)}$ the matrix which projects onto the subspace corresponding to the indices in the intersection $B_l \cap B_r$, i.e.
\begin{equation}
\label{def-proj1}
P^{(l,r)} = P^{(l)} P^{(r)}=P^{(r)} P^{(l)} .
\end{equation}
Recall that a test function $\varphi : \mathbb{R} \rightarrow \mathbb{R}$ belongs to the Sobolev space $\mathcal{H}_{s}$ if
\begin{equation}
||\varphi ||^{2}_{s}: = \int_{-\infty}^{\infty} (1+|t|)^{2s}|\widehat{\varphi}(t)|^{2}dt \ < \ \infty,
\label{ineq:fourier}
\end{equation}
where $\widehat{\varphi}$ is its Fourier transform.
First we consider Gaussian Wigner matrices.
\begin{thm}
\label{thm-GaussCLT}
Let  $W = \{W_{jk}:W_{jk}=W_{kj}\}_{j,k=1}^{n}$ be an $ n \times n$ real symmetric random matrix with Gaussian entries satisfying (\ref{condentries}) and $M = n^{-1/2}W$. Let $\mathcal{B}_1, \ldots \mathcal{B}_d$ be infinite subsets of $\mathbb{N}$ satisfying (\ref{u1}-\ref{gamma_limits}). Let $ \varphi_{1}, \cdots, \varphi_{d} : \mathbb{R} \rightarrow \mathbb{R}$ be test functions that satisfy the regularity condition $ ||\varphi_{l}||_{s} < \infty$, for some $ s > \frac{5}{2}. $ Then the random vector
\begin{equation}
\label{randvec}
  (\mathcal{N}_n^{(1)\circ}[\varphi_{1}], \dots, \ \mathcal{N}_n^{(d)\circ}[\varphi_{d}] ),
\end{equation}
converges in distribution to the zero mean Gaussian vector $ ( G_{1}, \cdots, G_{d}) \in \mathbb{R}^{d} $ 
with the covariance given by 
\begin{eqnarray}
\lefteqn{\mathbf{Cov}(G_{l}, G_{p})}\nonumber\\
 &=&\frac{\sigma^{2}}{4}\left( \varphi_{l}\right)_{1}\left(\varphi_{p}\right)_{1}(\frac{\gamma_{lp}}{\sqrt{\gamma_{l}\gamma_{p}}}) + \frac{1}{2}\sum^{\infty}_{k=2} k\left(\varphi_{l}\right)_{k}\left(\varphi_{p}\right)_{k} \left(\frac{\gamma_{lp}}{\sqrt{\gamma_{l}\gamma_{p}}}\right)^{k}\nonumber\\
&=&\frac{2}{\pi} \oint\limits_{ \begin{array}{cc}
            |z|^2=\gamma_{l} \\ \Im{z}>0
            \end{array} } 
\oint\limits_{\begin{array}{cc}
            |w|^2=\gamma_{p} \\ \Im{w}>0
            \end{array} }\varphi'_{l}\left(z+\frac{\gamma_{l}}{z}\right)\varphi'_{p} \left( w+\frac{\gamma_{p}}{w}\right)\frac{1}{2\pi}\ln\left| \frac{\gamma_{lp}-zw}{\gamma_{lp}-z\bar{w}} \right| \left( 1-\frac{\gamma_{l}}{z^{2}}  \right)\left(1-\frac{\gamma_{p}}{w^{2}}\right)dzdw\nonumber\\
          && + \ \frac{\gamma_{lp} (\sigma^{2} - 2)}{4 \pi^{2} \gamma_{l}\gamma_{p}}\int^{2\sqrt{\gamma_{l}}}_{-2\sqrt{\gamma_{l}}}  \frac{\lambda \varphi_{l}(\lambda)}{\sqrt{4\gamma_{l} - \lambda^{2}}}d\lambda\int^{2\sqrt{\gamma_{p}}}_{-2\sqrt{\gamma_{p}}}  \frac{\mu \varphi_{p}(\mu)}{\sqrt{4\gamma_{p} - \mu^{2}}}d\mu.
\label{GaussCov}
\end{eqnarray}

\end{thm}

In the expression for the covariance, $(\varphi_{l})_{k}$ denotes the coefficients in the expansion of $\varphi_{l}$ in the (rescaled) Chebyshev basis, i.e.
\begin{equation}
\label{eqn:chebyexp}
\varphi_{l}(x) = \sum^{\infty}_{k=0} (\varphi_{l})_{k} T^{\gamma_{l}}_k(x), \ \ \ (\varphi_{l})_{k} = \frac{2}{\pi} \int^{2\sqrt{\gamma_{l}}}_{-2\sqrt{\gamma_{l}}} \varphi_{l}(t) T^{\gamma_{l}}_{k}(t) \frac{dt}{\sqrt{4\gamma_{l}-t^{2}}}
\end{equation}
and 
\begin{equation}
T^{\gamma}_{k}(x) = cos\left(k\ arccos\left(\frac{x}{2\sqrt{\gamma}}\right)\right).
\end{equation}
Note the form of the kernel in the above contour integral expression for the covariance. Since it is the Greens function for the Laplacian 
on $\mathbb{H}$ with Dirichlet boundary conditions (appropriately scaled), we note that the limiting distributions form a family of 
correlated Gaussian free fields. This is consistent with the previous work of A. Borodin in \cite{cite_key2}, \cite{B2} 
for the covariance of linear eigenvalue statistics corresponding to polynomial test functions.
Now we formulate our result for non-Gaussian Wigner matrices. 
\begin{thm}
\label{thm-NG}
Let  $W = (W_{jk})_{j,k=1}^{n}$ be an $ n \times n$ random matrix and $M = n^{-1/2}W$. Let $\mathcal{B}_1, \ldots \mathcal{B}_d$ be infinite subsets of $\mathbb{N}$ satisfying (\ref{u1}-\ref{u3}) and (\ref{gamma_limits}). Assume the following conditions:\\
(1) All the entries of $W$ are independent random variables.\\
(2) The fourth moment of the non-zero off-diagonal entries does not depend on $n$:
\[\mu_4=\mathbb E\{W_{jk}^4\}.\]
(3) There exists a constant $\sigma_6$ such that for any $j,k$, $\E\{|W_{jk}|^{6}\}<\sigma_6$.\\
Let $ \varphi_{1}, \cdots, \varphi_{d} : \mathbb{R} \rightarrow \mathbb{R}$ be 
test functions that satisfy the regularity condition $ ||\varphi_{l}||_{s} < \infty$, for some $ s >5.5$. \ 
Then the random vector (\ref{randvec}) converges in distribution to the zero mean Gaussian vector 
$ ( \widetilde G_{1}, \cdots, \widetilde G_{d}) \in \mathbb{R}^{d} $ with the covariance given by 
\begin{equation}
\mathbf{Cov}(\widetilde G_{l}, \widetilde G_{p})=\mathbf{Cov}(G_l,G_p)+\frac{\kappa_4\gamma_{lp}^2}{2\pi^2\gamma_l^2\gamma_p^2 }\int_{-2\sqrt{\gamma_l}}^{2\sqrt{\gamma_l}}\varphi_l(\lambda)\frac{2\gamma_l-\lambda^2}{\sqrt{4\gamma_l-\lambda^2}}d\lambda
\int_{-2\sqrt{\gamma_p}}^{2\sqrt{\gamma_p}}\varphi_p(\mu)\frac{2\gamma_p-\mu^2}{\sqrt{4\gamma_p-\mu^2}}d\mu
\end{equation}
where $\mathbf{Cov}(G_l,G_p)$ is given by (\ref{GaussCov}).

\end{thm}
In the course of the proof of Theorem \ref{thm-GaussCLT}, it has been necessary to understand the following bilinear form.
\begin{defn}
\label{def:bilin}
Let $M$ be a Wigner matrix satisfying (\ref{condentries}),  and let $P^{(l)}$, $P^{(l,r)}$  be the projection matrices defined in (\ref{def-proj}) and (\ref{def-proj1}). For functions $f,g \in \mathcal{H}_{s}, s > \frac{3}{2}$, define
\begin{eqnarray}
\langle f, g  \rangle_{lr} &:=& \lim_{n\to \infty} \frac{1}{n}\sum_{j,k \in B_l \cap B_r}\E\left[ f(M^{(l)})_{jk}\cdot g(M^{(r)})_{kj}\right]\nonumber\\
&=& \lim_{n \rightarrow \infty} \frac{1}{n}\E \left[ \Tr\left\{ P^{(l)} f(M^{(l)})\cdot P^{(l,r)}\cdot g(M^{(r)})P^{(r)} \right\}\right].\nonumber\\
\label{form:bilin}
\end{eqnarray}
\end{defn}
\begin{rem}
The bilinear form $\langle \cdot, \cdot  \rangle_{lr}$ is well defined on $ \mathcal{H}_{s}\times \mathcal{H}_{s}  $ as a consequence of 
Proposition \ref{lem:exten}. The bilinear form is also well defined for polynomial $f$ and $g$, see Subsection \ref{subsec:bilin} and also 
Lemma \ref{lem:diag} below.
\end{rem}
The following diagonalization lemma is an important technical tool for the proof of Theorem \ref{thm-GaussCLT}.
\begin{lem}
\label{lem:diag}
The two families $\{U^{\gamma_{l}}_{k}\}^{\infty}_{k=0}$ and $\{U^{\gamma_{r}}_{q}\}^{\infty}_{q=0}$ of rescaled Chebyshev polynomials of the second kind diagonalize the bilinear form (\ref{form:bilin}). More precisely,
\begin{equation}
\label{cool}
\frac{1}{\sqrt{\gamma_{l}\gamma_{r}}} \ \langle U^{\gamma_{l}}_{k}, U^{\gamma_{r}}_{q} \rangle_{lr} \ =  \ \delta_{kq} \ \left(\frac{\gamma_{lr}}{\sqrt{\gamma_{l}\gamma_{r}}}\right)^{k+1}.
\end{equation}

Let $f,g \in \mathcal{H}_s,$ for some $ s > \frac{3}{2}$. A consequence of (\ref{cool}) is that

\begin{equation}
\langle f, g  \rangle_{lr} = \frac{1}{4\pi^{2}\gamma_{l}\gamma_{r}} \int^{2\sqrt{\gamma_{l}}}_{-2\sqrt{\gamma_{l}}} \int^{2\sqrt{\gamma_{r}}}_{-2\sqrt{\gamma_{r}}} f(x)g(y)\left[ \sum^{\infty}_{k=0}  U^{\gamma_{l}}_{k}(x)U^{\gamma_{r}}_{k}(y)\frac{\gamma^{k+1}_{lr}}{\gamma^{k/2}_{l}\gamma^{k/2}_{r}} \right] \sqrt{4\gamma_{l}-x^{2}}\sqrt{4\gamma_{r}-y^{2}}dydx.
\label{wet2}
\end{equation}

\end{lem}
In Subsection \ref{subsec:bilin} it will also be proved that, with $f,g$ given as above,  almost surely 
\begin{eqnarray}
\lefteqn{\lim_{n \rightarrow \infty} \frac{1}{n} \Tr\left\{ P^{(l)} f(M^{(l)})\cdot P^{(l,r)}\cdot g(M^{(r)})P^{(r)} \right\}}\nonumber\\
&=& \frac{1}{4\pi^{2}\gamma_{l}\gamma_{r}} \int^{2\sqrt{\gamma_{l}}}_{-2\sqrt{\gamma_{l}}} \int^{2\sqrt{\gamma_{r}}}_{-2\sqrt{\gamma_{r}}} f(x)g(y)\left[ \sum^{\infty}_{k=0}  U^{\gamma_{l}}_{k}(x)U^{\gamma_{r}}_{k}(y)\frac{\gamma^{k+1}_{lr}}{\gamma^{k/2}_{l}\gamma^{k/2}_{r}} \right] \sqrt{4\gamma_{l}-x^{2}}\sqrt{4\gamma_{r}-y^{2}}dydx.\nonumber\\
\label{safe}
\end{eqnarray}

\begin{rem}
\par
Recall that the rescaled Chebyshev polynomials of the second kind are orthonormal with respect to the Wigner semicircle law, i.e.
\begin{equation}
\label{orthogcheby2}
\frac{1}{2\pi \gamma}\int^{2\sqrt{\gamma}}_{-2\sqrt{\gamma}} U^{\gamma}_{k}(x)U^{\gamma}_{q}(x) \sqrt{4\gamma-x^{2}}dx = \delta_{kq}.
\end{equation}
Also,
\begin{equation}
U^{\gamma}_{k}(2\sqrt{\gamma}\cos(\theta)) = \frac{\sin((k+1)\theta)}{\sin(\theta)}.
\end{equation}
\end{rem}
The proof of Theorem \ref{thm-GaussCLT} appears in Section 3 and the proof of Theorem \ref{thm-NG} appears in Section 4.

\begin{rem}
\par
Theorems \ref{thm-GaussCLT} and  \ref{thm-NG} prove convergence of finite-dimensional distributions. This paper does not address the functional convergence
which would require a tightness result.
\end{rem}


\section{Proof of Theorem 2.1}

\subsection{Stein-Tikhomirov Method}
We follow the approach used by A. Lytova and L. Pastur in \cite{cite_key3} for the full Wigner matrix case, see also \cite{lp}, \cite{shch1}, 
\cite{shch2}. Essentially, 
it is a modification of the Stein-Tikhomirov method (see e.g. \cite{Cek}). 
This approach was also used to prove the CLT for linear eigenvalue statistics of band 
random matrices in \cite{LS}, which is connected to our work through the Chu-Vandermonde identity (see Subsection \ref{subsec:bilin}). 
While several steps of our proof are similar to the ones
 in \cite{cite_key3}, 
the fact that we are dealing with submatrices introduces new technical difficulties.

We will prove Theorem \ref{thm-GaussCLT} in the present section and extend the technique to non-Gaussian Wigner matrices later.  
The following inequalities will be used often. As a consequence of the Poincar\'{e} inequality, 
one can bound from above the variance of $\Tr \varphi(M)$
for a differentiable test functions $\varphi$ as

\begin{eqnarray}
\label{ineq-Poin}
\label{Poin-1}
\mathbf{Var}\{\Tr \varphi(M)\} & \leq &\frac{4(\sigma^{2}+1)}{n} \E\left[\Tr\{\varphi'(M) (\varphi'(M))^{*}\}\right]\\
&\leq& 4(\sigma^{2}+1)\left( \sup_{x \in \R} |\varphi'(x)|\right)^2.
\label{Poin-2}
\end{eqnarray}
We refer the reader to \cite{cite_key3} for the details. The next inequality is due to M. Shcherbina, see \cite{cite_key4}. 
Let $s > 3/2$ and $ \varphi \in \mathcal{H}_{s}$. Then there is a constant $C_s > 0$, so that
\begin{equation}
\label{ineq-MS}
\mathbf{Var}\{  \Tr \varphi(M) \} \leq C_{s}||\varphi ||^{2}_{s}.
\end{equation}
Let $\epsilon > 0$ and set $ s = \frac{5}{2}+\epsilon$. Recall that the regularity assumption on the test functions is that $||\varphi_{l} ||_{5/2+\epsilon} < \infty $, for $ 1 \leq l \leq d$. There exists a $ C_\epsilon > 0$ so that
\begin{equation}
\label{eqn:Varcontrol}
\mathbf{Var}\{  \mathcal{N}^{(l)}[\varphi_{l}] \} = 
\mathbf{Var}\{ \Tr \varphi_l(M(B_l)) \} \leq C_{\epsilon}||\varphi_{l} ||^{2}_{5/2+\epsilon}.
\end{equation}
The inequality holds because of (\ref{ineq-MS}), since $M(B_l)$ is an ordinary $|B_l|\times |B_l|$ Gaussian Wigner matrix. 
We note that the bound is $n$-independent.

It is sufficient  to prove the CLT for all linear combinations of the components of the random vector (\ref{randvec}). 
Consider a linear combination $  \xi := \sum^{d}_{l = 1} \alpha_{l} \mathcal{N}^{(l)\circ}[\varphi_{l}] $, and denote the characteristic function by
 
\begin{equation} 
  Z_{n}(x) = \E[ e^{ix\xi}].
\end{equation}
  
It is a basic fact that the characteristic function 
of the Gaussian distribution with variance $V$ is given by
 
\begin{equation}
Z(x) := e^{-x^2V/2}.
\label{eqn:char}
\end{equation}
As a consequence of the Levy Continuity theorem, to prove theorem  \ref{thm-GaussCLT} it will be sufficient to demonstrate that
for each $ x \in \mathbb{R}$,
\begin{equation}
\label{eqn:fin}
\lim_{n \rightarrow \infty} Z_{n}(x) = Z(x),
\end{equation}
where $Z(x)$ is given as above with
\begin{equation}
\label{eqn:vard}
V := \ \lim_{n \to \infty} \ \left[ \sum_{l = 1}^{d} \alpha^{2}_{l}  \mathbf{Var}\left(\mathcal{N}_n^{(l)\circ}[\varphi_{l}]\right)  + \ 
2  \sum_{1 \leq l < r \leq d} \alpha_{l}\alpha_{r}  \mathbf{Cov}\left(\mathcal{N}_n^{(l)\circ}[\varphi_{l}],\ \mathcal{N}_n^{(r)\circ}[\varphi_{r}]\right)\right].
\end{equation}
So $V$ is the limiting variance of $\xi$. 
It will be demonstrated that $ Z_{n}(x)$ converges uniformly to the solution of the following equation

\begin{equation}
Z(x) \ = \ 1 - V \int_{0}^{x} yZ(y)dy.
\label{eqn:Z}
\end{equation}

Note that (\ref{eqn:char}) is the unique solution of  (\ref{eqn:Z}) within the class of bounded and continuous functions. Therefore, to prove the theorem, it is sufficient to demonstrate that the pointwise limit of $Z_n(x)$ is a continuous and bounded function which satisfies equation (\ref{eqn:Z}), with $V$ given by (\ref{eqn:vard}).

Observe that
\begin{equation}
Z'_{n}(x) = i\E[\xi e^{ix\xi}]=i\sum^{d}_{l=1}\alpha_{l} \E\{\mathcal{N}_n^{(l)\circ}[\varphi_{l}] e^{ix\xi}\}.
\label{joyjoy}
\end{equation}

Now it follows by the Cauchy-Schwartz inequality and (\ref{eqn:Varcontrol}) that
\begin{eqnarray}
\label{heyo}
\left| Z'_{n}(x) \right|  &\leq& \sum^{d}_{l=1}|\alpha_{l}| \sqrt{\mathbf{Var}\{ \mathcal{N}^{(l)}[\varphi_{l}] \} }\leq Const\sum^{d}_{l=1}|\alpha_{l}| \ ||\varphi_{l}||_{5/2+\epsilon}.
\end{eqnarray}

Since
 $ Z_{n}(0) = 1$, we have by the fundamental theorem of calculus that

\begin{equation}
\label{eqn:FTCapp}
Z_{n}(x) = 1 + \int^{x}_{0} Z'_{n}(y)dy.
\end{equation}

Then to prove the CLT it is sufficient to show that any uniformly converging subsequences $\{Z_{n_m}\}$ and $\{Z'_{n_m}\}$, satisfy
\begin{equation}
\label{eqn:Zsub}
\lim_{n_m  \to \infty} Z_{n_m}(x) = Z(x),
\end{equation}
and
\begin{equation}
\label{eqn:Zprime}
\lim_{n_m \to \infty} Z'_{n_m}(x)  = -xVZ(x).
\end{equation}

A pre-compactness argument based on the Arzela-Ascoli theorem will be developed below, which ensures that the subsequences converge uniformly, implying that the limit is a continuous function. The estimate $|Z_n(x)| \leq 1$, for all $n$, shows that the sequence is uniformly bounded. Generally we will abuse the subsequence notation by writing $\{n \}$ for a uniformly converging subsequence. Since (\ref{heyo}) combined with $||\varphi_{l}||_{5/2+\epsilon} < \infty$ justify an application of  the dominated convergence theorem in (\ref{eqn:FTCapp}), it follows from (\ref{eqn:Zsub}) and (\ref{eqn:Zprime}) that the limit of $Z_n(x)$ satisfies equation (\ref{eqn:Z}). Therefore the pointwise limit (\ref{eqn:fin}) holds. We turn our attention to the pre-compactness argument, and will argue later that (\ref{eqn:Zsub}) and (\ref{eqn:Zprime}) hold. Similar notation is used as in \cite{cite_key3}. Denote by
\begin{eqnarray}
&&D_{jk}:=\partial/\partial M_{jk}\label{D_jk};\\
&&U^{(l)}(t):=e^{itM^{(l)}}, \ U^{(l)}_{jk}(t):=(U^{(l)}(t))_{jk};\label{ujk}\\
&&u^{(l)}_{n}(t):=\Tr \{P^{(l)}U^{(l)}(t)P^{(l)}\}, \  u^{(l)\circ}_{n}(t):=u^{(l)}_{n}(t)-\E\{u^{(l)}_{n}(t)\}.\label{moskva}
\end{eqnarray}
For the benefit of the reader, what is needed is recorded below. Recall that $U^{(l)}(t)$ is a unitary matrix, and writing $ \beta_{jk} := (1+\delta_{jk})^{-1}$, we have
\begin{equation}
|U^{(l)}_{jk}|\leq 1, \ \sum_{k=1}^n|U^{(l)}_{jk}|^2=1, \  \|U^{(l)}\|=1.
\label{Uprop}
\end{equation}
Moreover,
\begin{equation}
D_{jk}U^{(l)}_{ab}(t)=i\beta_{jk} \mathbf{1}_{\{j,k \in B_{l}\}}\left(U^{(l)}_{aj}*U^{(l)}_{bk}(t)+U^{(l)}_{ak}*U^{(l)}_{bj}(t)\right),
\label{dif}
\end{equation}
where 
\begin{equation}
\label{svertka}
f*g(t):=\int_0^t f(y)\*g(t-y)\*dy.
\end{equation}
Applying the Fourier inversion formula
\begin{equation}
\label{eqn:invfour}
\varphi_{l}(\lambda)=\int_{-\infty}^{\infty} e^{it\lambda}\widehat{\varphi}_{l}(t)dt,
\end{equation}
it follows that
\begin{equation}
\mathcal{N}^{(l)\circ}[\varphi_{l}] =  \int^{\infty}_{-\infty} \widehat{\varphi}_{l}(t) u^{(l)\circ}_{n}(t) dt.
\end{equation}

Now define

\begin{equation}
\label{def_en}
e_{n}(x) := e^{ix\xi}.
\end{equation}

Using the Fourier representation of the linear eigenvalue statistics in (\ref{joyjoy}), it follows that
\begin{equation}
\label{eqn:deriv}
 Z'_{n}(x) = i \ \sum^{d}_{l = 1} \alpha_{l} 
  \int^{\infty}_{-\infty} \widehat{\varphi}_{l}(t) Y^{(l)}_{n}(x,t) dt,
\end{equation}
where
\begin{equation}
Y^{(l)}_{n}(x,t) := \E \left[u^{(l)\circ}_{n}(t) e_n(x) \right].
\end{equation}

The limit of $Y^{(l)}_n(x,t)$ is determined later in the proof. Since 
\begin{equation}
  \overline{Y^{(l)}_{n}(x,t)} = Y^{(l)}_{n}(-x,-t),
\end{equation}
we need only consider $t \geq 0$. It will now be demonstrated that each sequence $\{Y^{(l)}_{n}\}$ is bounded and equicontinuous on compact subsets of $\{x \in \mathbb{R}, t\geq0\}$, and that every uniformly converging subsequence has the same limit $Y^{(l)}$, implying (\ref{eqn:Zsub}) and (\ref{eqn:Zprime}). See proposition \ref{superman}.

Let $\varphi(x) = e^{itx}$, and note that $\sup_{x \in \R} |\varphi'(x)| = |t|$. Applying the inequality (\ref{Poin-2}) to the linear eigenvalue statistic $ \mathcal{N}^{(l)}[\varphi]$, we obtain

\begin{equation}
\label{eqn:u}
\mathbf{Var}\{u^{(l)}_{n}(t)\} = \mathbf{Var}\{ \mathcal{N}^{(l)}[\varphi]\} \leq 4(\sigma^2 + 1)t^2.
\end{equation}

Now set $\varphi(x) = ixe^{itx}$, and notice that 
\[\frac{d}{dt}u^{(l)}_{n}(t) =  i \Tr \{ M^{(l)}e^{itM^{(l)}}  \}.  \]
Using the inequality (\ref{Poin-1}) and the fact that $n^{-1}\E\Tr (M^{(l)})^2 \leq \sigma^2+1$, it follows that
\begin{eqnarray}
\label{eqn:uprime}
\mathbf{Var}\{\frac{d}{dt}u^{(l)}_{n}(t)\} &\leq& \frac{4(\sigma^{2}+1)}{n} \E\left[\Tr\{\varphi'(M^{(l)}) (\varphi'(M^{(l)}))^{*}\}\right]\nonumber\\
&\leq& \frac{4(\sigma^{2}+1)}{n} \E\left[\Tr\{ 1+t^2(M^{(l)})^2 \}\right]\nonumber\\
&\leq& 4(\sigma^2+1)[1+(\sigma^2+1)t^2].
\end{eqnarray}
Using the Cauchy-Schwartz inequality, the bound $|e_{n}(x)| \leq 1,$ (\ref{eqn:u}) and (\ref{eqn:uprime}), we obtain
\begin{eqnarray}
\label{fish}
\left|Y^{(l)}_{n}(x,t)\right| \leq \mathbf{Var}^{1/2}\{ u^{(l)}_{n}(t) \} \leq  2(\sigma^2+1)^{1/2}|t|,
\end{eqnarray}
and also
\begin{equation}
\label{duck}
\left|\frac{\partial}{\partial t} Y^{(l)}_{n}(x,t)\right| \leq \mathbf{Var}^{1/2}\{\frac{d}{dt}u^{(l)}_{n}(t)  \} \leq  2\sqrt{(\sigma^2+1 + (\sigma^2+1)^2t^2)}.
\end{equation}

Observe that 
\[\frac{d}{dx}e_n(x) = ie_n(x)\sum^d_{r=1} \alpha_r \  \mathcal{N}^{(r)\circ}[\varphi_r].\]
Using the above derivative with the Cauchy-Schwartz inequality, (\ref{eqn:Varcontrol}) and (\ref{eqn:u}), we have that
\begin{eqnarray}
\label{goose}
\left|\frac{\partial}{\partial x} Y^{(l)}_{n}(x,t)\right| &=& \left|i\sum^d_{r=1}\alpha_r \ \E[u^{(l)\circ}_n(t)  \mathcal{N}_n^{(r)\circ}[\varphi_r] e_n(x)]  \right|\nonumber\\
&\leq&  \mathbf{Var}^{1/2}\{ u^{(l)}_{n}(t) \}  \sum^d_{r=1} |\alpha_r| \ \mathbf{Var}^{1/2}\{ \mathcal{N}^{(r)}[\varphi_r]  \}\nonumber\\
&\leq& Const \cdot |t| \sum^d_{r=1} |\alpha_r| \ ||  \varphi_r ||_{5/2+\epsilon}.
\end{eqnarray}

It follows from (\ref{fish}), the mean value theorem combined with (\ref{duck}) and (\ref{goose}), and $ ||  \varphi_r ||_{5/2+\epsilon}<\infty$, that each sequence $Y^{(l)}_n(x,t)$ is bounded and equicontinuous on compact subsets of $\R^2 $. The following proposition justifies this restriction.
\begin{prop}
\label{superman}
In order to prove the functions $Y^{(l)}_n(x,t)$ converge uniformly to appropriate limits so that (\ref{eqn:deriv}) implies (\ref{eqn:Zprime}), it is sufficient to prove the convergence of $Y^{(l)}_n(x,t)$ on  arbitrary compact subsets of $\{ x \in \R, t \geq 0\}$.
\end{prop}

\begin{proof}
Let $\delta > 0$. Recall that the regularity assumption on the test functions $\varphi_l$ are
\[ \int_\R (1+|h|)^{5+\epsilon}|\widehat{\varphi}_l(h)|^2 dh < \infty,\]
i.e. that $\varphi_l \in \mathcal{H}_s$, with $s = 5/2 + \epsilon$. Using the Cauchy-Schwartz inequality, it follows that 
\begin{eqnarray}
\int_\R (1+|h|)|\widehat{\varphi}_l(h)|dh \leq \sqrt{ \int_\R \frac{dh}{(1+|h|)^{3+\epsilon}}} \cdot \sqrt{\int_\R(1+|h|)^{5+\epsilon}|\widehat{\varphi}_l(h)|^2dh},\nonumber\\
\end{eqnarray}
which implies that
\begin{equation}
\label{Lib_Int}
\int_{\R} |h|\cdot |\widehat{\varphi}_l(h)|dh < \infty.
\end{equation}

A consequence of the finiteness of the integral in (\ref{Lib_Int}), for each $1 \leq l \leq d$, is that there exists a $T > 0$ so that
\begin{equation}
\label{arkansas}
2(\sigma^2+1)^{1/2}\sum^d_{l=1} |\alpha_l| \int_{|t|\geq T} |t| \cdot |\widehat{\varphi}_l(t)|dt < \delta.
\end{equation}

Using (\ref{eqn:deriv}), we can write

\begin{equation}
\label{california}
 Z'_{n}(x) = i \ \sum^{d}_{l = 1} \alpha_{l} 
  \int^T_{-T} \widehat{\varphi}_{l}(t) Y^{(l)}_{n}(x,t) dt + i \ \sum^{d}_{l = 1} \alpha_{l} 
  \int_{|t| \geq T} \widehat{\varphi}_{l}(t) Y^{(l)}_{n}(x,t) dt.
\end{equation}

Then (\ref{california}), (\ref{fish}), (\ref{arkansas}) imply that

\begin{eqnarray}
\label{conneticut}
\left| Z'_{n}(x) -   i \ \sum^{d}_{l = 1} \alpha_{l} 
 \int^T_{-T} \widehat{\varphi}_{l}(t) Y^{(l)}_{n}(x,t) dt   \right| &\leq& \sum^{d}_{l = 1} |\alpha_{l}| \int_{|t| \geq T} |\widehat{\varphi}_{l}(t)|\cdot |Y^{(l)}_{n}(x,t)| dt\nonumber\\
&\leq& 2(\sigma^2+1)^{1/2}\sum^d_{l=1} |\alpha_l| \int_{|t|\geq T} |t| \cdot |\widehat{\varphi}_l(t)|dt\nonumber\\
&<& \delta.
\end{eqnarray}

Notice that the estimate (\ref{conneticut}) is $n$-independent, so that in particular the estimate holds in the limit $n \to \infty$. Since $\delta$ was arbitrary, this completes the proof of the proposition.
\end{proof}

This completes the pre-compactness argument, which allows us to pass to the limit in (\ref{eqn:deriv}) and in (\ref{eqn:FTCapp}), and conclude that $Z_n(x)$ converges pointwise to the unique solution of equation (\ref{eqn:Z}) belonging to $C_b(\R)$, implying (\ref{eqn:fin}), and hence the conclusion of the theorem. Now we show the limiting behavior of the sequences $Y^{(l)}_n(x,t)$ imply (\ref{eqn:Zsub}) and (\ref{eqn:Zprime}). Consider the identity
\[e^{itM^{(l)}} = I + i \int^{t}_{0} M^{(l)} e^{ihM^{(l)}}dh.\]
Apply this identity, noting that $ M^{(l)}_{jk} = 0, \ \text{if} \ j,k \notin B_{l}$, to obtain that
\begin{eqnarray}
\label{eqn_un}
u^{(l)\circ}_n(t) &=& \Tr\{ P^{(l)}U^{(l)}(t)P^{(l)}\} - \E [\Tr\{ P^{(l)}U^{(l)}(t)P^{(l)}\} ]\nonumber\\
&=& i\int^t_0 \sum^n_{j,k=1}  \left[M^{(l)}_{jk}U^{(l)}_{jk}(t_1) - \E[M^{(l)}_{jk}U^{(l)}_{jk}(t_1)]\right].\nonumber\\
\end{eqnarray}
Recalling that $ Y^{(l)}_{n}(x,t) = \E \left[u^{(l)\circ}_{n}(t) e_n(x) \right]$, and applying the decoupling formula for Gaussian random variables, it follows from (\ref{eqn_un}) that
\begin{eqnarray}
Y^{(l)}_{n}(x,t)&=&i \int^{t}_{0} \sum^{n}_{j,k = 1} \E [ M^{(l)}_{jk}U^{(l)}_{jk}(t_{1})e^{\circ}_{n}(x)] dt_{1}\nonumber\\
&=&  \frac{2i}{n} \int^{t}_{0} \sum_{1 \leq j < k \leq n} \mathbf{1}_{\{j,k \in B_{l}\}} \  \E
\left[  D_{jk}U^{(l)}_{jk}(t_{1}) e^{\circ}_{n}(x)\right]dt_{1}.\nonumber\\
&& + \frac{i\sigma^{2}}{n} \int^{t}_{0} \sum^{n}_{j = 1} \mathbf{1}_{\{j \in B_{l}\}}\E\left[ D_{jj}U^{(l)}_{jj}(t_{1}) e^{\circ}_{n}(x) \right]dt_{1}.\nonumber\\
\label{newYn}
\end{eqnarray}

It will be useful to rewrite (\ref{newYn}) as 
\begin{eqnarray}
Y^{(l)}_{n}(x,t)&=& \underbrace{\frac{i}{n} \int^{t}_{0} \sum^{n}_{j, k = 1} \mathbf{1}_{\{j,k \in B_{l}\}}(1+\delta_{jk}) \E \left[  D_{jk}U^{(l)}_{jk}(t_{1}) e^{\circ}_{n}(x)\right]dt_{1}}_{=: T_{1}}\nonumber\\
&& +  \underbrace{\frac{i( \sigma^{2} - 2)}{n} \int^{t}_{0} \sum^{n}_{j = 1} \mathbf{1}_{\{j \in B_{l}\}}\E\left[ D_{jj}U^{(l)}_{jj}(t_{1}) e^{\circ}_{n}(x) \right]dt_{1}}_{=: T_{2}}.\nonumber\\
\end{eqnarray}

The reason for the rewrite is that it splits the functions $Y^{(l)}_n(x,t)$ into a part that depends on the distribution of the diagonal entries and a part that corresponds to the same term as for the Gaussian Orthogonal Ensemble, for which $\sigma^2 =2$. Recalling that $e_n(x) $ is given by (\ref{def_en}), again writing   $ \beta_{jk} = (1+\delta_{jk})^{-1}$ and using the identity
\[D_{jk} \Tr f(M)  = 2\beta_{jk} f'(M)_{jk},\]
it follows by a direct calculation that
\begin{equation}
\label{alaska}
D_{jk} e_{n}(x) = 2i\beta_{jk}xe_{n}(x)  \sum^{d}_{ r = 1} \alpha_{r} 
\left( P^{(r)}\varphi_{r}'(M^{(r)})P^{(r)}\right)_{jk}.
\end{equation}

Then for $ 1 \leq l \leq d$, using (\ref{alaska}) and (\ref{dif}), it follows that
\begin{eqnarray}
\label{alabama}
T_{1} &=&  \frac{-1}{n} \int^{t}_{0} \int^{t_{1}}_{0} \E \left[ \sum^{n}_{j,k = 1}  \mathbf{1}_{\{j,k \in B_{l}\}}
 U^{(l)}_{jj}(t_{2}) U^{(l)}_{kk}(t_{1} - t_{2})e^{\circ}_{n}(x) \right] dt_{2}dt_{1}\nonumber\\
&&-  \frac{1}{n} \int^{t}_{0} \int^{t_{1}}_{0} \E \left[ \sum^{n}_{j,k = 1}  \mathbf{1}_{\{j,k \in B_{l}\}}
 U^{(l)}_{jk}(t_{2}) U^{(l)}_{jk}(t_{1} - t_{2}) e^{\circ}_{n}(x)\right] dt_{2}dt_{1}\nonumber\\
&&- \frac{2x}{n} \int^{t}_{0} \E \left[ \sum^{n}_{j,k = 1} \mathbf{1}_{\{j,k \in B_{l}\}} U^{(l)}_{jk}(t_{1}) 
e_{n}(x)\sum^{d}_{ r = 1} \alpha_{r} \left(P^{(r)}\varphi_{r}'(M^{(r)})P^{(r)}\right)_{jk}\right] dt_{1},\nonumber\\
\end{eqnarray}
and also that
\begin{eqnarray}
\label{f_T2}
T_{2} &=&  \underbrace{\frac{-(\sigma^{2}-2)}{n} \int^{t}_{0} \int^{t_{1}}_{0} \E \left[ \sum^{n}_{j = 1} \mathbf{1}_{\{j \in B_{l}\}} U^{(l)}_{jj}(t_{2}) U^{(l)}_{jj}(t_{1} - t_{2}) e^{\circ}_{n}(x) \right] dt_{2}dt_{1}}_{=:T_{21}}\nonumber\\
&& \underbrace{- \frac{(\sigma^{2}-2)x}{n} \int^{t}_{0} \E \left[ \sum^{n}_{j = 1}\mathbf{1}_{\{j \in B_{l}\}} U^{(l)}_{jj}(t_{1})e_{n}(x)\sum^{d}_{ r = 1} \alpha_{r} \left(P^{(r)} \varphi_{r}'(M^{(r)})P^{(r)}\right)_{jj}\right] dt_{1}}_{=:T_{22}}.\nonumber\\
\end{eqnarray}
Using the semigroup property
\[U^{(l)}(t)U^{(l)}(h) = U^{(l)}(t+h),\]
it follows form (\ref{alabama}) that $T_1$ can be written
\begin{eqnarray}
\label{f_T1}
T_{1}&=& \underbrace{-\frac{1}{n} \int^{t}_{0} \int^{t_{1}}_{0} \E
 \left[ u^{(l)}_{n}(t_{1} - t_{2}) u^{(l)}_{n}( t_{2})e^{\circ}_{n}(x)\right] dt_{2}dt_{1}}_{=:T_{11}}\nonumber\\
&&\underbrace{-\frac{1}{n} \int^{t}_{0} t_{1} \E\left[  u^{(l)}_{n}(t_{1}) e^{\circ}_{n}(x)\right]dt_{1}}_{=:T_{12}}\nonumber\\
&&\underbrace{-\frac{2x}{n} \sum^{d}_{r = 1} \alpha_{r} \int^{t}_{0} \E
\left[\Tr\{P^{(l)}U^{(l)}(t_{1})P^{(l,r)}\varphi'_{r}(M^{(r)})P^{(r)}\} e_{n}(x) \right] dt_{1}}_{=:T_{13}}.
\end{eqnarray}
Define
\begin{equation}
\label{eqn_vnbar}
\bar{v}^{(l)}_{n}(t) := \frac{1}{n} \E[ u^{(l)}_{n}(t)]. 
\end{equation}
The following proposition presents the functions $Y^{(l)}_n(x,t)$ in a form that is amenable to asymptotic analysis.
\begin{prop}
\label{eqn_rewrite}
The equation  $ Y^{(l)}_{n}(x,t) = T_{1} + T_{2}$, can be written as
\begin{eqnarray}
Y^{(l)}_{n}(x,t) + 2\int^{t}_{0} \int^{t_{1}}_{0} \bar{v}^{(l)}_{n}(t_{1}-t_{2}) Y^{(l)}_{n}(x,t_{2})dt_{2}dt_{1}= x Z_{n}(x)\left[A^{(l)}_{n}(t) + Q^{(l)}_{n}(t) \right]  + r^{(l)}_{n}(x,t),\nonumber\\
\label{eqn:Integral1}
\end{eqnarray}
where
\begin{equation}
\label{eqn:limA}
A^{(l)}_{n}(t) :=  -2 \sum^{d}_{r= 1} \alpha_{r} \int^{t}_{0} \frac{1}{n}\E\left[\Tr\{P^{(l)}U^{(l)}(t_{1})P^{(l,r)}\varphi'_{r}(M^{(r)})P^{(r)}\}  \right] dt_{1},
\end{equation}
\begin{equation}
\label{eqn:limQ}
Q^{(l)}_{n}(t) :=  \frac{-(\sigma^{2}-2)}{n} \sum^{d}_{r= 1} \alpha_{r} \int^{t}_{0} \sum^{n}_{j=1} \mathbf{1}_{\{j \in B_{l} \cap B_r\}} \E\left[U^{(l)}_{jj}(t_{1}) \varphi'_{r}(M^{(r)})_{jj} \}\right] dt_{1},
\end{equation}
and
\begin{eqnarray}
\label{fishy}
r^{(l)}_{n}(x,t) &=&\nonumber\\
&&  \frac{-1}{n} \int^{t}_{0} t_{1} Y^{(l)}_{n}(x,t_{1})dt_{1}\label{fishy_sub1}\\
&& - \frac{1}{n} \int^{t}_{0} \int^{t_{1}}_{0} \E\left[u^{(l)\circ}_{n}(t_{1} - t_{2}) u^{(l)\circ}_{n}( t_{2}) e^{\circ}_{n}(x)\right] dt_{2}dt_{1}\label{fishy_sub2}\\
&&- \frac{2x}{n}   \sum^{d}_{r = 1} \alpha_{r} \int^{t}_{0} \E
\left[\Tr\{P^{(l)}U^{(l)}(t_{1})P^{(l,r)}\varphi'_{r}(M^{(r)})P^{(r)}\} e^{\circ}_{n}(x) \right] dt_{1}\label{fishy_sub3}\\
 && - \frac{(\sigma^{2}-2)}{n} \int^{t}_{0} \int^{t_{1}}_{0} \E \left[ \sum^{n}_{j = 1} \mathbf{1}_{\{j \in B_{l} \}}  U^{(l)}_{jj}(t_{2}) U^{(l)}_{jj}(t_{1} - t_{2})e^{\circ}_{n}(x) \right] dt_{2}dt_{1}\label{fishy_sub4}\\
&&  - \frac{x(\sigma^{2}-2)}{n}\sum^{d}_{r = 1} \alpha_{r} \int^{t}_{0} \sum^{n}_{j = 1} \mathbf{1}_{\{j \in B_{l} \cap B_r\}} \E\left[ U^{(l)}_{jj}(t_{1}) \varphi'_{r}(M^{(r)})_{jj}e^{\circ}_{n}(x)   \right] dt_{1}.\label{fishy_sub5}\\
\end{eqnarray}
\end{prop}

\begin{proof}
Begin with the term $T_{11}$, defined in (\ref{f_T1}). Write
\begin{equation}
T_{11} = -\frac{1}{n} \int^{t}_{0} \int^{t_{1}}_{0} \E
 \left[\left(u^{(l)\circ}_{n}(t_{1} - t_{2}) + n\bar{v}_n(t_1-t_2)\right)\cdot \left(u^{(l)\circ}_{n}( t_{2}) +n\bar{v}_n(t_2)\right)e^{\circ}_{n}(x)\right] dt_{2}dt_{1},
\end{equation}
so that
\begin{eqnarray}
T_{11} &=&\nonumber\\
&& - \frac{1}{n} \int^{t}_{0} \int^{t_{1}}_{0} \E\left[u^{(l)\circ}_{n}(t_{1} - t_{2})u^{(l)\circ}_{n}( t_{2})e^{\circ}_{n}(x)\right] dt_{2}dt_{1}\nonumber\\
&& -\int^t_0 \int^{t_1}_0 \bar{v}_n(t_1-t_2)\E\left[u^{(l)\circ}_n(t_2)e^{\circ}_n(x)  \right]dt_2dt_1\nonumber\\
&& -\int^t_0 \int^{t_1}_0 \bar{v}_n(t_2)\E\left[u^{(l)\circ}_n(t_1-t_2)e^{\circ}_n(x)  \right]dt_2dt_1\nonumber\\
&& -n\int^t_0 \int^{t_1}_0 \bar{v}_n(t_1-t_2)\cdot \bar{v}_n(t_2)\underbrace{\E\left[e^{\circ}_n(x)  \right]}_{=0}dt_2dt_1.\nonumber\\
\end{eqnarray}
Noting that 
\[\E\left[u^{(l)\circ}_n(t_2)e^{\circ}_n(x)  \right] = Y^{(l)}_n(x,t_2), \ \E\left[ u^{(l)\circ}_n(t_1-t_2)e^{\circ}_n(x)  \right] = Y^{(l)}_n(x,t_1-t_2),\]
and also that
\[\int^t_0 \int^{t_1}_0 \bar{v}_n(t_2)Y^{(l)}_n(x, t_1-t_2)dt_2dt_1 = \int^t_0 \int^{t_1}_0 \bar{v}_n(t_1-t_2)Y^{(l)}_n(x, t_2)dt_2dt_1,\]
it follows that
\begin{eqnarray}
T_{11} &=&\nonumber\\
&& - \frac{1}{n} \int^{t}_{0} \int^{t_{1}}_{0} \E\left[u^{(l)\circ}_{n}(t_{1} - t_{2})u^{(l)\circ}_{n}( t_{2}) e^{\circ}_{n}(x)\right] dt_{2}dt_{1}\label{T11_sub1}\\
&& -2\int^t_0 \int^{t_1}_0 \bar{v}_n(t_1-t_2)Y^{(l)}_n(x,t_2)dt_2dt_1.\label{T11_sub2}\\
\end{eqnarray}

The term (\ref{T11_sub1}) goes into the remainder, which becomes (\ref{fishy_sub2}). Also, (\ref{T11_sub2}) is added to the left-hand side of (\ref{eqn:Integral1}). Now consider the term $T_{12}$, defined in (\ref{f_T1}). We have that
\begin{equation}
T_{12}  = -\frac{1}{n} \int^{t}_{0} t_{1} Y^{(l)}_n(x,t_1)dt_{1},
\end{equation}
which  becomes (\ref{fishy_sub1}) in the remainder. Consider the term $T_{13}$, also defined in (\ref{f_T1}). Writing 
\begin{equation}
T_{13} = -\frac{2x}{n} \sum^{d}_{r = 1} \alpha_{r} \int^{t}_{0} \E
\left[\Tr\{P^{(l)}U^{(l)}(t_{1})P^{(l,r)}\varphi'_{r}(M^{(r)})P^{(r)}\}\cdot \left(e^{\circ}_{n}(x)+Z_n(x)\right) \right] dt_{1},
\end{equation}
it follows, with $ A^{(l)}_n(t)$ given by (\ref{eqn:limA}), that
\begin{eqnarray}
T_{13} &=&\nonumber\\
&& - \frac{2x}{n}   \sum^{d}_{r = 1} \alpha_{r} \int^{t}_{0} \E
\left[\Tr\{P^{(l)}U^{(l)}(t_{1})P^{(l,r)}\varphi'_{r}(M^{(r)})P^{(r)}\}e^{\circ}_{n}(x) \right] dt_{1}\label{T13_sub1}\\
&& + \  xZ_n(x) A^{(l)}_n(t).\label{T13_sub2}\\
\end{eqnarray}
Then (\ref{T13_sub1}) becomes (\ref{fishy_sub3}) in the remainder, while (\ref{T13_sub2}) remains on the right-hand side of (\ref{eqn:Integral1}).
Now consider the term $T_{21}$, defined in (\ref{f_T2}).  This term becomes (\ref{fishy_sub4}) in the remainder. Finally, consider the term $T_{22}$, also defined in (\ref{f_T2}). Write 
\begin{equation}
T_{22} =- \frac{(\sigma^{2}-2)x}{n} \int^{t}_{0} \E \left[ \sum^{n}_{j = 1}\mathbf{1}_{\{j \in B_{l}\}} U^{(l)}_{jj}(t_{1})\cdot \left(e^{\circ}_{n}(x) + Z_n(x)\right) \sum^{d}_{ r = 1} \alpha_{r} \left(P^{(r)} \varphi_{r}'(M^{(r)})P^{(r)}\right)_{jj}\right] dt_{1},
\end{equation}
so that, with $Q^{(l)}_n(t)$ given by (\ref{eqn:limQ}) ,
\begin{eqnarray}
T_{22} &=&\nonumber\\
&&- \frac{(\sigma^{2}-2)x}{n} \int^{t}_{0} \E \left[ \sum^{n}_{j = 1}\mathbf{1}_{\{j \in B_{l}\}} U^{(l)}_{jj}(t_{1})e^{\circ}_{n}(x)\sum^{d}_{ r = 1} \alpha_{r} \left(P^{(r)} \varphi_{r}'(M^{(r)})P^{(r)}\right)_{jj}\right] dt_{1}\label{T22_sub1}\\
&&+ \ xZ_n(x)\cdot Q^{(l)}_{n}(t).\label{T22_sub2}\\
\end{eqnarray}
The term (\ref{T22_sub1}) becomes (\ref{fishy_sub5}) in the remainder. Also, the term (\ref{T22_sub2}) remains on the right-hand side of  (\ref{eqn:Integral1}). This completes the argument for proposition \ref{eqn_rewrite}.
\end{proof}
We now turn our attention to the remainder term, $r^{(l)}_n(x,t)$, of proposition \ref{eqn_rewrite}. The content of the following proposition is that the remainder is negligible in the limit.
\begin{prop}
\label{rem_goto0}
Each term of $r^{(l)}_n(x,t)$ converges to $0$ uniformly on compact subsets of $ \{ x \in \mathbb{R}, t \geq 0\}$, for $ 1 \leq l \leq d$. In other words, we have the uniform limit
\begin{equation}
\lim_{n \rightarrow \infty}r^{(l)}_{n}(x,t) = 0.
\end{equation}
\end{prop}

\begin{proof}
Begin with the term (\ref{fishy_sub1}). Applying the estimate (\ref{fish}), we obtain
\begin{eqnarray}
\label{Lib1}
\left|\frac{1}{n} \int^{t}_{0} t_{1} Y^{(l)}_{n}(x,t_{1})dt_{1}\right| & \leq & \frac{1}{n}t^2\left|Y^{(l)}_n(x,t)  \right|\nonumber\\
&\leq & \frac{2(\sigma^2+1)^{1/2}}{n}|t|^3\nonumber\\ 
&=& O\left(\frac{1}{n}\right).\nonumber\\
\end{eqnarray}
Now consider the term (\ref{fishy_sub2}). Using the bound $|e^{\circ}_n(x)| \leq 2$, the Cauchy-Schwartz inequality, and (\ref{eqn:u}) twice, it follows that
\begin{eqnarray}
\label{Lib2}
\left|\frac{1}{n} \int^{t}_{0} \int^{t_{1}}_{0} \E\left[ u^{(l)\circ}_{n}(t_{1} - t_{2}) u^{(l)\circ}_{n}( t_{2}) e^{\circ}_{n}(x)\right] dt_{2}dt_{1}\right| &\leq& \frac{2}{n}t^2 \mathbf{Var}^{1/2}\{u^{(l)}_n(t)\}\mathbf{Var}^{1/2}\{u^{(l)}_n(t)\}\nonumber\\
&\leq& \frac{8(\sigma^2+1)^{1/2}}{n}t^4\nonumber\\
&=& O\left( \frac{1}{n}  \right).\nonumber\\
\end{eqnarray}
Consider the term (\ref{fishy_sub3}) next.  Applying (\ref{wet2}) of lemma \ref{lem:diag} to the exponential function and $\varphi'_r$, and noting that $\varphi'_r \in \mathcal{H}_{\frac{3}{2}+\epsilon}$, it follows that
\begin{eqnarray}
\lefteqn{ \lim_{n \rightarrow \infty} \frac{1}{n}\E \left[ \Tr\left\{ P^{(l)} U^{(l)}(t_1)P^{(l,r)} \varphi'_r(M^{(r)})P^{(r)} \right\}\right]}\nonumber\\
&& =  \frac{1}{4\pi^{2}\gamma_{l}\gamma_{r}} \int^{2\sqrt{\gamma_{l}}}_{-2\sqrt{\gamma_{l}}} \int^{2\sqrt{\gamma_{r}}}_{-2\sqrt{\gamma_{r}}} e^{it_1x}\varphi'_r(y)\left[ \sum^{\infty}_{k=0}  U^{\gamma_{l}}_{k}(x)U^{\gamma_{r}}_{k}(y)\frac{\gamma^{k+1}_{lr}}{\gamma^{k/2}_{l}\gamma^{k/2}_{r}} \right] \sqrt{4\gamma_{l}-x^{2}}\sqrt{4\gamma_{r}-y^{2}}dydx.\nonumber\\
\label{hud}
\end{eqnarray}
While the exponential function does not belong to $ \mathcal{H}_{\frac{3}{2}+\epsilon}$, we can truncate the exponential function in a smooth fashion outside the support of the semicircle law, so that the truncated exponential function belongs to $\mathcal{H}_{\frac{3}{2}+\epsilon}$. We may replace the exponential function by its truncated version because the eigenvalues of the submatrices concentrate in the support of the semicircle law with overwhelming probability. Then
\begin{eqnarray}
\lefteqn{ \lim_{n \rightarrow \infty} \frac{1}{n}  \Tr\left\{ P^{(l)} U^{(l)}(t_1)P^{(l,r)} \varphi'_r(M^{(r)})P^{(r)} \right\}}\nonumber\\
&& =  \frac{1}{4\pi^{2}\gamma_{l}\gamma_{r}} \int^{2\sqrt{\gamma_{l}}}_{-2\sqrt{\gamma_{l}}} \int^{2\sqrt{\gamma_{r}}}_{-2\sqrt{\gamma_{r}}} e^{it_1x}\varphi'_r(y)\left[ \sum^{\infty}_{k=0}  U^{\gamma_{l}}_{k}(x)U^{\gamma_{r}}_{k}(y)\frac{\gamma^{k+1}_{lr}}{\gamma^{k/2}_{l}\gamma^{k/2}_{r}} \right] \sqrt{4\gamma_{l}-x^{2}}\sqrt{4\gamma_{r}-y^{2}}dydx.\nonumber\\
\label{qrr}
\end{eqnarray}

Here it is not so important to know the exact value of the limit, but we will use the fact that we have convergence in the mean and almost surely to the same limit. Note the convergence in (\ref{hud}) implies that the sequence of numbers
\[ \frac{1}{n}\E \left[ \Tr\left\{ P^{(l)} U^{(l)}(t_1)P^{(l,r)} \varphi'_r(M^{(r)})P^{(r)} \right\}\right],\]
is bounded. Also the convergence in (\ref{qrr}) implies that the random variables 
\[ \frac{1}{n}\Tr\left\{ P^{(l)} U^{(l)}(t_1)P^{(l,r)} \varphi'_r(M^{(r)})P^{(r)} \right\},\]
are  bounded with probability $1$. Using (\ref{hud}) and (\ref{qrr}) with the dominated convergence theorem, it now follows that
\begin{equation}
\lim_{n \to \infty}  \E \left| \frac{1}{n} \Tr P^{(l)}U^{(l)}(t_{1})P^{(l,r)}\varphi'_{r}(M^{(r)})P^{(r)} - \frac{1}{n}\E\left\{  \Tr P^{(l)}U^{(l)}(t_{1})P^{(l,r)}\varphi'_{r}(M^{(r)})P^{(r)}  \right\} \right|  = 0.
\label{wet}
\end{equation}
Combining the bound $|e_n(x)| \leq 1$ with (\ref{wet}), it follows that
\begin{eqnarray}
\label{fish2}
&& \left|\frac{1}{n} \E\left[\Tr\{P^{(l)}U^{(l)}(t_{1})P^{(l,r)}\varphi'_{r}(M^{(r)})P^{(r)}\}e^{\circ}_{n}(x) \right]\right|  \nonumber\\
&=& \left| \E \left[ \left(\frac{1}{n} \Tr P^{(l)}U^{(l)}(t_{1})P^{(l,r)}\varphi'_{r}(M^{(r)})P^{(r)} - \frac{1}{n}\E\left\{  \Tr P^{(l)}U^{(l)}(t_{1})P^{(l,r)}\varphi'_{r}(M^{(r)})P^{(r)}  \right\} \right)e_{n}(x) \right]\right|\nonumber\\
&\leq& \E \left| \frac{1}{n} \Tr P^{(l)}U^{(l)}(t_{1})P^{(l,r)}\varphi'_{r}(M^{(r)})P^{(r)} - \frac{1}{n}\E\left\{  \Tr P^{(l)}U^{(l)}(t_{1})P^{(l,r)}\varphi'_{r}(M^{(r)})P^{(r)}  \right\} \right| \to 0.\nonumber\\
\end{eqnarray}

Then, using (\ref{fish2}) in the remainder term (\ref{fishy_sub3}), it follows that
\begin{eqnarray}
\label{Lib3}
&& \left|- \frac{2x}{n}   \sum^{d}_{r = 1} \alpha_{r} \int^{t}_{0} \E
\left[\Tr\{P^{(l)}U^{(l)}(t_{1})P^{(l,r)}\varphi'_{r}(M^{(r)})P^{(r)}\} e^{\circ}_{n}(x) \right] dt_{1}\right|\to 0 \ \text{as} \ n \to \infty.\nonumber\\
\end{eqnarray}

Consider (\ref{fishy_sub4}), which is the next term in the remainder. Observe that, again using the Cauchy-Schwartz inequality and the fact that $|e_n(x)|\leq 1$,
\begin{eqnarray}
\label{libby3}
 \lefteqn{\E \left[ \frac{1}{n} \sum^{n}_{j = 1} \mathbf{1}_{\{j \in B_{l} \}}  U^{(l)}_{jj}(t_{2}) U^{(l)}_{jj}(t_{1} - t_{2}) e^{\circ}_{n}(x) \right]}\nonumber\\
 &=&  \E \left[ \frac{1}{n}\sum^{n}_{j \in B_{l}}  U^{(l)}_{jj}(t_{2}) U^{(l)}_{jj}(t_{1} - t_{2}) e^{\circ}_{n}(x) \right]\nonumber\\ 
 &\leq& \E\left| \frac{1}{n}\sum_{j \in B_l} U^{(l)}_{jj}(t_2)U^{(l)}_{jj}(t_1-t_2) -\frac{1}{n}\E\left\{ \sum_{j \in B_l} U^{(l)}_{jj}(t_2)U^{(l)}_{jj}(t_1-t_2)\right\}  \right|\nonumber\\
 &\leq&  \mathbf{Var}^{1/2}\left\{ \frac{1}{n} \sum_{j \in B_l}  U^{(l)}_{jj}(t_2)U^{(l)}_{jj}(t_1-t_2)   \right\}.\nonumber\\
\end{eqnarray}

For fixed $j, p, q \in B_l$, using (\ref{dif}), 
\begin{eqnarray}
\label{deriv_need1}
D_{pq}U^{(l)}_{jj}(t) &=& i\beta_{pq}\left[ U^{(l)}_{jp}*U^{(l)}_{jq}(t) +  U^{(l)}_{jp}*U^{(l)}_{jq}(t) \right]\nonumber\\
&=&  2i\beta_{pq}\int^t_0 U^{(l)}_{jp}(t-h)U^{(l)}_{jq}(h)dh.\nonumber\\
\end{eqnarray}
Using (\ref{deriv_need1}), recalling that $\beta_{pq} = (1+\delta_{pq})^{-1} \leq 1$, and the Cauchy-Schwartz inequality, it follows that
\begin{equation}
\label{libby}
\left|D_{pq} U^{(l)}_{jj}(t)\right|^2 \leq 4|t| \int^t_0 | U^{(l)}_{jp}(t-h)U^{(l)}_{jq}(h)|^2dh.
\end{equation}

Using (\ref{libby}), the fact that $|U^{(l)}_{jk}(t)| \leq 1$, and the inequality $2ab \leq a^2 + b^2$, it follows that 
\begin{eqnarray}
\label{libby2}
&&\lefteqn{\left|D_{pq}\{ U^{(l)}_{jj}(t_2) U^{(l)}_{jj}(t_1-t_2) \}\right|^2}\nonumber\\
&\leq& 2|D_{pq} U^{(l)}_{jj}(t_2)|^2 + 2|D_{pq} U^{(l)}_{jj}(t_1-t_2)|^2\nonumber\\
&\leq & 8|t| \left( \int^{t_2}_0 | U^{(l)}_{jp}(t_2-h)U^{(l)}_{jq}(h)|^2dh + \int^{t_1-t_2}_0 | U^{(l)}_{jp}(t_1-t_2-h)U^{(l)}_{jq}(h)|^2dh  \right).\nonumber\\ 
\end{eqnarray}
Using the Poincar\'{e} inequality,  (\ref{libby2}), adding more nonnegative terms, and using the property of the unitary matrices that 
\begin{equation}
\label{eqn_unitrow1}
\sum_{k=1}^n|U^{(l)}_{jk}(t)|^2=1,
\end{equation}
it follows that 
\begin{eqnarray}
\label{libby4}
\lefteqn{\mathbf{Var}\left\{U^{(l)}_{jj}(t_2)U^{(l)}_{jj}(t_1-t_2)\right\}}\nonumber\\
 &\leq& \sum_{\begin{array}{cc}
            p\leq q \\ p,q \in B_l 
            \end{array} } \E\left[ (M^{(l)}_{pq})^2  \right]\E \left[ \left| D_{pq}\{U^{(l)}_{jj}(t_2)U^{(l)}_{jj}(t_1-t_2)\}   \right|^2 \right]\nonumber\\ 
&\leq&\frac{8(\sigma^2+1)|t|}{n}  \sum^n_{p=1}\sum^n_{q=1} \E  \left[\int^{t_2}_0 \left| U^{(l)}_{jp}(t_2-h)U^{(l)}_{jq}(h) \right|^2dh + \int^{t_1-t_2}_0 \left| U^{(l)}_{jp}(t_1-t_2-h)U^{(l)}_{jq}(h) \right|^2dh \right]\nonumber\\ 
&\leq& \frac{8(\sigma^2+1)|t|}{n}  \sum^n_{p=1}  \E \left[\int^{t_2}_0 \left| U^{(l)}_{jp}(t_2-h) \right|^2dh + \int^{t_1-t_2}_0 \left| U^{(l)}_{jp}(t_1-t_2-h) \right|^2dh \right]\nonumber\\
&\leq& \frac{16(\sigma^2+1)|t|}{n} t_1\nonumber\\ 
&=& O\left(\frac{1}{n}\right).    
\end{eqnarray}

Now, combining  (\ref{libby3}) with (\ref{libby4}), we have that 
\begin{equation}
\label{libby6}
 \E \left[\frac{1}{n} \sum^{n}_{j = 1} \mathbf{1}_{\{j \in B_{l} \}}  U^{(l)}_{jj}(t_{2}) U^{(l)}_{jj}(t_{1} - t_{2}) e^{\circ}_{n}(x) \right] = O\left(  \frac{1}{n}\right),
\end{equation}
and it follows that
\begin{equation}
\label{Lib4}
\left| - \frac{(\sigma^{2}-2)}{n} \int^{t}_{0} \int^{t_{1}}_{0} \E \left[ \sum^{n}_{j = 1} \mathbf{1}_{\{j \in B_{l} \}}  U^{(l)}_{jj}(t_{2}) U^{(l)}_{jj}(t_{1} - t_{2}) e^{\circ}_{n}(x) \right] dt_{2}dt_{1}  \right| = O\left(\frac{1}{n}\right).
\end{equation}

Now consider the final term of the remainder, given by (\ref{fishy_sub5}). We apply the identity below
\begin{equation}
\label{libby5}
\varphi'_{r}(M^{(r)}) = i \int^{\infty}_{-\infty} h\widehat{\varphi}_{r}(h)U^{(r)}(h)dh,
\end{equation}
which is a consequence of the matrix version of the Fourier inversion formula (\ref{eqn:invfour}). Using (\ref{libby5}), the finiteness of the integral (\ref{Lib_Int}), the above estimate (\ref{libby6}), and the dominated convergence theorem, we have that
\begin{eqnarray}
\label{Lib5}
&& \left| - \frac{x(\sigma^{2}-2)}{n}\sum^{d}_{r = 1} \alpha_{r} \int^{t}_{0} \sum^{n}_{j = 1} \mathbf{1}_{\{j \in B_{l} \cap B_r\}} \E\left[ U^{(l)}_{jj}(t_{1}) (\varphi'_{r}(M^{(r)}))_{jj} e^{\circ}_{n}(x)   \right] dt_{1}   \right|\nonumber\\
&\leq& \sum^d_{r=1}|x(\sigma^2-2)\alpha_r| \left| \int^t_0 \int^{\infty}_{-\infty}h\widehat{\varphi}_r(h)  \frac{1}{n}\sum^n_{j \in B_l \cap B_r}\E\left[ U^{(l)}_{jj}(t_{1}) U^{(r)}_{jj}(h)_{jj} e^{\circ}_{n}(x)   \right] dhdt_{1}\right| \to 0. \nonumber\\
\end{eqnarray}

Combining (\ref{Lib1}), (\ref{Lib2}), (\ref{Lib3}), (\ref{Lib4}), (\ref{Lib5}), and comparing to the remainder term (\ref{fishy}), the proposition is proved.
\end{proof}

The goal now is to pass to the limit in (\ref{eqn:Integral1}).  In what follows let $\{U^{\gamma}_{k}(x)\}$ denote the (rescaled) Chebyshev polynomials of the second kind on $ [-2\sqrt{\gamma}, \ 2\sqrt{\gamma}]$,
\begin{equation}
U^{\gamma}_{k}(x) = \sum^{\lfloor k/2 \rfloor}_{j=0} (-1)^{j} \binom{k-j}{j}\left(\frac{x}{2\sqrt{\gamma}}\right)^{k-2j}.
\end{equation}
\begin{prop}
\label{prop:limit}
Let $A^{(l)}_{n}(t)$ be given by (\ref{eqn:limA}), $Q^{(l)}_n(t)$ given by (\ref{eqn:limQ}), and $\bar{v}_n(t)$ given by (\ref{eqn_vnbar}) . Then the limits of $A^{(l)}_{n}(t), Q^{(l)}_n(t) $ and $\bar{v}_n(t)$ as $n \rightarrow \infty$ exist and
\begin{eqnarray}
\label{eqn_A}
A^{(l)}(t) &:=& \lim_{n \to \infty}A^{(l)}_{n}(t)\nonumber\\
&=& -\frac{1}{2\pi^{2} \gamma_{l}} \sum^{d}_{r = 1}\frac{\alpha_{r}}{\gamma_{r}}\int^{t}_{0}\int^{2\sqrt{\gamma_{l}}}_{-2\sqrt{\gamma_{l}}}\int^{2\sqrt{\gamma_{r}}}_{-2\sqrt{\gamma_{r}}} e^{it_{1}x}\varphi'_{r}(y) \sqrt{4\gamma_{l}-x^{2}}\sqrt{4\gamma_{r}-y^{2}}F_{lr}(x,y)dydxdt_{1},\nonumber\\
\end{eqnarray}
where 
\begin{equation}
\label{eqn:kernel}
F_{lr}(x,y) = \sum^{\infty}_{k=0} U^{\gamma_{l}}_{k}(x)U^{\gamma_{r}}_{k}(y)\frac{\gamma^{k+1}_{lr}}{\gamma^{k/2}_{l}\gamma^{k/2}_{r}},
\end{equation}
the limit of $Q^{(l)}_n(t)$ is given by
\begin{eqnarray}
\label{eqn:asympQ}
Q^{(l)}(t) &:=& \lim_{n \to \infty} Q^{(l)}_{n}(t)\nonumber\\
&=& -\frac{(\sigma^{2} - 2)}{4\pi^{2} \gamma_{l}} \sum^{d}_{r = 1} \frac{\gamma_{lr}\alpha_{r}}{\gamma_{r}} \int^{t}_{0}\int^{2\sqrt{\gamma_{l}}}_{-2\sqrt{\gamma_{l}}} e^{it_{1}\lambda}\sqrt{4\gamma_{l} - \lambda^{2}}d\lambda dt_{1} \cdot \int^{2\sqrt{\gamma_{r}}}_{-2\sqrt{\gamma_{r}}} \varphi'_{r}(\mu) \sqrt{4\gamma_{r} - \mu^{2}} d\mu,\nonumber\\
\end{eqnarray}
and the limit of $\bar{v}_n(t)$, after rescaling by $\gamma_l$, is given by
\begin{equation}
v^{(l)}(t) := \frac{1}{\gamma_l} \lim_{n \to \infty} \bar{v}_n(t)= \frac{1}{2\pi \gamma_{l} }\int^{2\sqrt{\gamma_{l}}}_{-2\sqrt{\gamma_{l}}}e^{itx}\sqrt{4\gamma_{l}-x^{2}}dx.
\label{limitforv}
\end{equation}
\end{prop}
\begin{proof}
Recall that $  A^{(l)}_{n}(t) = -2 \sum^{d}_{r = 1} \alpha_{r} \int^{t}_{0} \frac{1}{n}\E\left[\Tr\{P^{(l)}U^{(l)}(t_{1})P^{(l,r)} \varphi'_{r}(M^{(r)})P^{(r)} \}\right] dt_{1}$. In the full Wigner matrix case one has $ A_{n}(t) = -2\int^{t}_{0}\frac{1}{n}\E\Tr\{e^{itM}\varphi'(M)\}dt_{1}$, and the limiting behavior follows immediately from the Wigner semicircle law. In the case of submatrices with asymptotically regular intersections there are additional technical difficulties due to the fact that for the $n \times n$ submatrices $ M^{(l)} = P^{(l)}MP^{(l)}$, we have
\begin{equation}
\Tr\{P^{(l,r)}U^{(l)}(t)\varphi'_{r}(M^{(r)})P^{(l,r)}\} \ =  \sum_{j,k \in B_{l}\cap B_{r}}  U^{(l)}_{jk}(t)\varphi'_{r}(M^{(r)})_{jk},
\end{equation}
so that the summation is restricted to entries common to both submatrices, i.e. to $ j,k \in B_{l} \cap B_{r}$. It follows from lemma \ref{lem:diag} that the limit of $A^{(l)}_{n}(t)$ exists and equals
\begin{equation}
A^{(l)}(t) =  -2 \sum^{d}_{r = 1} \alpha_{r} \int^{t}_{0} \langle e^{it_{1}x}, \varphi'_{r} \ \rangle_{lr} \ dt_{1},
\end{equation}
where 
\begin{equation}
\langle e^{it_{1}x}, \ \varphi'_{r} \rangle_{lr}\  =  \frac{1}{4\pi^{2}\gamma_{l}\gamma_{r}} \int^{2\sqrt{\gamma_{l}}}_{-2\sqrt{\gamma_{l}}} \int^{2\sqrt{\gamma_{r}}}_{-2\sqrt{\gamma_{r}}} e^{it_{1}x}\varphi'_{r}(y)F_{lr}(x,y) \sqrt{4\gamma_{l}-x^{2}}\sqrt{4\gamma_{r}-y^{2}}dydx.
\end{equation}
This establishes (\ref{eqn_A}). The proof of lemma \ref{lem:diag} will be given in section \ref{subsec:bilin}.

We turn our attention to $Q^{(l)}_n(t)$. First it will be argued that the variance of the matrix entries converge to zero. Using the Poincar\'{e} inequality, (\ref{libby}), (\ref{eqn_unitrow1}), and proposition \ref{superman}, it follows that
\begin{eqnarray}
\label{golo}
\lefteqn{\mathbf{Var}\left\{U^{(l)}_{jj}(t_1)\right\}}\nonumber\\
 &\leq& \sum_{p\leq q, p,q \in B_l} \E\left[ (M^{(l)}_{pq})^2  \right]\E \left[ \left| D_{pq}U^{(l)}_{jj}(t)   \right|^2 \right]\nonumber\\ 
&\leq&\frac{4(\sigma^2+1)|t_1|}{n}  \sum^n_{p=1}\sum^n_{q=1} \E  \int^{t_1}_0 \left| U^{(l)}_{jp}(t_1-t_2)U^{(l)}_{jq}(t_2) \right|^2dt_2\nonumber\\ 
&\leq& \frac{4(\sigma^2+1)|t_1|}{n}  \sum^n_{p=1}  \E \int^{t_1}_0 \left| U^{(l)}_{jp}(t_1-t_2) \right|^2dt_2\nonumber\\
&\leq& \frac{4(\sigma^2+1)t^2_1}{n} = O\left(n^{-1} \right).
\end{eqnarray}
Note that in the course of the calculation (\ref{golo}), we showed that 
\begin{equation}
\label{golo2}
\sum_{p\leq q} \E \left[ \left| D_{pq}U^{(l)}_{jj}(t_1)   \right|^2 \right] \leq 4t^2_1.
\end{equation}

The Cauchy-Schwartz inequality implies
\begin{equation}
\int_\R (1+t^2_1)|\widehat{\varphi}_r(t_1)|dt_1\leq \sqrt{ \int_\R \frac{dt_1}{(1+ t^2_1)^{1/2+\epsilon}}} \cdot \sqrt{\int_\R(1+ t^2_1)^{5/2+\epsilon}|\widehat{\varphi}_r(t_1)|^2dt_1}.
\end{equation}
Since $||\varphi_r||_{5/2+\epsilon} < \infty$, we have the estimate
\begin{equation}
\label{golo3}
\int^{\infty}_{-\infty} t^2_1|\widehat{\varphi}_r(t_1)|dt_1 < \infty.
\end{equation}
Using the Cauchy-Schwartz inequality and (\ref{libby5}), it follows that
\begin{eqnarray}
\label{golo4}
\left| D_{pq}\varphi'_r(M^{(r)})_{jj}  \right|^2&=& \left| \int^{\infty}_{-\infty}t_1\widehat{\varphi}_r(t_1) D_{pq}U^{(l)}_{jj}(t_1)dt_1 \right|^2\nonumber\\ 
&\leq& \int^{\infty}_{-\infty} t^2_1|\widehat{\varphi}_r(t_1)|dt_1 \cdot  \int^{\infty}_{-\infty} |\widehat{\varphi} _r(t_1)|\cdot \left| D_{pq}U^{(l)}_{jj}(t_1)   \right|^2dt_1.
\end{eqnarray}
Using the Poincar\'{e} inequality, (\ref{golo2}), (\ref{golo4}), we obtain
\begin{eqnarray}
\label{golo5}
\lefteqn{\mathbf{Var}\left\{ \varphi'_r(M^{(r)})_{jj}  \right\}}\nonumber\\
&\leq& \sum_{p\leq q} \E\left[ (M^{(l)}_{pq})^2  \right]  \E \left[ \left| D_{pq} \varphi'_r(M^{(r)})_{jj}    \right|^2 \right]\nonumber\\
&\leq& \frac{(\sigma^2+1)}{n} \cdot \int^{\infty}_{-\infty} t^2_1|\widehat{\varphi}_r(t_1)|dt_1\cdot \sum_{p\leq q}  \int^{\infty}_{-\infty}|\widehat{\varphi}_r(t_1)|\E \left[ \left| D_{pq}U^{(l)}_{jj}(t_1)   \right|^2 \right]dt_1\nonumber\\
&\leq& \frac{4(\sigma^2+1)}{n} \cdot \left(\int^{\infty}_{-\infty} t^2_1|\widehat{\varphi}_r(t_1)|dt_1\right)^2.
\end{eqnarray}
Using (\ref{golo3}), (\ref{golo5}), (\ref{golo}), and the Cauchy-Schwartz inequality, we obtain
\begin{equation}
\label{golo6}
\mathbf{Cov}\{ U^{(l)}_{jj}(t_1), \varphi'_{r}(M^{(r)})_{jj}  \} \leq \sqrt{\mathbf{Var}\{U^{(l)}_{jj}(t_1)\}} \cdot \sqrt{\mathbf{Var}\left\{ \varphi'_r(M^{(r)})_{jj} \right\}} = O\left( n^{-1}  \right).
\end{equation}
Using (\ref{golo6}) it is justified to replace the expectation $ \E[ U^{(l)}_{jj}(t) \varphi'_{r}(M^{(r)})_{jj} ]$ by the product $ \E[ U^{(l)}_{jj}(t)] \cdot \E[ \varphi'_{r}(M^{(r)})_{jj} ]$, when passing to the limit. We use proposition $2.1$ of \cite{PRS}, which guarantees that for $f \in C^7_c(\R)$,
\begin{equation}
\label{golo7}
\lim_{n \to \infty} \E[ f(M)_{jj}] = \int_{\R} f(x) d\mu_{sc}(x).
\end{equation}
 In order to apply this asymptotic to the exponential function, which is smooth enough, we truncate the function in a smooth fashion outside the support of $\mu_{sc}$. We are justified in replacing the exponential function by its truncated version because the eigenvalues of the submatrices concentrate in the support of the semicircle law, with overwhelming probability. It is for this same reason that we may assume $\varphi'_r$ is compactly supported. This function is not sufficiently smooth, but we can avoid this problem by a density argument using standard convolution, and then apply the bound (\ref{ineq-MS}) on the variance of linear eigenvalue statistics.

Let $\eta \in C^{\infty}_c(\R)$ satisfy $\int_\R \eta(x)dx = 1$, and consider the mollifiers $ \eta_y(x) := y^{-1}\eta(xy^{-1})$. Then \\* $\varphi'_r * \eta_y \in C^{\infty}_c(\R)$, and using standard Fourier theory it can be shown that  
\begin{equation}
\label{smooth_approx}
\lim_{y \to 0} ||\varphi'_r - \varphi'_r * \eta_y  ||^2_{3/2+\epsilon} = 0.
\end{equation}

It follows from (\ref{golo6}) and (\ref{golo7}) that
\begin{eqnarray}
\label{limit_fin}
\lefteqn{ \lim_{n \to \infty} \frac{1}{n}   \sum^{n}_{j = 1} \mathbf{1}_{\{j \in B_{l} \cap B_r\}}  \E\left[ U^{(l)}_{jj}(t) \varphi'_{r}(M^{(r)})_{jj} \right]=}\nonumber\\
&& \gamma_{lr}\left( \frac{1}{2\pi \gamma_{l}} \int^{2\sqrt{\gamma_{l}}}_{-2\sqrt{\gamma_{l}}} e^{it\lambda} \sqrt{4\gamma_{l} - \lambda^{2}} d\lambda  \right)\cdot \left( \frac{1}{2\pi \gamma_{r}} \int^{2\sqrt{\gamma_{r}}}_{-2\sqrt{\gamma_{r}}} \varphi'_{r}(\mu) \sqrt{4\gamma_{r} - \mu^{2}} d\mu \right).
\end{eqnarray} 
Using (\ref{limit_fin}), we pass to the limit in (\ref{eqn:limQ}), and obtain (\ref{eqn:asympQ}). The limit of 
\[\bar{v}^{(l)}_{n}(t) = \frac{1}{n} \E[ u^{(l)}_{n}(t)] \approx \frac{\gamma_l}{|B_l|} \E[ \Tr \{P^{(l)}U^{(l)}(t)P^{(l)}\}],\]
is given by (rescaled) Wigner semicircle law, as a consequence of the zero eigenvalues. Alternatively, it can be computed using the bilinear form in lemma \ref{lem:diag}, with $f(x) = e^{itx}$ and $g(x) =1$. To facilitate solving the integral equation (\ref{eqn:Integral2}), below, it will be useful to rescale by $\gamma_l$. We obtain
\begin{eqnarray}
v^{(l)}(t) &=&  \frac{1}{\gamma_l} \langle e^{itx}, 1 \rangle_{ll}\nonumber\\
&=& \frac{1}{2\pi \gamma_{l} }\int^{2\sqrt{\gamma_{l}}}_{-2\sqrt{\gamma_{l}}}e^{itx}\sqrt{4\gamma_{l}-x^{2}}dx,
\end{eqnarray}
which establishes (\ref{limitforv}). The proposition is proved.
\end{proof}
Now using propositions \ref{eqn_rewrite}, \ref{rem_goto0}, \ref{prop:limit}, we pass to the limit $n_m \to \infty$ in (\ref{eqn:Integral1}), and determine that the limit $Y^{(l)}$ of every uniformly converging subsequence $\{Y^{(l)}_{n_m}\}$ satisfies the equation
\begin{equation}
\label{eqn:Integral2}
Y^{l}(x,t) + 2\gamma_{l}\int^{t}_{0}\int^{t_{1}}_{0} v^{(l)}(t_{1}-t_{2})Y^{(l)}(x,t_{2})dt_{2}dt_{1} = xZ(x)\left[ A^{(l)}(t) + Q^{(l)}(t)\right],
\end{equation}
where $A^{(l)}(t)$ is given by (\ref{eqn_A}), $ Q^{(l)}(t)$ is given by (\ref{eqn:asympQ}), and $v^{(l)}(t)$ is given by (\ref{limitforv}).

Now the argument will proceed by solving the integral equation (\ref{eqn:Integral2}). We use a version of the technique used by L. Pastur and A. Lytova in \cite{cite_key3}, to solve  this equation. Define
\begin{equation}
f(z) := (\sqrt{z^{2}-4\gamma_{l}}-z)/2\gamma_{l},
\end{equation}
which is the Stieltjes transform of the rescaled semicircle law, where $ \sqrt{z^{2}-4\gamma_{l}} = z + O(1/z)$ as $ z \to \infty$. A direct calculation shows that $\tilde{v}^{(l)} = f$, where $\tilde{v}^{(l)}$ denotes the generalized Fourier transform of $v^{(l)}$. We obtain
\begin{eqnarray}
\tilde{v}^{(l)}(z) &:=& \frac{1}{2\pi i\gamma_{l}} \int^{\infty}_{0}\int^{2\sqrt{\gamma_{l}}}_{-2\sqrt{\gamma_{l}}}e^{it(x-z)}\sqrt{4\gamma_{l}-x^{2}}dxdt\nonumber\\
&=&\frac{1}{2\pi \gamma_{l}}\int^{2\sqrt{\gamma_{l}}}_{-2\sqrt{\gamma_{l}}}\frac{1}{x-z}\sqrt{4\gamma_{l}-x^{2}}dx\nonumber\\
&=&f(z).
\end{eqnarray}
We check that
\begin{equation}
z+2\gamma_{l}f(z) = \sqrt{z^{2}-4\gamma_{l}} \neq 0, \ \ \Im{z} \neq 0.
\end{equation} 

Set
\begin{eqnarray}
T(t) &:=& \frac{i}{2\pi} \int_{L}\frac{e^{izt}dz}{z+2\gamma_{l}f(z)}= -\frac{1}{\pi} \int^{2\sqrt{\gamma_{l}}}_{-2\sqrt{\gamma_{l}}}\frac{e^{i\lambda t}d\lambda}{\sqrt{4\gamma_{l}-\lambda^{2}}},
\end{eqnarray}
after replacing the integral over $L$ by the integral over $[-2\gamma_{l}, 2\gamma_{l}]$, and taking into account that $ \sqrt{z^{2}-4\gamma_{l}} $ \ is \ $\pm i \sqrt{4\gamma_{l}-\lambda^{2}}$, on the upper and lower edges of the cut. Then the solution of (\ref{eqn:Integral2}) is 

\begin{equation}
Y^{(l)}(x,t) = -xZ(x)\int^{t}_{0}T(t-t_{1}) \frac{d}{dt_{1}}\left[A^{(l)}(t_{1})+ Q^{(l)}(t_{1})\right]dt_{1}.
\end{equation}

Then, with $F_{lr}$ given by (\ref{eqn:kernel}),
\begin{eqnarray}
\label{eqn:A}
\lefteqn{ \int^{t}_{0}T(t-t_{1}) \frac{d}{dt_{1}}A^{(l)}(t_{1})dt_{1}}\nonumber\\
&=& \frac{1}{2\pi^{3} \gamma_{l}} \sum^{d}_{r = 1}\frac{\alpha_{r}}{\gamma_{r}}\int^{t}_{0}\int^{2\sqrt{\gamma_{l}}}_{-2\sqrt{\gamma_{l}}}\int^{2\sqrt{\gamma_{l}}}_{-2\sqrt{\gamma_{l}}}\int^{2\sqrt{\gamma_{r}}}_{-2\sqrt{\gamma_{r}}} e^{i(t-t_{1})\lambda}e^{it_{1}x}\varphi'_{r}(y) \frac{\sqrt{4\gamma_{l}-x^{2}}\sqrt{4\gamma_{r}-y^{2}}}{\sqrt{4\gamma_{l}-\lambda^{2}}}\nonumber\\
&& \times F_{lr}(x,y)dydxd\lambda dt_{1}\nonumber\\
&=&\frac{1}{2i\pi^{3} \gamma_{l}} \sum^{d}_{r = 1}\frac{\alpha_{r}}{\gamma_{r}}\int^{2\sqrt{\gamma_{l}}}_{-2\sqrt{\gamma_{l}}}\int^{2\sqrt{\gamma_{l}}}_{-2\sqrt{\gamma_{l}}}\int^{2\sqrt{\gamma_{r}}}_{-2\sqrt{\gamma_{r}}} \frac{\left[e^{itx}-e^{it\lambda}\right]\varphi'_{r}(y)}{(x-\lambda)} \frac{\sqrt{4\gamma_{l}-x^{2}}\sqrt{4\gamma_{r}-y^{2}}}{\sqrt{4\gamma_{l}-\lambda^{2}}}F_{lr}(x,y)dydxd\lambda,\nonumber\\ 
\end{eqnarray}
and
\begin{eqnarray}
\label{eqn:Q}
 \lefteqn{\int^{t}_{0}T(t-t_{1}) \frac{d}{dt_{1}}Q^{(l)}(t_{1})dt_{1}}\nonumber\\
&=& -\frac{\gamma_{lr}(\sigma^{2}-2)}{4\pi^{3} \gamma_{l}} \sum^{d}_{r = 1}\frac{\alpha_{r}}{\gamma_{r}}\int^{t}_{0}\int^{2\sqrt{\gamma_{l}}}_{-2\sqrt{\gamma_{l}}}\int^{2\sqrt{\gamma_{l}}}_{-2\sqrt{\gamma_{l}}}\frac{ e^{i(t-t_{1})\lambda}}{\sqrt{4\gamma_{l}-\lambda^{2}}} e^{it_{1}\eta}\sqrt{4\gamma_{l} - \eta^{2}}d\eta d\lambda\nonumber\\
&& \times \int^{2\sqrt{\gamma_{r}}}_{-2\sqrt{\gamma_{r}}} \varphi'_{r}(\mu) \sqrt{4\gamma_{r}-\mu^{2}}d\mu dt_{1} \nonumber\\
&=& -\frac{\gamma_{lr}(\sigma^{2}-2)}{4\pi^{3} \gamma_{l}i} \sum^{d}_{r = 1}\frac{\alpha_{r}}{\gamma_{r}}\int^{2\sqrt{\gamma_{l}}}_{-2\sqrt{\gamma_{l}}}\int^{2\sqrt{\gamma_{l}}}_{-2\sqrt{\gamma_{l}}}\left[ \frac{e^{it\eta}-e^{it\lambda}}{\eta-\lambda}\right] \frac{\sqrt{4\gamma_{l} - \eta^{2}}}{\sqrt{4\gamma_{l}-\lambda^{2}}}d\eta d\lambda \cdot \int^{2\sqrt{\gamma_{r}}}_{-2\sqrt{\gamma_{r}}} \varphi'_{r}(\mu) \sqrt{4\gamma_{r}-\mu^{2}}d\mu.\nonumber\\ 
\end{eqnarray}

Using the regularity condition $ ||\varphi_{l}||_{5/2+\epsilon} < \infty$ for $  1 \leq l \leq d$, (\ref{eqn:A}), (\ref{eqn:Q}), and the dominated convergence theorem to pass to limit in (\ref{eqn:deriv}) yields
\begin{eqnarray}
& Z'(x)& \nonumber\\
 &=& i \ \sum^{d}_{l = 1} \alpha_{l} 
  \int^{\infty}_{-\infty} \widehat{\varphi}_{l}(t) Y^{(l)}(x,t) dt\nonumber\\
  &=& -\frac{xZ(x)}{2\pi^{3}} \ \sum^{d}_{l = 1} \sum^{d}_{r=1} \frac{\alpha_{l}\alpha_{r} }{\gamma_{l}\gamma_{r}}
  \int^{\infty}_{-\infty}\int^{2\sqrt{\gamma_{l}}}_{-2\sqrt{\gamma_{l}}}\int^{2\sqrt{\gamma_{l}}}_{-2\sqrt{\gamma_{l}}}\int^{2\sqrt{\gamma_{r}}}_{-2\sqrt{\gamma_{r}}} \widehat{\varphi}_{l}(t)\frac{\left[e^{itx}-e^{it\lambda}\right]{\varphi'_{r}(y)}}{(x-\lambda)}\nonumber\\
&&\times \frac{\sqrt{4\gamma_{l}-x^{2}}\sqrt{4\gamma_{r}-y^{2}}}{\sqrt{4\gamma_{l}-\lambda^{2}}} F_{lr}(x,y)dydxd\lambda  dt\nonumber\\
&& -\frac{\gamma_{lr}(\sigma^{2}-2)xZ(x)}{4\pi^{3}} \sum^{d}_{l = 1} \sum^{d}_{r = 1}\frac{\alpha_{l}\alpha_{r}}{\gamma_{l} \gamma_{r}} \int^{\infty}_{-\infty}  \int^{2\sqrt{\gamma_{l}}}_{-2\sqrt{\gamma_{l}}}\int^{2\sqrt{\gamma_{l}}}_{-2\sqrt{\gamma_{l}}}\left[ \frac{\widehat{\varphi_{l}}(t)e^{it\eta}-\widehat{\varphi_{l}}(t)e^{it\lambda}}{\eta-\lambda}\right] \frac{\sqrt{4\gamma_{l} - \eta^{2}}}{\sqrt{4\gamma_{l}-\lambda^{2}}}d\eta d\lambda \nonumber\\
&& \times \int^{2\sqrt{\gamma_{r}}}_{-2\sqrt{\gamma_{r}}} \varphi'_{r}(\mu) \sqrt{4\gamma_{r}-\mu^{2}}d\mu dt.
\end{eqnarray}

Applying the Fourier inversion formula (\ref{eqn:invfour}), it follows that
\begin{eqnarray}
\label{Zderivative}
&& Z'(x) =\nonumber\\
&& - \frac{xZ(x)}{2\pi^{3}} \ \sum^{d}_{l = 1} \sum^{d}_{r=1} \frac{\alpha_{l}\alpha_{r} }{\gamma_{l}\gamma_{r}}
 \int^{2\sqrt{\gamma_{l}}}_{-2\sqrt{\gamma_{l}}}\int^{2\sqrt{\gamma_{l}}}_{-2\sqrt{\gamma_{l}}}\int^{2\sqrt{\gamma_{r}}}_{-2\sqrt{\gamma_{r}}}\frac{\left[ \varphi_{l}(x)-\varphi_{l}(\lambda) \right]\varphi'_{r}(y)}{(x-\lambda)} \frac{\sqrt{4\gamma_{l}-x^{2}}\sqrt{4\gamma_{r}-y^{2}}}{\sqrt{4\gamma_{l}-\lambda^{2}}}\nonumber\\
 && \times F_{lr}(x,y)dydxd\lambda \nonumber\\
 && -xZ(x)\frac{\gamma_{lr}(\sigma^{2}-2)}{4\pi^{3}} \sum^{d}_{l = 1} \sum^{d}_{r = 1}\frac{\alpha_{l}\alpha_{r}}{\gamma_{l} \gamma_{r}}   \int^{2\sqrt{\gamma_{l}}}_{-2\sqrt{\gamma_{l}}}\int^{2\sqrt{\gamma_{l}}}_{-2\sqrt{\gamma_{l}}}\left[ \frac{\varphi_{l}(\eta)-\varphi_{l}(\lambda)}{\eta-\lambda}\right] \frac{\sqrt{4\gamma_{l} - \eta^{2}}}{\sqrt{4\gamma_{l}-\lambda^{2}}}d\eta d\lambda\nonumber\\
 && \times \int^{2\sqrt{\gamma_{r}}}_{-2\sqrt{\gamma_{r}}} \varphi'_{r}(\mu) \sqrt{4\gamma_{r}-\mu^{2}}d\mu.
\end{eqnarray}

We will use the fact that
\begin{equation}
\label{eqn:chebyint}
\int^{2\sqrt{\gamma}}_{-2\sqrt{\gamma}} \frac{\left[ T^{\gamma}_{k}(x) - T^{\gamma}_{k}(\lambda)  \right]}{(x-\lambda)}\frac{d\lambda}{\sqrt{4\gamma-\lambda^{2}}} = \ \frac{\pi}{2\sqrt{\gamma}} U^{\gamma}_{k-1}(x), \ \ k \geq 1.
\end{equation}
Expand the test function $\varphi_{l}$ in the Chebyshev basis to obtain
\begin{equation}
\label{chebyexpansion}
\varphi_{l}(x) = \sum^{\infty}_{k=0} (\varphi_{l})_{k} T^{\gamma_{l}}_{k}(x), \ \ \ (\varphi_{l})_{k} = \frac{2}{\pi} \int^{2\sqrt{\gamma_{l}}}_{-2\sqrt{\gamma_{l}}} \varphi_{l}(t) T^{\gamma_{l}}_{k}(t) \frac{dt}{\sqrt{4\gamma_{l}-t^{2}}}.
\end{equation}

Returning to the computation of $Z'(x)$,  using (\ref{Zderivative}), (\ref{eqn:chebyint}), and (\ref{chebyexpansion}), it follows that
\begin{eqnarray}
\label{eqn:expderiv}
Z'(x)&=&-\frac{xZ(x)}{4\pi^{2}} \ \sum^{d}_{l = 1} \sum^{d}_{r=1} \sum^{\infty}_{k=1}\frac{\alpha_{l}\alpha_{r} }{\gamma^{3/2}_{l}\gamma_{r}} (\varphi_{l})_{k}\int^{2\sqrt{\gamma_{l}}}_{-2\sqrt{\gamma_{l}}}\int^{2\sqrt{\gamma_{r}}}_{-2\sqrt{\gamma_{r}}}U^{\gamma_{l}}_{k-1}(x)\varphi'_{r}(y) \sqrt{4\gamma_{l}-x^{2}}\sqrt{4\gamma_{r}-y^{2}}\nonumber\\
&& \times F_{lr}(x,y)dydx\nonumber\\
&&  -xZ(x)\frac{\gamma_{lr}(\sigma^{2}-2)}{8\pi^{2}} \sum^{d}_{l = 1} \sum^{d}_{r = 1}\frac{\alpha_{l} \alpha_{r}}{\gamma^{3/2}_{l} \gamma_{r}} \sum^{\infty}_{k=1} (\varphi_{l})_{k}   \int^{2\sqrt{\gamma_{l}}}_{-2\sqrt{\gamma_{l}}} U^{\gamma_{l}}_{k-1}(\eta) \sqrt{4\gamma_{l} - \eta^{2}}d\eta\nonumber\\
&&\times \int^{2\sqrt{\gamma_{r}}}_{-2\sqrt{\gamma_{r}}} \varphi'_{r}(\mu) \sqrt{4\gamma_{r}-\mu^{2}}d\mu.
\end{eqnarray}

Using the orthogonality of the Chebyshev polynomials (\ref{orthogcheby2}), 
\begin{equation}
\label{useorthog}
 \sum^{\infty}_{k =1} (\varphi_{l})_{k}    \int^{2\sqrt{\gamma_{l}}}_{-2\sqrt{\gamma_{l}}} U^{\gamma_{l}}_{k-1}(\eta) \sqrt{4\gamma_{l} - \eta^{2}}d\eta = 2\sqrt{\gamma_{l}}\int^{2\sqrt{\gamma_{l}}}_{-2\sqrt{\gamma_{l}} } \frac{\lambda \varphi_{l}(\lambda)}{\sqrt{4\gamma_{l}-\lambda^{2}}} d\lambda.
 \end{equation}
Integrating by parts yields
\begin{equation}
\label{integparts}
\int^{2\sqrt{\gamma_{r}}}_{-2\sqrt{\gamma_{r}}} \varphi'_{r}(\mu) \sqrt{4\gamma_{r}-\mu^{2}}d\mu = \int^{2\sqrt{\gamma_{r}}}_{-2\sqrt{\gamma_{r}}} \frac{\mu \varphi_{r}(\mu)}{\sqrt{4\gamma_{r}-\mu^{2}}} d\mu,
\end{equation}
so that 
\begin{equation}
\label{chebytermone}
 \frac{\gamma_{lr}(\sigma^{2}-2)}{4\pi^{2}\gamma_{l}\gamma_{r}}\int^{2\sqrt{\gamma_{l}}}_{-2\sqrt{\gamma_{l}} } \frac{\lambda \varphi_{l}(\lambda)}{\sqrt{4\gamma_{l}-\lambda^{2}}} d\lambda \cdot \int^{2\sqrt{\gamma_{r}}}_{-2\sqrt{\gamma_{r}}} \frac{\mu \varphi_{r}(\mu)}{\sqrt{4\gamma_{r}-\mu^{2}}} d\mu =\frac{(\sigma^{2} - 2)}{4} \frac{\gamma_{lr}}{\sqrt{\gamma_{l}\gamma_{r}}} (\varphi_{l})_{1} (\varphi_{r})_{1}.
\end{equation}

Since
\begin{equation}
\frac{d}{dy} T^{\gamma}_{k}(y) = \frac{k}{2\sqrt{\gamma}}U^{\gamma}_{k-1}(y),
\end{equation}
we expand $\varphi_{r}(y) $ in the Chebyshev basis to obtain
\begin{equation}
\label{eqn:chebyderivexp}
\varphi'_{r}(y) = \frac{1}{2\sqrt{\gamma_{r}}} \sum^{\infty}_{m=1} m(\varphi_{r})_{m} U^{\gamma_{r}}_{m-1}(y).
\end{equation}

Recalling that $F_{lr}$ is given by (\ref{eqn:kernel}), it follows that
\begin{eqnarray}
\label{eqn:orthcomp}
&& \sum^{\infty}_{k=1} \sum^{\infty}_{m=1}  m(\varphi_{l})_{k} (\varphi_{r})_{m}\left[\int^{2\sqrt{\gamma_{l}}}_{-2\sqrt{\gamma_{l}}}\int^{2\sqrt{\gamma_{r}}}_{-2\sqrt{\gamma_{r}}}U^{\gamma_{l}}_{k-1}(x)U^{\gamma_{r}}_{m-1}(y) \sqrt{4\gamma_{l}-x^{2}}\sqrt{4\gamma_{r}-y^{2}}F_{lr}(x,y)dydx\right]\nonumber\\
&=&\sum^{\infty}_{k=1} \sum^{\infty}_{m=1}\sum^{\infty}_{j=0} m(\varphi_{l})_{k} (\varphi_{r})_{m}\int^{2\sqrt{\gamma_{l}}}_{-2\sqrt{\gamma_{l}}}\int^{2\sqrt{\gamma_{r}}}_{-2\sqrt{\gamma_{r}}}U^{\gamma_{l}}_{k-1}(x)U^{\gamma_{r}}_{m-1}(y)  U^{\gamma_{l}}_{j}(x)U^{\gamma_{r}}_{j}(y)\nonumber\\
&& \times\sqrt{4\gamma_{l}-x^{2}}\sqrt{4\gamma_{r}-y^{2}}dydx\frac{\gamma^{j+1}_{lr}}{\gamma^{j/2}_{l}\gamma^{j/2}_{r}}.\nonumber\\
&=&4\pi^{2}\gamma_{l}\gamma_{r}\sum^{\infty}_{k=1} k(\varphi_{l})_{k}(\varphi_{r})_{k} \left( \frac{\gamma^{k}_{lr}}{\gamma^{(k-1)/2}_{l}\gamma^{(k-1)/2)}_{r}}\right).
\end{eqnarray}

Using (\ref{eqn:orthcomp}), (\ref{useorthog}), (\ref{integparts}) and  (\ref{chebytermone}),  in (\ref{eqn:expderiv}), it follows that
\begin{eqnarray}
\label{eqn:Zprimefin}
Z'(x)&=& -\frac{xZ(x)}{2} \sum^{d}_{l=1}\sum^{d}_{r=1} \alpha_{l}\alpha_{r}  \left[  \frac{(\sigma^{2} - 2)}{2} \frac{\gamma_{lr}}{\sqrt{\gamma_{l}\gamma_{r}}} (\varphi_{l})_{1} (\varphi_{r})_{1} +   \sum^{\infty}_{k=1} k (\varphi_{l})_{k}(\varphi_{r})_{k} \left(\frac{\gamma_{lr}}{\sqrt{\gamma_{l}\gamma_{r}}}\right)^{k}\right]\nonumber\\
&=& -xZ(x)  \sum^{d}_{l=1}\alpha_{l}^{2} \left[ \frac{\sigma^{2}}{4} (\varphi_{l})^{2}_{1} +\frac{1}{2} \sum^{ \infty}_{k=2} k (\varphi_{l})^{2}_{k} \right] \nonumber\\
&&-xZ(x)\sum_{1\leq l < r \leq d}2 \alpha_{l}\alpha_{r} \left[ \frac{\sigma^{2}}{4} (\varphi_{l})_{1}(\varphi_{r})_{1} (\frac{\gamma_{lr}}{\sqrt{\gamma_{l}\gamma_{r}}}) + \frac{1}{2}\sum^{\infty}_{k=2} k(\varphi_{l})_{k}(\varphi_{r})_{k} \left(\frac{\gamma_{lr}}{\sqrt{\gamma_{l}\gamma_{r}}}\right)^{k} \right].\nonumber\\
\end{eqnarray}

We have obtained the expression for the asymptotic covariance (\ref{GaussCov}) in terms of Chebyshev polynomials. Now we write this expression as a contour integral. Let
\[\beta := \frac{\gamma_{lr}}{\sqrt{\gamma_{l}\gamma_{r}}},\]
make the change of coordinates $ x = 2\sqrt{\gamma_{l}}cos(\theta), \ y = 2\sqrt{\gamma_{r}}cos(\omega)$, and use (\ref{eqn:chebyexp}) to obtain that
\begin{eqnarray}
\lefteqn{\frac{1}{2}\sum^{\infty}_{k=1} k\beta^{k}(\varphi_{l})_{k}(\varphi_{r})_{k}}\nonumber\\
&=&\frac{2}{\pi^{2}} \sum^{\infty}_{k=1}k\beta^{k} \int^{2\sqrt{\gamma_{l}}}_{-2\sqrt{\gamma_{l}}}\int^{2\sqrt{\gamma_{r}}}_{-2\sqrt{\gamma_{r}}}\varphi_{l}(x)\varphi_{r}(y)T_{k}\left(\frac{x}{2\sqrt{\gamma_{l}}}\right)T_{k}\left(\frac{y}{2\sqrt{\gamma_{r}}}\right) \frac{dxdy}{\sqrt{4\gamma_{l}-x^{2}}\sqrt{4\gamma_{r}-y^{2}}}\nonumber\\
&=& \frac{2}{\pi^{2}}  \int^{\pi}_{0}\int^{\pi}_{0}\sum^{\infty}_{k=1} k\beta^{k}cos(k\theta)cos(k\omega)\varphi_{l}\left(2\sqrt{\gamma_{l}}cos\theta\right)\varphi_{r} \left(  2\sqrt{\gamma_{r}}cos\omega\right)d\theta d\omega.
\end{eqnarray}
Integrating by parts in $\theta, \omega$ it follows that
\begin{eqnarray}
\label{eqn:Bcov}
\frac{1}{2}\sum^{\infty}_{k=1} k\beta^{k}(\varphi_{l})_{k}(\varphi_{r})_{k}&=&\frac{2}{\pi^{2}}\int^{\pi}_{0}\int^{\pi}_{0}\varphi'_{l}\left(2\sqrt{\gamma_{l}}cos\theta\right)\varphi'_{r} \left(  2\sqrt{\gamma_{r}}cos\omega \right) \left[ \sum^{\infty}_{k=1}\frac{\beta^{k}}{k} sin(k\theta)sin(k\omega)  \right]\nonumber\\
 &&\times (2\sqrt{\gamma_{l}}sin\theta)(2\sqrt{\gamma_{r}} sin\omega) d\theta d\omega.
\end{eqnarray}

To evaluate the infinite sum above, recall that for $z \in \mathbb{C}$ with $ |z| <1$, we have 
\begin{equation}
\label{eqn:log}
\ln\left(1-z\right) = -\ \sum^{\infty}_{k=1} \frac{z^{k}}{k}.
\end{equation}
Noting that $\beta <1$, using (\ref{eqn:log}), it follows that
\begin{eqnarray}
\lefteqn{\sum^{\infty}_{k=1}\frac{\beta^{k}}{k} sin(k\theta)sin(k\omega)}\nonumber\\
&=& -\frac{1}{4} \sum^{\infty}_{k=1}\frac{\beta^{k}}{k}\left[ e^{ik\theta}-e^{-ik\theta}  \right]\left[ e^{ik\omega}-e^{-ik\omega}  \right]\nonumber\\
&=&-\frac{1}{4} \left[ -\ln\left(1-\beta e^{i(\theta+\omega)}\right)+\ln\left(1- \beta e^{i(\theta-\omega)}\right) -\ln\left(1-\beta e^{-i(\theta+\omega)}\right)  +\ln\left(1-\beta e^{-i(\theta-\omega)}\right) \right]\nonumber\\
&=&-\frac{1}{4} \left[\ln\left[\left(1-\beta e^{i(\theta-\omega)}  \right)\overline{\left(1-\beta e^{i(\theta-\omega)}    \right)} \right] - \ln\left[\left(1-\beta e^{i(\theta+\omega)}  \right)\overline{\left(1-\beta e^{i(\theta+\omega)}    \right)} \right] \right].
\end{eqnarray}

Making the change of coordinates $ z = \sqrt{\gamma_{l}}e^{i\theta}, \ w = \sqrt{\gamma_{r}}e^{i\omega}$, and recalling that $\beta = \frac{\gamma_{lr}}{\sqrt{\gamma_l\gamma_r}}$, this can be written as
\begin{eqnarray}
\sum^{\infty}_{k=1}\frac{\beta^{k}}{k} \sin(k\theta)\sin(k\omega)&=&-\frac{1}{4} \ln\left[\frac{ \left(1-\beta e^{i(\theta-\omega)}  \right)\overline{\left(1-\beta e^{i(\theta-\omega)}    \right)} }{ \left(1-\beta e^{i(\theta+\omega)}  \right)\overline{\left(1-\beta e^{i(\theta+\omega)}  \right)}} \right]\nonumber\\
&=&-\frac{1}{4} \ln\left[\frac{ \left(1-\frac{\gamma_{lr}}{\gamma_l\gamma_r} z\overline{w}  \right)\overline{\left(1-\frac{\gamma_{lr}}{\gamma_l\gamma_r} z\overline{w}  \right)} }{ \left(1-\frac{\gamma_{lr}}{\gamma_l\gamma_r} zw  \right)\overline{\left(1-\frac{\gamma_{lr}}{\gamma_l\gamma_r} zw  \right)} } \right]\nonumber\\
&=&-\frac{1}{4}\ln\left[ \frac{|\gamma_{lr}-z\overline{w}|^2}{|\gamma_{lr}-zw|^2}   \right]\nonumber\\
&=&\frac{1}{2}\ln\left|\frac{\gamma_{lr}-zw}{\gamma_{lr}-z\overline{w}}  \right|.
\label{eqn:critlog}
\end{eqnarray}
Combining  (\ref{eqn:Bcov}), (\ref{eqn:critlog}), and noting that
\[\left(2\sqrt{\gamma_l}\sin\theta  \right)\left(2\sqrt{\gamma_r}\sin\omega  \right)d\theta d\omega = \left(1-\frac{\gamma_l}{z^2}\right)\left(1- \frac{\gamma_r}{w^2}\right)dzdw,\]
it follows that 
\begin{eqnarray}
\label{eqn_ContInt}
\lefteqn{\frac{1}{2}\sum^{\infty}_{k=1} k\beta^{k}(\varphi_{l})_{k}(\varphi_{r})_{k}=}\nonumber\\ 
&&\frac{2}{\pi} \oint\limits_{ \begin{array}{cc}
            |z|^2=\gamma_{l} \\ \Im{z}>0
            \end{array} } \oint\limits_{\begin{array}{cc}
            |w|^2=\gamma_{r} \\ \Im{w}>0
            \end{array} }\varphi'_{l}\left(z+\frac{\gamma_{l}}{z}\right)\varphi'_{r} \left( w+\frac{\gamma_{r}}{w}\right)\frac{1}{2\pi}\ln\left| \frac{\gamma_{lr}-zw}{\gamma_{lr}-z\bar{w}} \right| \left( 1-\frac{\gamma_{l}}{z^{2}}  \right) \left(1-\frac{\gamma_{r}}{w^{2}}\right)dzdw.\nonumber\\
\end{eqnarray}

Compare (\ref{eqn:Zprimefin}) to (\ref{eqn:vard}). Using (\ref{eqn_ContInt}), (\ref{eqn:Zsub}), (\ref{eqn:Zprime}) and (\ref{eqn:Z}), it follows that the covariance can be written as
\begin{eqnarray}
\lefteqn{\lim_{n \to \infty} \mathbf{Cov} \{  \mathcal{N}^{(l)}[\varphi_{l}], \ \mathcal{N}^{(r)}[\varphi_{r}] \}}\nonumber\\
 &=& \frac{\sigma^{2}}{4}\left( \varphi_{l}\right)_{1}\left(\varphi_{r}\right)_{1}(\frac{\gamma_{lr}}{\sqrt{\gamma_{l}\gamma_{r}}}) + \frac{1}{2}\sum^{\infty}_{k=2} k(\varphi_{l})_{k}(\varphi_{r})_{k} \left(\frac{\gamma_{lr}}{\sqrt{\gamma_{l}\gamma_{r}}}\right)^{k}\nonumber\\
& =&
 \frac{2}{\pi} \oint\limits_{ \begin{array}{cc}
            |z|^2=\gamma_{l} \\ \Im{ \ z}>0
            \end{array} } \oint\limits_{\begin{array}{cc}
            |w|^2=\gamma_{r} \\ \Im{\ w}>0
            \end{array} }\varphi'_{l}\left(z+\frac{\gamma_{l}}{z}\right)\varphi'_{r} \left( w+\frac{\gamma_{r}}{w}\right)\frac{1}{2\pi}log\left| \frac{\gamma_{lr}-zw}{\gamma_{lr}-z\bar{w}} \right| \left( 1-\frac{\gamma_{l}}{z^{2}}  \right)\nonumber\\
  &&\times \left(1-\frac{\gamma_{r}}{w^{2}}\right)dwdz + \ \frac{\gamma_{lr} (\sigma^{2} - 2)}{4 \pi^{2} \gamma_{l}\gamma_{r}}\int^{2\sqrt{\gamma_{l}}}_{-2\sqrt{\gamma_{l}}}  \frac{\lambda \varphi_{l}(\lambda)}{\sqrt{4\gamma_{l} - \lambda^{2}}}d\lambda\int^{2\sqrt{\gamma_{r}}}_{-2\sqrt{\gamma_{r}}}  \frac{\mu \varphi_{r}(\mu)}{\sqrt{4\gamma_{r} - \mu^{2}}}d\mu.\nonumber
\end{eqnarray}

\subsection{The Bilinear Form}
\label{subsec:bilin}
The main goal of this section is to prove Lemma \ref{lem:diag}, to which we now turn our attention. Begin with the following definition.
\begin{defn}
\label{def:bilin-finiten}
Let $M$ be a Wigner matrix satisfying (\ref{condentries}),  and let $P^{(l)}$, $P^{(l,r)}$  be the projection matrices defined in (\ref{def-proj}) and (\ref{def-proj1}). For polynomial functions $f,g: \R \to \R$, define
\begin{eqnarray}
\langle f, g  \rangle_{lr,n} &:=& \frac{1}{n}\sum_{j,k \in B_l \cap B_r}\E\left[ f(M^{(l)})_{jk}\cdot g(M^{(r)})_{kj}\right]\nonumber\\
&=& \frac{1}{n}\E \left[ \Tr\left\{ P^{(l)} f(M^{(l)})\cdot P^{(l,r)}\cdot g(M^{(r)})P^{(r)} \right\}\right].
\label{form:bilin}
\end{eqnarray}
\end{defn}

The large $n$ limit of $ \langle f, g  \rangle_{lr,n}$ exists for polynomial functions because all moments of the matrix entries of $M$ are finite. 
Then $ \lim_{n \to \infty} \langle f, g  \rangle_{lr,n} = \langle f, g  \rangle_{lr},$ where $\langle \cdot, \cdot \rangle_{lr}$ 
is the bilinear form defined in definition \ref{def:bilin}.  

We will compute the bilinear form $\langle f, g \rangle_{lr} $ for monomial functions  $ f(x) = x^{k}, \ g(x) = x^{q}$. 
We will also consider the random variables $n^{-1}\Tr \{P^{(l)}f(M^{(l)})P^{(l,r)} g(M^{(r)})P^{(r)}\},$ 
and prove their convergence almost surely to the non-random limit described in lemma \ref{lem:diag}. To this end, we will use some results and techniques 
from 
Free Probability. We refer the reader to \cite{cite_key1}  for the relevant background concerning noncommutative probability spaces, 
asymptotic freeness of Wigner matrices, as well as the definition and the 
properties of the multilinear free cumulant functionals $\kappa_{p},$ for $ p \geq 1$. 

Consider the matrices $ M, P^{(l)}, P^{(r)}$ as noncommutative random variables in the noncommutative probability spaces 
$ (Mat_{n}(\mathbb{C}), \E [\frac{1}{n} \Tr])$ and also $(Mat_n(\C), \frac{1}{n}\Tr\{\cdot\})$. Since $M$ is a Wigner random matrix and 
$\{P^{(l)}, P^{(r)}\}$ are deterministic Hermitian matrices, it follows from part $(i)$ of Theorem $5.4.5$ in \cite{cite_key1} 
that $M$ is asymptotically free from $\{ P^{(l)}, P^{(r)}\}$ with respect to the functional $n^{-1}\E\Tr(\cdot)$. In addition, it 
follows from part $(ii)$ of Theorem $5.4.5$ in \cite{cite_key1} that $M$ is almost surely asymptotically free from $\{P^{(l)},P^{(r)}\}$ 
with respect to the functional $n^{-1}\Tr(\cdot)$.  The collection
of all non-crossing partitions over a set with $p$ letters is denoted below by $NC(p)$. An important consequence of the asymptotic freeness of these 
matrices is that mixed free cumulants of $M$ and $ \{P^{(l)}, P^{(r)}\}$ vanish in the limit, with respect to both functionals, see Theorem 
$5.3.15$ of \cite{cite_key1}. Therefore, letting $\kappa_{\pi}$ denote a product of free cumulant functionals corresponding to the block structure 
of the partition $\pi$, it follows that
\begin{eqnarray}
\langle x^{k}, \ x^{q} \rangle_{lr} &=& \lim_{n \to \infty} \frac{1}{n}\E\left[\Tr\left\{ P^{(l)}\left(P^{(l)}MP^{(l)}\right)^{k}\left(P^{(r)}MP^{(r)}\right)^{q}P^{(r)}\right\} \right]\nonumber\\
&=& \lim_{n \to \infty} \frac{1}{n}\E\left[\Tr \{P^{(l)}MP^{(l)} \cdots P^{(l)}MP^{(l)}P^{(r)}MP^{(r)}\cdots P^{(r)}MP^{(r)}  \right]\nonumber\\
&=& \sum_{\pi \in NC(2(k+q)+1)} \kappa_{\pi}(P^{(l)}, M, P^{(l)}, \cdots, M, P^{(l)},P^{(r)},M, \cdots, M, P^{(r)})\nonumber\\
&=& \sum_{ \begin{array}{cc}
            \pi_{1} \in NC(odd),\pi_{2} \in NC(even) \\ \pi_{1} \cup \pi_{2} \in NC(2(k+q)+1)
            \end{array} } \kappa_{\pi_{2}}(M)\kappa_{\pi_{1}}(P^{(l)}, \cdots, P^{(l,r)}, \cdots, P^{(r)}),\nonumber\\
\end{eqnarray}
and also that almost surely
\begin{eqnarray}
\lefteqn{ \lim_{n \to \infty} \frac{1}{n}\Tr\left\{ P^{(l)}\left(P^{(l)}MP^{(l)}\right)^{k}\left(P^{(r)}MP^{(r)}\right)^{q}P^{(r)}\right\}}\nonumber\\
&=& \sum_{\pi \in NC(2(k+q)+1)} \kappa_{\pi}(P^{(l)}, M, P^{(l)}, \cdots, M, P^{(l)},P^{(r)},M, \cdots, M, P^{(r)})\nonumber\\
&=&  \sum_{ \begin{array}{cc}
            \pi_{1} \in NC(odd),\pi_{2} \in NC(even) \\ \pi_{1} \cup \pi_{2} \in NC(2(k+q)+1)
            \end{array} } \kappa_{\pi_{2}}(M) \ \kappa_{\pi_{1}}(P^{(l)}, \cdots, P^{(l,r)}, \cdots, P^{(r)}).\nonumber\\
\end{eqnarray}

Above $NC(odd)$, for example, denotes the set of non-crossing partitions on the odd integers in the indicated set. Since the calculation of the joint moments in each non-commutative probability space $(Mat_n(\C), n^{-1}\E\Tr)$ and $(Mat_n(\C), n^{-1}\Tr)$ is identical, we make no distinction between their free cumulants. Lets denote by $NCP(p)$ the set of all non-crossing partitions over $p$ letters which are also pair partitions. Recall that $NC(p)$ is a poset, the notion of partition refinement induces a partial order on $NC(p)$, which will be denoted by $ \pi \leq \sigma$ if, with $\pi, \sigma \in NC(p)$, each block of $\pi$ is contained within a block of $\sigma$. Now a notion of the complement of a partition will be developed.

\begin{defn}
 With $ \pi \in NC(p_{1})$, define the \emph{non-crossing complement}  $\pi^{c} \in NC(p_{2})$ to be the unique non-crossing partition on $p_{2}$ letters so that $\pi \cup \pi^{c} \in NC(p_{1}+p_{2})$, and $ \sigma \leq \pi^{c}$ for all other $ \sigma \in NC(p_{2})$ satisfying $ \pi \cup \sigma \in NC(p_{1}+p_{2})$.
\end{defn}

\indent
Since the limiting spectral distribution of $M$ is Wigner semicircle law with respect to the functional $n^{-1}\E\Tr$, and almost surely the Wigner semicircle law with respect to the functional $n^{-1}\Tr$, we have that $\kappa_{2}(M) = 1$ and  $ \kappa_{p}(M) = 0$ for $ p\neq 2$. It follows now that 
 \begin{equation}
  \langle x^{k}, \ x^{q} \rangle_{lr} = 0, \ \ \text{if} \ \ k+q \ \ \text{is odd},
 \end{equation} 
  and also that almost surely
\begin{equation}
 \lim_{n \to \infty} \frac{1}{n}\Tr\left\{\left(P^{(l)}MP^{(l)}\right)^{k}\left(P^{(r)}MP^{(r)}\right)^{q}\right\} = 0,  \ \ \text{if} \ \ k+q \ \ \text{is odd}.
\end{equation}  
  
Supposing then that $k+q$ is even,  and continuing the calculation,
\begin{eqnarray}
\langle x^{k}, \ x^{q} \rangle_{lr} &=& \sum_{\pi_{2} \in NCP( even)} \ \sum_{\begin{array}{cc}\pi_{1} \in NC(odd)  \\ \pi_{1} \cup \pi_{2} \in NC(2(k+q)+1)  \end{array}}\kappa_{\pi_{1}}(P^{(l)},  \cdots, P^{(r)})\nonumber\\
&=& \sum_{\pi_{2} \in NCP( k+q)} \ \sum_{\begin{array}{cc}\pi_{1} \in NC(k+q+1)  \\ \pi_{1} \leq \pi^{\text{c}}_{1}\end{array}}\kappa_{\pi_{1}}(P^{(l)},  \cdots, P^{(r)})\nonumber\\
&=& \sum_{\pi_{2} \in NCP( k+q)} \prod^{|\pi^{\text{c}}_{1}|}_{i = 1} \lim_{n \rightarrow \infty} \frac{1}{n}\E\Tr \{\prod_{P^{(j)} \in S_{i}} P^{(j)}   \},\nonumber\\ 
\end{eqnarray}
where $\ \pi^{\text{c}}_{1} = \{S_{1}, \cdots, S_{|\pi^{\text{c}}_{1}|} \} $ are the blocks of the non-crossing complement of a given partition. We have used the complement partitions to write the sum of the free cumulants over the partitions of the projection matrices into a product of joint moments of the projection matrices.

Similarly, with respect to the functional $n^{-1}\Tr$, we have that almost surely
\begin{eqnarray}
\lefteqn{ \lim_{n \to \infty} \frac{1}{n}\Tr\left\{\left(P^{(l)}MP^{(l)}\right)^{k}\left(P^{(r)}MP^{(r)}\right)^{q}\right\}}\nonumber\\
 &=& \sum_{\pi_{2} \in NCP( even)} \ \sum_{\begin{array}{cc}\pi_{1} \in NC(odd)  \\ \pi_{1} \cup \pi_{2} \in NC(2(k+q)+1)  \end{array}}\kappa_{\pi_{1}}(P^{(l)},  \cdots, P^{(r)})\nonumber\\
&=& \sum_{\pi_{2} \in NCP( k+q)} \ \sum_{\begin{array}{cc}\pi_{1} \in NC(k+q+1)  \\ \pi_{1} \leq \pi^{\text{c}}_{1}\end{array}}\kappa_{\pi_{1}}(P^{(l)},  \cdots, P^{(r)})\nonumber\\
&=& \sum_{\pi_{2} \in NCP( k+q)} \prod^{|\pi^{\text{c}}_{1}|}_{i = 1} \lim_{n \rightarrow \infty} \frac{1}{n}\Tr \{\prod_{P^{(j)} \in S_{i}} P^{(j)}   \}.
\end{eqnarray}
\indent
Recall that the non-crossing pair partitions are in bijection with Dyck paths, $ NCP(k+q) \to D_{(k+q)}.$ Thus the computation for each functional reduces to counting Dyck paths. The number of Dyck paths $(h(0), \cdots, h(k+q))$ with $h(k) = j$ is

\begin{equation}
\left[ \binom{k}{\frac{k+j}{2}} - \binom{k}{\frac{k+j+2}{2}} \right]\left[ \binom{q}{\frac{q+j}{2}} - \binom{q}{\frac{q+j+2}{2}} \right] = \frac{(j+1)^{2}}{(k+1)(q+1)}\binom{k+1}{\frac{k+j+2}{2}}\binom{q+1}{\frac{q+j+2}{2}}.\nonumber
\end{equation}

Note that $ \lim_{n \to \infty} n^{-1}\Tr\left(P^{(l)}\right)^a\left(P^{(r)}\right)^b = \gamma_{lr}$, \ for any $a,b \geq 1$. Also note that below the partition $\pi^{c}_{1}$ depends on the Dyck path $ d \in D_{(k+q)}$(which corresponds to some non-crossing pair partition). Also note that by $|\pi^{c}_{1}|$ we denote the number of blocks of $\pi^{c}_{1}$. Suppose for now that both $k,q$ are even integers.

\indent

The height of the path at $h(k)$ must be even, say $h(k) = 2j$. Those blocks which consist only of the matrices $P^{(l)}$ will contribute a factor of $\gamma_l$ to the product of joint moments. The number of blocks which contain only the matrices $P^{(l)}$ corresponds to the number of down edges of the path in the first $k$ steps. Denote by $u$ the number of up edges and $d$ the number of down edges of the path up to step $k$. Then $u+d = k$ and $u-d = 2j$, which implies that $d = k/2 -j$. The number of blocks which contain only the matrices $P^{(r)}$ is equal to the number of up edges of the path in the final $q$ steps. This number corresponds to the exponent on the factor $\gamma_r$ in the product of joint moments. Denote now by $u$ the number of up edges and $d$ the number of down edges of the path in the final $q$ steps. The $u+d = q$ and $d-u = 2j$, which implies that $u = q/2-j$. The remaining blocks of the partition contain projection matrices of mixed type and will contribute a factor $\gamma_{lr}$ to the product of joint moments. Since the total number of blocks in the partition is $\frac{k+q}{2}+1$, the number of factors of $\gamma_{lr}$ in the product of joint moments is $2j+1$. Partitioning the Dyck paths into equivalence classes based on the height $h(k)$, we get that
\begin{eqnarray}
 \langle x^{k}, x^{q}  \rangle_{lr} &=&  \sum_{d \in D_{(k+q)}} \prod^{|\pi^{\text{c}}_{1}|}_{i=1} \lim_{n \rightarrow \infty} \frac{1}{n}\E\Tr \{\prod_{P^{(j)} \in S_{i}} P^{(j)} \}\nonumber\\ 
 &=& \sum^{\frac{k}{2}}_{j=0} \ \sum_{\begin{array}{cc} 
d \in D_{(k+q)} \\ h(k) = 2j
\end{array}} \gamma_{l}^{\frac{k}{2}-j} \gamma^{\frac{q}{2}-j}_{r} \gamma^{2j+1}_{lr}\nonumber\\
&=&\sum^{\frac{k}{2}}_{j=0} \ \frac{(2j+1)^{2}}{(k+1)(q+1)} 
  \binom{k+1}{\frac{k+2j+2}{2}} \binom{q+1}{\frac{q+2j+2}{2}} \gamma_{l}^{\frac{k}{2}-j} \gamma^{\frac{q}{2}-j}_{r} \gamma^{2j+1}_{lr},\nonumber
 \end{eqnarray}
and also, almost surely,
\begin{eqnarray}
\lefteqn{\lim_{n \to \infty} \frac{1}{n}\Tr\left\{\left(P^{(l)}MP^{(l)}\right)^{k}\left(P^{(r)}MP^{(r)}\right)^{q}\right\}}\nonumber\\
 &=&  \sum_{d \in D_{(k+q)}} \prod^{|\pi^{\text{c}}_{1}|}_{i=1} \lim_{n \rightarrow \infty} \frac{1}{n}\Tr \{\prod_{P^{(j)} \in S_{i}} P^{(j)} \}\nonumber\\ 
 &=& \sum^{\frac{k}{2}}_{j=0} \ \sum_{\begin{array}{cc} 
d \in D_{(k+q)} \\ h(k) = 2j
\end{array}} \gamma_{l}^{\frac{k}{2}-j} \gamma^{\frac{q}{2}-j}_{r} \gamma^{2j+1}_{lr}\nonumber\\
&=&\sum^{\frac{k}{2}}_{j=0} \ \frac{(2j+1)^{2}}{(k+1)(q+1)} 
  \binom{k+1}{\frac{k+2j+2}{2}} \binom{q+1}{\frac{q+2j+2}{2}} \gamma_{l}^{\frac{k}{2}-j} \gamma^{\frac{q}{2}-j}_{r} \gamma^{2j+1}_{lr}.\nonumber
 \end{eqnarray}

Now suppose that both $k,q$ are odd. The height of the path at $h(k)$ must be odd, say $h(k) = 2j+1$. Similar to the even case, the number of blocks which consist only of the matrices $P^{(l)}$ equals the exponent of $\gamma_l$ in the product of joint moments. The number of blocks which contain only the matrices $P^{(l)}$ corresponds to the number of down edges of the path in the first $k$ steps. Denote by $u$ the number of up edges and $d$ the number of down edges of the path up to step $k$. Then $u+d = k$ and $u-d = 2j+1$, which implies that $d = (k-1)/2 -j$. The number of blocks which contain only the matrices $P^{(r)}$ is equal to the number of up edges of the path in the final $q$ steps. This number corresponds to the exponent on the factor $\gamma_r$ in the product of joint moments. Denote now by $u$ the number of up edges and $d$ the number of down edges of the path in the final $q$ steps. The $u+d = q$ and $d-u = 2j+1$, which implies that $u = (q-1)/2-j$. The remaining blocks of the partition contain projection matrices of mixed type and will contribute a factor of $\gamma_{lr}$ to the product of joint moments. Since the total number of blocks in the partition is $\frac{k+q}{2}+1$, the number of factors of $\gamma_{lr}$ in the product of joint moments is $2j+2$. Partitioning the Dyck paths into equivalence classes based on the height $h(k)$, we get that
\begin{eqnarray}
 \langle x^{k}, x^{q}  \rangle_{lr} &=&  \sum_{d \in D_{(k+q)}} \prod^{|\pi^{\text{c}}_{1}|}_{i=1} \lim_{n \rightarrow \infty} \E\frac{1}{n}\Tr \{\prod_{P^{(j)} \in S_{i}} P^{(j)} \}\nonumber\\ 
 &=& \sum^{ \frac{k-1}{2} }_{j=0} \ \sum_{\begin{array}{cc} 
d \in D_{(k+q)} \\ h(k) = 2j+1
\end{array}} \gamma_{l}^{\frac{k-1}{2}-j} \gamma^{ \frac{q-1}{2}-j}_{r} \gamma^{2j+2}_{lr}\nonumber\\
&=&\sum^{ \frac{k-1}{2}}_{j=0} \ \frac{(2j+2)^{2}}{(k+1)(q+1)} 
  \binom{k+1}{\frac{k+2j+3}{2}} \binom{q+1}{\frac{q+2j+3}{2}} \gamma_{l}^{\frac{k-1}{2}-j} \gamma^{\frac{q-1}{2}-j}_{r} \gamma^{2j+2}_{lr},\nonumber
 \end{eqnarray}
and also, almost surely,
\begin{eqnarray}
\lefteqn{ \lim_{n \to \infty} \frac{1}{n}\Tr\left\{\left(P^{(l)}MP^{(l)}\right)^{k}\left(P^{(r)}MP^{(r)}\right)^{q}\right\}}\nonumber\\
 &=&  \sum_{d \in D_{(k+q)}} \prod^{|\pi^{\text{c}}_{1}|}_{i=1} \lim_{n \rightarrow \infty}\frac{1}{n}\Tr \{\prod_{P^{(j)} \in S_{i}} P^{(j)} \}\nonumber\\ 
 &=& \sum^{ \frac{k-1}{2} }_{j=0} \ \sum_{\begin{array}{cc} 
d \in D_{(k+q)} \\ h(k) = 2j+1
\end{array}} \gamma_{l}^{\frac{k-1}{2}-j} \gamma^{ \frac{q-1}{2}-j}_{r} \gamma^{2j+2}_{lr}\nonumber\\
&=&\sum^{ \frac{k-1}{2}}_{j=0} \ \frac{(2j+2)^{2}}{(k+1)(q+1)} 
  \binom{k+1}{\frac{k+2j+3}{2}} \binom{q+1}{\frac{q+2j+3}{2}} \gamma_{l}^{\frac{k-1}{2}-j} \gamma^{\frac{q-1}{2}-j}_{r} \gamma^{2j+2}_{lr}.\nonumber
 \end{eqnarray}
Now for polynomials $f(x) = \sum^{p}_{i=0} a_{i}x^{i}$ and $ g(x) = \sum^{m}_{j=0}b_{j}x^{j},$ we have by linearity that
\begin{equation}
\langle f, g \rangle_{lr} = \sum^{p}_{i=0}\sum^{m}_{j=0}a_{i}b_{j}\langle x^{i}, x^{j} \rangle_{lr}.
\label{neednow}
\end{equation}
The intersection of countably many events, each with probability $1$, occurs with probability $1$. There are only countably many polynomials with rational coefficients, so we have proved that the random variables
\[ \frac{1}{n} \Tr \{P^{(l)}f(M^{(l)})P^{(l,r)} g(M^{(r)})P^{(r)}\},\]
converge almost surely to the same, non-random limit given by the right hand side of (\ref{neednow}), whenever $f,g$ are polynomials with rational coefficients.

\indent
The bilinear form $ \langle f, g \rangle_{lr}$ is diagonalized in the next proposition. 
\begin{prop}
\label{prop_diag}
The two families $\{U^{\gamma_{l}}_{k}\}^{\infty}_{k=0}$ and $\{U^{\gamma_{r}}_{q}\}^{\infty}_{q=0}$ of rescaled Chebyshev polynomials of the second kind are biorthogonal with respect to the bilinear form (\ref{form:bilin}). More precisely,
\begin{equation}
\frac{1}{\sqrt{\gamma_{l}\gamma_{r}}} \ \langle U^{\gamma_{l}}_{k}, U^{\gamma_{r}}_{q} \rangle_{lr} \ =  \ \delta_{kq} \ (\frac{\gamma_{lr}}{\sqrt{\gamma_{l}\gamma_{r}}})^{k+1}.
\end{equation}
\end{prop}
The Proposition \ref{prop_diag} is proven in the Appendix 2.
\begin{rem}
Previously we have shown that whenever $f,g$ are polynomials with rational coefficients, almost surely (a.s.)
\[\lim_{n \rightarrow \infty} \frac{1}{n} \Tr\left\{ P^{(l)} f(M^{(l)})\cdot P^{(l,r)}\cdot g(M^{(r)})P^{(r)} \right\} = \langle f, g \rangle_{lr}.\]

The Chebyshev polynomials have rational coefficients, so it follows from the above argument that a.s.

\begin{equation}
\label{cheby-lim-as}
\frac{1}{\sqrt{\gamma_{l}\gamma_{r}}} \lim_{n \to \infty}\frac{1}{n}\Tr\{ P^{(l)} U^{\gamma_{l}}_k(M^{(l)})P^{(l,r)}U^{\gamma_{r}}_q(M^{(r)})P^{(r)}\}  =   \delta_{kq} \ (\frac{\gamma_{lr}}{\sqrt{\gamma_{l}\gamma_{r}}})^{k+1}.
\end{equation}
\end{rem}

Now the bilinear form $ \langle \cdot, \cdot \rangle_{lr}$ will be extended to functions other than polynomials. For this part of the argument, the 
bound on the variance of linear eigenvalue statistics in \ref{ineq-MS} is essential.
\begin{prop}
\label{lem:exten}
Let $ f,g \in \mathcal{H}_s$  for some $ s > \frac{3}{2} $, i.e. for some $ \epsilon > 0$,
\begin{equation}
\int^{\infty}_{-\infty} |\widehat{f}(t)|^{2}\left(1 + |t|\right)^{3+\epsilon} dt < \infty, \ \int^{\infty}_{-\infty} |\widehat{g}(t)|^{2}\left(1 + |t|\right)^{3+\epsilon} dt < \infty.
\end{equation}
Then the limit of $\langle f, g \rangle_{lr,n}$ (see definition \ref{def:bilin-finiten}) as $n \to \infty$ exists and
\begin{equation}
\langle f, g \rangle_{lr}\  =  \frac{1}{4\pi^{2}\gamma_{l}\gamma_{r}} \int^{2\sqrt{\gamma_{l}}}_{-2\sqrt{\gamma_{l}}} \int^{2\sqrt{\gamma_{r}}}_{-2\sqrt{\gamma_{r}}} f(x)g(y)F_{lr}(x,y) \sqrt{4\gamma_{l}-x^{2}}\sqrt{4\gamma_{r}-y^{2}}dydx,
\end{equation}
and also, almost surely,
\begin{eqnarray}
\lefteqn{\lim_{n \rightarrow \infty} \frac{1}{n} \Tr\left\{ P^{(l)} f(M^{(l)})\cdot P^{(l,r)}\cdot g(M^{(r)})P^{(r)} \right\}}\nonumber\\
&=& \frac{1}{4\pi^{2}\gamma_{l}\gamma_{r}} \int^{2\sqrt{\gamma_{l}}}_{-2\sqrt{\gamma_{l}}} \int^{2\sqrt{\gamma_{r}}}_{-2\sqrt{\gamma_{r}}} f(x)g(y)F_{lr}(x,y)\sqrt{4\gamma_{l}-x^{2}}\sqrt{4\gamma_{r}-y^{2}}dydx,\nonumber\\
\end{eqnarray}
where the kernel $F_{lr}(x,y)$ is given by (\ref{eqn:kernel}).
\end{prop}
The Proposition \ref{lem:exten} is proven in ther Appendix 3.
Lemma \ref{lem:diag} now follows from Propositions \ref{prop_diag} and \ref{lem:exten}. This also completes the proof of Theorem \ref{thm-GaussCLT}.
\section{Proof of Theorem \ref{thm-NG}}
It is enough to prove the case of $d=2$, i.e. the limiting covariance of $\mathcal{N}_n^{(1)\circ}[\varphi_{1}]$ and 
$\mathcal{N}_n^{(2)\circ}[\varphi_{2}]$. Let $U(t),\widetilde U(t), u_n(t),\widetilde u_n(t)$ be $U^{(1)}(t),U^{(2)}(t), u_n^{(1)}(t), u_n^{(2)}(t)$ 
defined in (\ref{ujk}-\ref{moskva}) respectively.
$U(t)$ and $\widetilde U(t)$ are unitrary matrices and
\[U(t)U^*(t)=\widetilde U(t)\widetilde U^*(t)=I,\ \ \ \ \ \ \ |U_{jk}|\leq 1,\ \ \ \ \ \ \ \sum_{k=1}^n|U_{jk}|^2=1.\]
By Remark 3.3 in \cite{cite_key3}, we have the following bounds 
\begin{eqnarray}
&&\mathbf{Var}\{u_n(t)\}\leq C(\sigma_6)(1+|t|^3)^2 ,\\
&&\mathbf{Var}\{\widetilde u_n(t)\}\leq C(\sigma_6)(1+|t|^3)^2 ,\\
&&\mathbf{Var}\{\mathcal N^{(1)\circ}_n(t)\}\leq C(\sigma_6)\left(\int_{-\infty}^{\infty}(1+|t|^3)|\widehat\varphi_1(t)|\ dt\right)^2 ,\label{var_N_1}\\
&&\mathbf{Var}\{\mathcal N_n^{(2)\circ}(t)\}\leq C(\sigma_6)\left(\int_{-\infty}^{\infty}(1+|t|^3)|\widehat\varphi_2(t)|\ dt\right)^2 .\label{var_N_2}
\end{eqnarray}
Let $w$ be a linear combination of random variables $\mathcal{N}_n^{(1)\circ}[\varphi_{1}]$ and $\mathcal{N}_n^{(2)\circ}[\varphi_{2}]$, and $Z_n(x)$ be the characteristic function of $w$, i.e.
\begin{equation}
w=\alpha\mathcal{N}_{n}^{(1)\circ}[\varphi_{1}]+\beta\mathcal{N}_{n}^{(2)\circ}[\varphi_{2}],\ \ Z_n(x)=\E\{e^{ixw}\}.
\end{equation} 
We note that 
\begin{equation}
Z_n(x)=1+\int_0^x Z_n'(t)dt;\  Z'_n(x)=i\E\{we^{ixw}\},
\end{equation}

By the Cauchy-Schwarz inequality and (\ref{var_N_1}-\ref{var_N_2}) we get
\begin{equation}
|Z_n'(x)|\leq(|\alpha|+|\beta|)C^{1/2}(\sigma_6)\int_{-\infty}^{\infty}(1+|t|^3)|\widehat\varphi_1(t)|\ dt,
\end{equation}

Using the Fourier inversion formula $f(\lambda)=\int e^{it\lambda}\widehat{f}(t)\ dt$ we obtain
\begin{equation}
\mathcal{N}_n^{(1)\circ}[\varphi_{1}]=\int_{-\infty}^{\infty}\widehat{\varphi_1}(t)u_n^{\circ}(t)dt,\ \ \ \mathcal{N}_n^{(2)\circ}[\varphi_{2}]=\int_{-\infty}^{\infty}\widehat{\varphi_2}(t)\widetilde{u}_n^{\circ}(t)dt.
\end{equation}
Therefore,
\begin{eqnarray}
&&w=\int_{-\infty}^{\infty}\alpha\widehat{\varphi_1}(t)u_n^{\circ}(t)+\beta \widehat{\varphi_2}(t)\widetilde{u}_n^{\circ}(t)dt,\\
&&Z'_n(x)=i\alpha\int_{-\infty}^{\infty}\widehat{\varphi_1}(t)Y_n(x,t)dt+i\beta\int_{-\infty}^{\infty}\widehat{\varphi_2}(t)\widetilde Y_n(x,t)dt,
\end{eqnarray}
where
\begin{equation}
Y_n(x,t)=\E\{u_n^{\circ}(t)e_n(x)\},\ \ \ \widetilde Y_n(x,t)=\E\{\widetilde u_n^{\circ}(t)e_n(x)\},\ \ \ e_n(x)=e^{ixw}.
\end{equation}
By the Cauchy-Schwarz inequality,
\begin{eqnarray}
|Y_n(x,t)|\leq \E\{|u_n^{\circ}(t)|\}\leq C^{1/2}(\sigma_6)(1+|t|^3),\label{boundY_n}\\
|\widetilde Y_n(x,t)|\leq \E\{|\widetilde u_n^{\circ}(t)|\}\leq C^{1/2}(\sigma_6)(1+|t|^3)\label{boundtildeY_n},
\end{eqnarray}
and
\begin{eqnarray}
|\frac{\partial}{\partial x}Y_n(x,t)|&=&|\E\{\alpha u_n^{\circ}\mathcal N_n^{(1)\circ}[\varphi_1]e_n(x)+\beta u_n^{\circ}\mathcal N_n^{(2)\circ}[\varphi_2]e_n(x)\}|\nonumber\\
&\leq& C(\sigma_6)(1+|t|^3)\int_{-\infty}^{\infty}(1+|t|^3)(|\alpha\widehat\varphi_1(t)|+|\beta\widehat\varphi_2(t)|)dt.
\end{eqnarray}
Also
\begin{equation}
\frac{\partial}{\partial t} Y_n(x,t)=\mathbb E\{u_n'(t)e_n^{\circ}(x)\}=\frac{i}{\sqrt n}\sum_{j,k\in B_1}\mathbb E\{W_{jk}\Phi_n\},
\end{equation}
where
\[\Phi_n=U_{jk}(t)e_n^{\circ}(x).\]
Recall that for $D_{jk}=\partial/\partial M_{jk},\ \ \beta_{jk}=(1+\delta_{jk})^{-1},$ 
\begin{eqnarray}
D_{jk}U_{ab}(t)=1_{j,k\in B_1}i\beta_{jk}[U_{aj}*U_{bk}(t)+U_{bj}*U_{ak}(t)]\label{DU},\\
D_{jk}\widetilde U_{ab}(t)=1_{j,k\in B_2}i\beta_{jk}[\widetilde U_{aj}*\widetilde U_{bk}(t)+\widetilde U_{bj}*\widetilde U_{ak}(t)],\label{DtildeU}
\end{eqnarray}
and
\begin{eqnarray}
D_{jk}e_n(x)&=&2i\beta_{jk}xe_n(x)(1_{j,k\in B_1}\alpha (\varphi_1)'_{jk}(M_1)+1_{j,k\in B_2}\beta (\varphi_2)'_{jk}(M_2))\\
&=&-2\beta_{jk}xe_n(x)\left(1_{j,k\in B_1}\int_{-\infty}^{\infty} t U_{jk}(t)\alpha\widehat{\varphi_1}(t)dt+1_{j,k\in B_2}\int_{-\infty}^{\infty} t\widetilde U_{jk}\beta\widehat{\varphi_2}(t)dt\right).
\label{Den}
\end{eqnarray}
\begin{lem}
\label{lemforDl}
Let $\varphi_1,\varphi_2$ have fourth bounded derivatives. Then
\begin{equation}
|D_{jk}^l(U_{jk}(t)e_n^{\circ}(x))|\leq C_l(x,t),\ \ \ 0\leq l\leq 5,
\end{equation}
where $C_l(x,t)$ is a degree $l$ polynomial of $|x|,|t|$ with positive coefficients. 
\end{lem}
\begin{proof}
From (\ref{DU}) and (\ref{DtildeU}), we have
\begin{equation}
|D_{jk}^lU_{ab}(t)|,|D_{jk}^l\widetilde U_{ab}(t)|\leq Const_l|t|^l, \ \ \ 0\leq l\leq 5.
\end{equation}
(\ref{Den}) implies
\begin{equation}
|D_{jk}^le_n(x)|\leq Const'_l(1+|x|^l) \ \ \ 0\leq l\leq 5.
\end{equation}
These two inequalities complete the proof of Lemma \ref{lemforDl}
\end{proof}
We now apply the Decoupling Formula (\ref{decf}) with $p=2$ to obtain
\begin{eqnarray}
\frac{\partial}{\partial t} Y_n(x,t)&=&\frac{i}{ n}\sum_{j,k\in B_1} (1+(\sigma^2-1)\delta_{jk})\E\{D_{jk}\Phi_n\}+O(1)\nonumber\\
&=&\frac{i}{n}\sum_{j,k\in B_1} (1+\delta_{jk})\E\{D_{jk}\Phi_n\}+\frac{i(\sigma^2-2)}{n}\sum_{j\in B_1} \E\{D_{jj}\Phi_n\}+O(1).\label{partialtY_n}
\end{eqnarray}
where the error term is bounded by $C_3(x,t)$ as $n\to\infty$.
The first term in (\ref{partialtY_n}) is
\begin{eqnarray*}
&&-\frac{t}{n}Y_n(t,x)-\frac{1}{n}\int_0^t\E\{u_n(t-t_1)\}Y_n(x,t_1)dt_1-\frac{1}{n}\E\{\int_0^tu_n(t_1)u^{\circ}_n(t-t_1)dt_1e_n^{\circ}(x)\}\\
&&-\frac{2i}{n}\E\{xe_n(x)\left(\int_{-\infty}^{\infty} t_1 u_n(t+t_1)\alpha\widehat{\varphi_1}(t_1)dt_1+\int_{-\infty}^{\infty} t_1TrP^{(1,2)}U_n(t)P^{(1,2)}\widetilde U_{n}(t_1)\beta\widehat{\varphi_2}(t_1)dt_1\right).
\end{eqnarray*}
The first term and the second term are bounded because of (\ref{boundY_n}). The last term is bounded by 
$$2|x|\int_{-\infty}^{\infty}|t|(|\alpha||\widehat\varphi_1(t_1)|+|\beta||\widehat\varphi_2(t_1)|)dt_1,$$
and the third term is bounded by $2|t|C^{1/2}(\sigma_6)(1+|t|^3)$.

The second term in (\ref{partialtY_n}) is
\begin{eqnarray*}
&&\frac{2-\sigma^2}{n}\sum_{j\in B_1} \E\left\{\int_0^tU_{jj}(t_1)U_{jj}(t-t_1)dt_1e_n^{\circ}(x)\right\}+\frac{ix(2-\sigma^2)}{n}\sum_{j\in B_1} \E\left\{e_n(x)\int_{-\infty}^{\infty} t_1 U_{jj}(t)U_{jj}(t_1)\alpha\widehat{\varphi_1}(t_1)dt_1\right\}\\
&&+\frac{ix(2-\sigma^2)}{n}\sum_{j\in B_1\cap B_2}\E\left\{e_n(x)\int_{-\infty}^{\infty} t_1U_{jj}(t)\widetilde U_{jj}(t_1)\beta\widehat{\varphi_2}(t_1)dt_1\right\}
\end{eqnarray*}
The first term is bounded by $2|2-\sigma^2||t|$, and the second term is bounded by 
$$2|x||2-\sigma^2|\int_{-\infty}^{\infty}|t|(|\alpha||\widehat\varphi_1(t_1)|+|\beta||\widehat\varphi_2(t_1)|)dt_1.$$ So 
\[\left|\frac{\partial}{\partial t} Y_n(x,t)\right|\leq C_5(x,t).\]
By symmetry, $\widetilde Y_n(x,t)$ has similar bounds. Therefore,
we conclude that the sequences $\{Y_n\},\{\widetilde Y_n\}$ are bounded and equicontinuous on any finite subset of $\mathbb R^2$. 
We will prove now that any uniformly converging subsequence of $\{Y_n\}(\{\widetilde Y_n\})$ has same limit $Y(\widetilde Y)$.

We deal with $Y_n$ first, and by the symmetric property, we can find $\widetilde Y_n$. We use the identity
\begin{equation}
u_n(t)=n_1+i\int_0^t \sum_{j,k\in B_1}M_{jk}U_{jk}(t_1)dt_1,
\end{equation}
to write
\begin{equation}
\label{Yn}
Y_n(x,t)=\frac{i}{\sqrt {n}}\int_0^t\sum_{j,k\in B_1}\mathbb E\{W_{jk} U_{jk}(t_1)e_n^{\circ}(x)\}dt_1.
\end{equation}

By applying decoupling formula (\ref{decf}) with $p=3$ to (\ref{Yn}), we have
\begin{equation}
Y_n(x,t)=\frac{i}{\sqrt {n}}\int_0^t\sum_{j,k\in B_1}\left[\sum_{l=0}^3\frac{\kappa_{l+1,jk}}{n^{l/2}l!}
\mathbb E\{D^l_{jk}(U_{jk}(t_1)e_n^{\circ}(x))\}+\varepsilon_{3,jk}\right]dt_1,
\end{equation}
where 
\begin{equation}
\kappa_{1,jk}=0,\kappa_{2,jk}=1+\delta_{jk}(\sigma^2-1),
\end{equation}
\begin{equation}
\kappa_{3,jk}=\mu_3,\kappa_{4,jk}=\kappa_4, j\neq k, 
\end{equation}
and $\kappa_{3,jj},\kappa_{4,jj}$ are uniformly bounded, i.e. there exist constants $\sigma_3,\sigma_4$ such that
\begin{equation}
\label{kappa34}
|\kappa_{3,jj}|\leq\sigma_3, |\kappa_{4,jj}|\leq\sigma_4 ,
\end{equation}
and
\begin{equation}
|\varepsilon_{3,jk}|\leq n^{-2}C_3\E\{|W_{jk}|^5\}\sup_{t\in\mathbb R}\left|D_{jk}^4\Phi_n(x)\right|\leq n^{-2}C_4(x,t).
\end{equation}

Let
\begin{eqnarray}
&&T_l=\frac{i}{n^{(l+1)/2}}\int_0^t\sum_{j,k\in B_1}\frac{\kappa_{l+1,jk}}{l!}\mathbb E\{D^l_{jk}(U_{jk}(t_1)e_n^{\circ}(x))\}dt_1,l=1,2,3,\\
&&\mathcal E_n=\frac{i}{\sqrt n}\int_0^t\sum_{j,k\in B_1}\varepsilon_{3,jk}dt.
\end{eqnarray}
Then
\[Y_n(x,t)=T_1+T_2+T_3+\mathcal E_n,\]
and
\[|\mathcal E_n|\leq\frac{n_1^2}{n^{5/2}}C_5(x,t)\to0,\ \ \ as\ n\to\infty.\]

We note that if $W_{jk}$'s are Gaussian, then $Y_n(x,t)=T_1$. Thus, $T_1$ coincide with the $Y_n$ in Theorem \ref{thm-GaussCLT}.

Let 
\[\bar{v}_n(t)=n^{-1}\E\{u_n(t)\},\ \ \ \bar{\widetilde v}_n(t)=n^{-1}\E\{\widetilde u_n(t)\}.\]
Then
\begin{equation}
Y_n(x,t)+2\int_0^tdt_1\int_0^{t_1}\bar{v}_n(t_1-t_2)Y_n(x,t_2)dt_2=xZ_n(x)A_n(t)+r_n(x,t)+T_2+T_3+\mathcal E_n,
\end{equation}
where 
\begin{eqnarray}
A_n(t)=-\frac{2\alpha}{n}\int_0^t\mathbb E\{TrU(t_1)P_1\varphi_1'(M_1)P_1\}dt_1-\frac{2\beta}{n}\int_0^t\mathbb E\{TrU(t_1)P_2\varphi_2'(M_2)P_{2}\}dt_1,
\end{eqnarray}
and $r_n(x,t)\to0$ on any bounded subset of $\{(x,t):x\in \R,t>0\}$.

Let $A(t)=\lim_{n\to\infty}A_n(t)$. It follows from the proof of Theorem \ref{thm-GaussCLT} that $A(t)$ coincides with the one established in the 
Gaussian case. 

\begin{prop} 
\label{prop_T2}
$T_2\to0$ on any bounded subset of $\{(x,t):x\in\R,t>0\}$.
\end{prop}

\begin{proof}
The second derivative (l=2) is
\begin{eqnarray*}
\lefteqn{D^2_{jk}(U_{jk}(t_1)e_n^{\circ}(x))=\beta_{jk}^2\times}\\
&&\{-(6U_{jj}*U_{jk}*U_{kk}+2U_{jk}*U_{jk}*U_{jk})(t_1)e_n^{\circ}(x)\\
&&-4i(U_{jj}*U_{kk}+U_{jk}*U_{jk})(t_1)xe_n(x)\left[\int_{-\infty}^{\infty} t U_{jk}(t)\alpha\widehat{\varphi_1}(t)dt+1_{j,k\in B_2}\int_{-\infty}^{\infty} t\widetilde U_{jk}\beta\widehat{\varphi_2}(t)dt\right]\\
&&+4U_{jk}(t_1)x^2e_n(x)\left[\int_{-\infty}^{\infty} t U_{jk}(t)\alpha\widehat{\varphi_1}(t)dt+1_{j,k\in B_2}\int_{-\infty}^{\infty} t\widetilde U_{jk}\beta\widehat{\varphi_2}(t)dt\right]^2\\
&&-2iU_{jk}(t_1)xe_n(x)\left[\int_{-\infty}^{\infty} t(U_{jj}*U_{kk}+U_{jk}*U_{jk})(t)\alpha\widehat{\varphi_1}(t)dt\right.\\
&&\left.+1_{j,k\in B_2}\int_{-\infty}^{\infty} t(\widetilde U_{jj}*\widetilde U_{kk}+\widetilde U_{jk}*\widetilde U_{jk})(t)\beta\widehat {\varphi_2}(t)dt\right]\}.
\end{eqnarray*}
Let
\begin{eqnarray*}
T_{21}&=&\frac{i\kappa_3}{2n^{3/2}}\int_0^t\E\left\{\sum_{j,k\in B_1}-\beta_{jk}^2(6U_{jj}*U_{jk}*U_{kk}+2U_{jk}*U_{jk}*U_{jk})(t_1)e_n^{\circ}(x)\right.
\\&&-4i\beta^2_{jk}(U_{jj}*U_{kk}+U_{jk}*U_{jk})(t_1)xe_n(x)\int t_2 U_{jk}(t_2)\alpha\widehat{\varphi_1}(t_2)dt_2
\\&&+4\beta^2_{jk}U_{jk}(t_1)x^2e_n(x)(\int t_2 U_{jk}(t_2)\alpha\widehat{\varphi_1}(t_2)dt_2)^2
\\&&\left.-2i\beta_{jk}^2U_{jk}(t_1)xe_n(x)\int t_2(U_{jj}*U_{kk}+U_{jk}*U_{jk})(t_2)\beta\widehat{\varphi_1}(t_2)dt_2\right\}dt_1,
\end{eqnarray*}
\begin{eqnarray*}
T_{22}&=&\frac{i\kappa_3}{2n^{3/2}}\int_0^t\E\left\{\sum_{j,k\in B_1\cap B_2}4\beta^2_{jk}U_{jk}(t_1)x^2e_n(x)(\int t_2\widetilde U_{jk}(t_2)\beta\widehat{\varphi_2}(t_2)dt_2)^2\right.
\\&&+8\beta^2_{jk}U_{jk}(t_1)x^2e_n(x)\int t_2 U_{jk}(t_2)\alpha\widehat{\varphi_1}(t_2)dt_2\int t_3\widetilde U_{jk}(t_3)\beta\widehat{\varphi_2}(t_3)dt_3
\\&&\left.-2i\beta_{jk}^2U_{jk}(t_1)xe_n(x)\int t_2[\widetilde U_{jj}*\widetilde U_{kk}+
\widetilde U_{jk}*\widetilde U_{jk}](t_2)\beta\widehat {\varphi_2} (t_2)dt_2\right\}dt_1,
\end{eqnarray*}
\begin{eqnarray*}
T_{23}&=&\frac{i}{2n^{3/2}}\int_0^t\sum_{j\in B_1}\kappa_{3,jj}\mathbb E\{D^2_{jj}(U_{jj}(t_1)e_n^{\circ}(x))\}dt_1.
\end{eqnarray*}
Then $T_2=T_{21}+T_{22}+T_{23}$.
It has been shown in \cite{cite_key3} that $|T_{21}|\leq |t|C_2(x,t)n_1/n^{3/2}$ on any bounded subset of $\{(x,t):x\in\R,t>0\}$. 
Also, by Proposition \ref{lemforDl} and (\ref{kappa34}), one has $|T_{23}|\leq |t|C_2(x,t)n_1/n^{3/2}$. 

In $T_{22}$, there are three types of a sum,
\begin{eqnarray*}
S_1=n^{-3/2}\sum_{j,k\in B_1\cap B_2}U_{jk}(t_1)\widetilde U_{jk}(t_2)\widetilde U_{jk}(t_3),\\
S_2=n^{-3/2}\sum_{j,k\in B_1\cap B_2}U_{jk}(t_1)U_{jk}(t_2)\widetilde U_{jk}(t_3),\\
S_3=n^{-3/2}\sum_{j,k\in B_1\cap B_2}U_{jk}(t_1)\widetilde U_{jj}(t_2)\widetilde U_{kk}(t_3).
\end{eqnarray*}
Applying the Cauchy-Schwarz inequality we obtain
\begin{eqnarray*}
|S_1|\leq n^{-3/2}\sum_{j,k\in B_2}|\widetilde U_{jk}(t_2)\widetilde U_{jk}(t_3)|\leq \frac{n_2}{n^{3/2}},\\
|S_2|\leq n^{-3/2}\sum_{j,k\in B_1}|U_{jk}(t_1) U_{jk}(t_2)|\leq \frac{n_1}{n^{3/2}}.
\end{eqnarray*}
Writing
\[S_3=\frac{n_{12}}{n^{3/2}}(P_{12}U(t_1)P_{12}V(t_2),V(t_3)),\]
where
\[V(t)=n_{12}^{-1/2}(\widetilde U_{jj}(t))_{j\in B_1\cap B_2}^T.\]
$\|V(t)\|\leq 1$, $\|P_{12}U(t)P_{12}\|\leq1$, we conclude that $S_{3}\leq\frac{n_{12}}{n^{3/2}}$, hence $T_{22}\leq |t|/n^{3/2}$. 
This completes the proof of Proposition \ref{prop_T2}.
\end{proof}

\begin{prop} 
\[T_3=T_{31}+T_{32}+R_3(x,t),\]
where
\begin{eqnarray*}
T_{31}=\frac{i\kappa_4}{n^2}\int_0^t\sum_{j,k\in B_1}\mathbb E\left\{U_{jj}*U_{kk}(t_1)xe_n(x)\int t_2 U_{jj}*U_{kk}(t_2)\alpha\widehat{\varphi_1}(t_2)dt_2\right\}dt_1,\\
T_{32}=\frac{i\kappa_4}{n^2}\int_0^t\sum_{j,k\in B_1\cap B_2}\mathbb E\left\{U_{jj}*U_{kk}(t_1)xe_n(x)\int t_2\widetilde U_{jj}*\widetilde U_{kk}(t_2)\beta\widehat{\varphi_2}(t_2)dt_2\right\}dt_1.
\end{eqnarray*}
and 
$R_3(x,t)\to 0$ on any bounded subset of $\{(x,t):x\in\R,t>0\}$.
\end{prop}
\begin{proof}
\[T_3=\frac{i\kappa_4}{6n^{2}}\int_0^t\sum_{j,k\in B_1}\mathbb E\{D^3_{jk}(U_{jk}(t_1)e_n^{\circ}(x))\}dt_1+\widetilde T_3,\]
where
\[\widetilde T_3=\frac{i}{6n^{2}}\int_0^t\sum_{j\in B_1}(\kappa_{4,jj}-\kappa_4)\mathbb E\{D^3_{jj}(U_{jj}(t_1)e_n^{\circ}(x))\}dt_1.\]
By Proposition \ref{lemforDl} and (\ref{kappa34}), we have $|\widetilde T_3|\leq |t|C_3(x,t)n_1/n^2$.

The third derivative (l=3)
\begin{eqnarray}
\lefteqn{D^3_{jk}(U_{jk}(t_1)e_n^{\circ}(x))=\beta_{jk}^3\times}\nonumber\\
&&\{-i(36U_{jj}*U_{jk}*U_{jk}*U_{kk}+6U_{jj}*U_{jj}*U_{kk}*U_{kk}+6U_{jk}*U_{jk}*U_{jk}*U_{jk})(t_1)e_n^{\circ}(x)\nonumber\\
&&+6(6U_{jj}*U_{jk}*U_{kk}+2U_{jk}U_{jk}*U_{jk})(t_1)xe_n(x)\left(\int t U_{jk}(t)\alpha\widehat{\varphi_1}(t)dt\right.\nonumber\\
&&\left.+1_{j,k\in B_2}\int t\widetilde U_{jk}\beta\widehat{\varphi_2}(t)dt\right)\nonumber\\
&&+12i(U_{jj}*U_{kk}+U_{jk}*U_{jk})(t_1)x^2e_n(x)\left(\int t U_{jk}(t)\alpha\widehat{\varphi_1}(t)dt+1_{j,k\in B_2}\int t\widetilde U_{jk}\beta\widehat{\varphi_2}(t)dt\right)^2\nonumber\\
&&+6(U_{jj}*U_{kk}+U_{jk}*U_{jk})(t_1)xe_n(x)\left(\int t (U_{jj}*U_{kk}+U_{jk}*U_{jk})(t)\alpha\widehat{\varphi_1}(t)dt\right.\nonumber\\
&&\left.+1_{j,k\in B_2}\int t(\widetilde U_{jj}*\widetilde U_{kk}+ \widetilde U_{jk}*\widetilde U_{jk})\beta\widehat{\varphi_2}(t)dt\right)\nonumber\\
&&-8U_{jk}(t_1)x^3e_n(x)\left(\int t U_{jk}(t)\alpha\widehat{\varphi_1}(t)dt+1_{j,k\in B_2}\int t\widetilde U_{jk}\beta\widehat{\varphi_2}(t)dt\right)^3\nonumber\\
&&+12iU_{jk}(t_1)x^2e_n(x)\left(\int t U_{jk}(t)\alpha\widehat{\varphi_1}(t)dt+1_{j,k\in B_2}\int t\widetilde U_{jk}\beta\widehat{\varphi_2}(t)dt\right)\nonumber\\
&&\times\left(\int t (U_{jj}*U_{kk}+U_{jk}*U_{jk})(t)\alpha\widehat{\varphi_1}(t)dt+1_{j,k\in B_2}\int t(\widetilde U_{jj}*\widetilde U_{kk}+\widetilde U_{jk}*\widetilde U_{jk})\beta\widehat{\varphi_2}(t)dt\right)\nonumber\\
&&+2U_{jk}(t_1)xe_n(x)\left[\int t(6U_{jj}*U_{jk}*U_{kk}+2U_{jk}*U_{jk}*U_{jk})(t)\alpha\widehat{\varphi_1}(t)dt\right.\nonumber\\
&&\left.+1_{j,k\in B_2}\int t(6\widetilde U_{jj}*\widetilde U_{jk}*\widetilde U_{kk}+2\widetilde U_{jk}*\widetilde U_{jk}*\widetilde U_{jk})(t)\beta\widehat {\varphi_2}(t)dt\right]\}\label{thirdD}.
\end{eqnarray}
So any term of
\[\frac{i\kappa_4}{6n^{2}}\int_0^t\sum_{j,k\in B_1}\mathbb E\{D^3_{jk}(U_{jk}(t_1)e_n^{\circ}(x))\}dt_1\]
containing at least one off-diagonal entry $U_{jk}$ or $\widetilde U_{jk}$ is bounded by $C_3(x,t)n_1/n^2$. Let $R_3(x,t)$ be the sum of $\widetilde T_3$ and these terms. Then $|R_3(x,t)|\leq C_3(x,t)n_1/n^2+|t|C_3(x,t)n_1/n^2$.
So two terms in (\ref{thirdD}) containg diagonal entries of $U$ and $\widetilde U$ only left contribute to $T_3$. They are
$T_{31}$ and $T_{32}$. 
\end{proof}
Let
\[v(t)=\frac{1}{2\pi\gamma_1}\int_{-2\sqrt{\gamma_1}}^{2\sqrt{\gamma_1}}e^{it\lambda}\sqrt{4\gamma_1-\lambda^2}d\lambda,\ \  \widetilde v(t)=\frac{1}{2\pi\gamma_2}\int_{-2\sqrt{\gamma_2}}^{2\sqrt{\gamma_2}}e^{it\lambda}\sqrt{4\gamma_2-\lambda^2}d\lambda.\]
By Wigner semicircle law, one has
\[\lim_{n\to\infty}\bar v_n(t)=\gamma_1v(t),\ \  \ \lim_{n\to\infty}\bar {\widetilde v}_n(t)=\gamma_2\widetilde v(t).\]
Then
\begin{equation}
(v*v)(t)=-\frac{i}{2\pi\gamma_1^2}\int_{-2\sqrt{\gamma_1}}^{2\sqrt{\gamma_1}}e^{it\mu}\mu\sqrt{4\gamma_1-\mu^2}d\mu=\frac{1}{\pi t\gamma_1^2}\int_{-2\sqrt{\gamma_1}}^{2\sqrt{\gamma_1}}e^{it\mu}\frac{2\gamma_1-\mu^2}{\sqrt{4\gamma_1-\mu^2}}d\mu,
\end{equation}
\begin{equation}
(\widetilde v*\widetilde v)(t)=-\frac{i}{2\pi\gamma_2^2}\int_{-2\sqrt{\gamma_2}}^{2\sqrt{\gamma_2}}e^{it\mu}\mu\sqrt{4\gamma_2-\mu^2}d\mu=\frac{1}{\pi t\gamma_2^2}\int_{-2\sqrt{\gamma_2}}^{2\sqrt{\gamma_2}}e^{it\mu}\frac{2\gamma_2-\mu^2}{\sqrt{4\gamma_2-\mu^2}}d\mu.
\end{equation}
Let 
\begin{equation}
I(t)=\int_0^t(v*v)(t_1)dt_1,\ \ \widetilde I(t)=\int_0^t(\widetilde v*\widetilde v)(t_1)dt_1.
\end{equation}
Denote
\begin{equation}
B_{\varphi_l}=\frac{1}{\pi \gamma_l^2}\int_{-2\sqrt{\gamma_l}}^{2\sqrt{\gamma_l}}\varphi_l(\mu)\frac{2\gamma_l-\mu^2}{\sqrt{4\gamma_l-\mu^2}}d\mu,l=1,2.
\end{equation}
\begin{prop}
\begin{eqnarray}
T_{31}\to i\kappa_4xZ(x)I(t)\alpha  \gamma_1^2B_{\varphi_1},
\label{T_31}\\
T_{32}\to i\kappa_4xZ(x)I(t)\beta\gamma_{12}^2 B_{\varphi_2}.
\label{T_32}
\end{eqnarray}
uniformly on any bounded subset of $\{(x,t):x\in\R,t>0\}$.
\end{prop}
\begin{proof}
The proof of (\ref{T_31}) can be found in \cite{cite_key3}. And
\begin{eqnarray}
T_{32}&=&\frac{ix\kappa_4}{n^2}\int_0^t\sum_{j,k\in B_1\cap B_2}\int_0^{t_1}\int \int_0^{t_2}t_2\mathbb E\{U_{jj}(t_3)U_{kk}(t_1-t_3)\widetilde U_{jj}(t_4)\widetilde U_{kk}(t_2-t_4)e_n(x)\}\nonumber\\
&&\times\beta\widehat{\varphi_2}(t_2)dt_4dt_2dt_3dt_1\nonumber\\
&=&ix\kappa_4\int_0^t\int_0^{t_1}\int \int_0^{t_2}t_2\mathbb E\{v_n(t_3,t_4)v_n(t_1-t_3,t_2-t_4)e^{\circ}_n(x)\}\beta\widehat{\varphi_2}(t_2)dt_4dt_2dt_3dt_1\nonumber\\
&&+ix\kappa_4Z_n(x)\int_0^t\int_0^{t_1}\int \int_0^{t_2}t_2\mathbb E\{v_n(t_3,t_4)v_n(t_1-t_3,t_2-t_4)\}\beta\widehat{\varphi_2}(t_2)dt_4dt_2dt_3dt_1\nonumber\\
\end{eqnarray}
where
\begin{equation}
v_n(t_1,t_2)=n^{-1}\sum_{j\in B_1\cap B_2}U_{jj}(t_1)\widetilde U_{jj}(t_2).
\end{equation}
Then
\begin{equation}
|\E\{v_n(t_1,t_2)v_n(t_3,t_4)e^{\circ}_n(x)\}|\leq 4\E\{|v_n^{\circ}(t_1,t_2)|\}+4\E\{|v_n^{\circ}(t_3,t_4)|\},
\end{equation}
and
\begin{equation}
\E\{v_n(t_1,t_2)v_n(t_3,t_4)\}=\bar{v}_n(t_1,t_2)\bar{v}_n(t_3,t_4)+\E\{v_n(t_1,t_2)v^{\circ}_n(t_3,t_4)\},\end{equation}
where
\begin{equation}
\bar{v}_n(t_1,t_2)=\E\{v_n(t_1,t_2)\}.
\end{equation}
\begin{prop}
\[\bar{v}_n(t_1,t_2)=\gamma_{12}v(t_1)\widetilde v(t_2)+o(1),\]
uniformly on any compact set of $\R^2$.
\end{prop} 
\begin{proof}
Indeed, $\E\{U_{jj}(t_1)\* \widetilde{U}_{jj}(t_2)\}= v(t_1)\*\widetilde{v}(t_2) +o(1)$
uniformly in $1\leq j\leq n$ and $t_1, t_2$ from a compact set of $\R^2$, which follows from
\begin{eqnarray*}
 &&\E U_{jj}(t)=v(t) +o(1), \ \ \mathbf{Var} \{U_{jj}(t)\}=o(1),\\ 
 &&\E \widetilde U_{jj}(t)=\widetilde v(t) +o(1), \ \ \mathbf{Var} \{\widetilde U_{jj}(t)\}=o(1)
\end{eqnarray*} 
(see e.g. \cite{ORS}).
\end{proof}


So the limit of $T_{32}$ is
\[ix\kappa_4Z(x)\gamma_{12}^2\int_0^tv*v(t_1)dt_1\int_{-\infty}^{\infty} t_2\beta\widehat{\varphi_2}(t_2)\widetilde v*\widetilde v(t_2)dt_2=ix\kappa_4Z(x)\gamma_{12}^2I(t)\beta B_{\varphi_2}.\]
\end{proof}
So if $Y(x,t)=\lim_{n\to\infty}Y_n(x,t)$, then $Y(x,t)$ satisfies
\[Y(x,t)+2\gamma_1\int_{-\infty}^{\infty}dt_1\int_0^{t_1}v(t_1-t_2)Y(x,t_2)dt_2=xZ(x)\left[A(t)+i\kappa_4I(t)(\alpha \gamma_1^2B_{\varphi_1}+\beta\gamma_{12}^2 B_{\varphi_2})\right].\]
Therefore, if let $Y^*(x,t)$ be the solution of 
\[Y(x,t)+2\gamma_1\int_{-\infty}^{\infty} dt_1\int_0^{t_1}v(t_1-t_2)Y(x,t_2)dt_2=xZ(x)A(t),\]
then
\begin{equation}
Y(x,t)=Y^*(x,t)+\frac{i\kappa_4xZ(x)}{2\pi \gamma_1^2}\left[\alpha \gamma_1^2B_{\varphi_1}+\beta\gamma_{12}^2 B_{\varphi_2}\right]\int_{-2\sqrt{\gamma_1}}^{2\sqrt{\gamma_1}}\frac{e^{it\lambda}(2\gamma_1-\lambda^2)}{\sqrt{4\gamma_1-\lambda^2}}d\lambda.
\end{equation}
Symmetrically,
\begin{equation}
\widetilde Y(x,t)=\widetilde Y^*(x,t)+\frac{i\kappa_4xZ(x)}{2\pi\gamma_2^2}\left[\alpha\gamma_{12}^2 B_{\varphi_1}+\beta \gamma_2^2B_{\varphi_2}\right]\int_{-2\sqrt{\gamma_2}}^{2\sqrt{\gamma_2}}\frac{e^{it\lambda}(2\gamma_2-\lambda^2)}{\sqrt{4\gamma_2-\lambda^2}}d\lambda.
\end{equation}
Therefore,
\begin{eqnarray}
Z'(t)&=&i\alpha\int_{-\infty}^{\infty}\widehat{\varphi_1}(t)Y(x,t)dt+i\beta\int_{-\infty}^{\infty}\widehat{\varphi_2}(t)\widetilde Y(x,t)dt\nonumber\\
&=&-xVZ(x)-\alpha\frac{\kappa_4xZ(x)}{2\pi \gamma_1^2}\int_{-\infty}^{\infty}\widehat {\varphi_1}(t)\left[\alpha\gamma_1^2B_{\varphi_1}+\beta\gamma_{12}^2 B_{\varphi_2}\right]\int_{-2\sqrt{\gamma_1}}^{2\sqrt{\gamma_1}}\frac{e^{it\lambda}(2\gamma_1-\lambda^2)}{\sqrt{4\gamma_1-\lambda^2}}d\lambda dt\nonumber\\
&&-\beta\frac{\kappa_4xZ(x)}{2\pi\gamma_2^2}\int_{-\infty}^{\infty}\widehat {\varphi_2}(t)\left[\alpha\gamma_{12}^2 B_{\varphi_1}+\beta\gamma_2^2B_{\varphi_2}\right]\int_{-2\sqrt{\gamma_2}}^{2\sqrt{\gamma_2}}\frac{e^{it\lambda}(2\gamma_2-\lambda^2)}{\sqrt{4\gamma_2-\lambda^2}}d\lambda dt\nonumber\\
&=&-xVZ(x)-\alpha^2\frac{xZ(x)}{2} \gamma_1\gamma_2B_{\varphi_1}^2-\alpha\beta xZ(x)\gamma_{12}^2B_{\varphi_1}B_{\varphi_2}-\beta^2\frac{xZ(x)}{2} \gamma_1\gamma_2B_{\varphi_2}^2\nonumber\\
&=&-xVZ(x)-x\kappa_4Z(x)\left[\alpha^2\gamma_1^2\frac{B_{\varphi_1}^2}{2} +\alpha\beta \gamma_{12}^2B_{\varphi_1}B_{\varphi_2}+\beta^2\gamma_2^2\frac{B_{\varphi_2}^2}{2} \right]
\end{eqnarray}
where
\[V=
\alpha^2\mathbf{\mathbf{Var}}(G_1)  + 2\alpha\beta\mathbf{Cov}(G_1,G_2)+\beta^2\mathbf{Var}(G_2),\]
and $G_1, G_2$ are the random variables in Theorem \ref{thm-GaussCLT} with $d=2$.

Therefore,
\begin{eqnarray}
\lefteqn{\lim_{n\to\infty}\mathbf{Cov}(\mathcal N^{(1)\circ}_{n}[\varphi_1],\mathcal N^{(2)\circ}_{n}[\varphi_2])=}\nonumber\\
&&\mathbf{Cov}(G_1,G_2)+\frac{\gamma_{12}^2\kappa_4}{2\pi^2\gamma_1^2\gamma_2^2 }\int_{-2\sqrt{\gamma_1}}^{2\sqrt{\gamma_1}}\varphi_1(\mu)\frac{2\gamma_1-\mu^2}{\sqrt{4\gamma_1-\mu^2}}d\mu
\int_{-2\sqrt{\gamma_2}}^{2\sqrt{\gamma_2}}\varphi_2(\mu)\frac{2\gamma_2-\mu^2}{\sqrt{4\gamma_2-\mu^2}}d\mu.
\end{eqnarray}

By symmetry, for any $1\leq l\leq p\leq n$,
\begin{eqnarray}
\mathbf{Cov}(\tilde G_{l}, \tilde G_{p}) =\mathbf{Cov}(G_{l}, G_{p})+\frac{\gamma_{lp}^2\kappa_4}{2\pi^2\gamma_l^2\gamma_p^2 }\int_{-2\sqrt{\gamma_l}}^{2\sqrt{\gamma_l}}\varphi_l(\lambda)\frac{2\gamma_l-\lambda^2}{\sqrt{4\gamma_l-\lambda^2}}d\lambda
\int_{-2\sqrt{\gamma_p}}^{2\sqrt{\gamma_p}}\varphi_p(\mu)\frac{2\gamma_p-\mu^2}{\sqrt{4\gamma_p-\mu^2}}d\mu
\end{eqnarray}


\section{Appendix 1}
\begin{thm}[Decoupling Formula]\cite{cite_key3}
Let $\xi$ be a random variable such that $\E\{|\xi|^{p+2}\}<\infty$ for a certain nonnegative integer p. Then for any function $f:\mathbb{R}\to \mathbb{C}$ of the class $C^{p+1}$ with bounded derivatives $f^{(l)},l=1,...,p+1,$ we have
\begin{equation}
\E\{\xi f(\xi)\}=\sum_{l=0}^p\frac{\kappa_{l+1}}{l!}\E\{f^{(l)}(\xi)\}+\varepsilon_p.
\label{decf}
\end{equation}
where $\kappa_l$ denotes the $l$th cumulant of $\xi$ and the remainder term $\varepsilon_p$ admits the bound
\begin{equation}
|\varepsilon_p|\leq C_p\E\{|\xi|^{p+2}\}\sup_{t\in \mathbb{R}}f^{(p+1)}(t), \ C_p\leq\frac{1+(3+2p)^{p+2}}{(p+1)!}.
\label{epsilon}
\end{equation}
If $\xi$ is a Gaussian random variable with zero mean,
\begin{equation}
\E\{\xi f(\xi)\}=\E\{\xi^2\}\E\{f'(\xi)\}.
\label{dfgaus}
\end{equation}
\end{thm}

\section{Appendix 2}

Below we prove Proposition \ref{prop_diag} formulated in Section 3.
\begin{proof}
Since $  \langle x^{k}, \ x^{q} \rangle_{lr} = 0$ if $k+q$ is odd, it follows by linearity that
\begin{equation}
\langle U^{\gamma_{l}}_{k}, U^{\gamma_{r}}_{q} \rangle_{lr} = 0, \ \ \text{if} \ k+q \ \text{is odd.}
\end{equation}
We begin by computing $ \langle (\frac{x}{2\sqrt{\gamma_{l}}})^{2k}, U^{\gamma_{r}}_{2q} \rangle_{lr}$ and 
$ \langle (\frac{x}{2\sqrt{\gamma_{l}}})^{2k+1}, U^{\gamma_{r}}_{2q+1} \rangle_{lr}$. We obtain
\begin{eqnarray}
\lefteqn{\langle (\frac{x}{2\sqrt{\gamma_{l}}})^{2k}, U^{\gamma_{r}}_{2q}  \rangle_{lr} }\nonumber\\
&=&  (\frac{1}{\sqrt{\gamma_{l}}})^{2k} \langle x^{2k}, U^{\gamma_{r}}_{2q}(x) \rangle_{lr}\nonumber\\
&=&\gamma^{-k}_{l}\sum^{q}_{p=0} (-1)^{p} (\frac{1}{\sqrt{\gamma_{l}}})^{2q-2p}\binom{2q-p}{p}\langle x^{2k}, x^{2q-2p}  \rangle_{lr}\nonumber\\
&=&\frac{\gamma^{-k}_{l}\gamma^{-q}_{r}}{2k+1}\sum^{k}_{j=0}\sum^{q-j}_{p=0} \frac{(-1)^{p}\gamma^{p}_{l}(2j+1)^{2}}{2q-2p+1} \binom{2k+1}{k+j+1}\binom{2q-p}{p}\binom{2q-2p+1}{q-p+j+1}\gamma^{k-j}_{l}\gamma^{q-p-j}_{r}\gamma^{2j+1}_{lr}\nonumber\\
&=&\frac{1}{2k+1}\sum^{k}_{j=0}(2j+1)^{2}\binom{2k+1}{k+j+1}\left[\sum^{q-j}_{p=0}\frac{(-1)^{p}(2q-p)!}{p!(q-p+j+1)!(q-p-j)!}\right]\gamma^{-j}_{l}\gamma^{-j}_{r}\gamma^{2j+1}_{lr}\nonumber\\
\end{eqnarray}
and
\begin{eqnarray}
\lefteqn{\langle (\frac{x}{2\sqrt{\gamma_{l}}})^{2k+1}, U^{\gamma_{r}}_{2q+1} \rangle_{lr}}\nonumber\\
 &=&  (\frac{1}{\sqrt{\gamma_{l}}})^{2k+1} \langle x^{2k+1}, U^{\gamma_{r}}_{2q+1}(x) \rangle\nonumber\\
&=&(\frac{1}{\sqrt{\gamma_{l}}})^{2k+1}\sum^{q}_{p=0} (-1)^{p} (\frac{1}{\sqrt{\gamma_{r}}})^{2q-2p+1}\binom{2q-p+1}{p}\langle x^{2k+1}, x^{2q-2p+1}  \rangle_{lr}\nonumber\\
&=&\frac{\gamma^{-\frac{1}{2}}_{l}\gamma^{-\frac{1}{2}}_{r}}{2k+2}\sum^{k}_{j=0}\sum^{q-j}_{p=0} \frac{(-1)^{p}\gamma^{p}_{r}(2j+2)^{2}}{2q-2p+2} \binom{2k+2}{k+j+2}\binom{2q-p+1}{p}\binom{2q-2p+2}{q-p+j+2}\gamma^{-j}_{l}\gamma^{-p-j}_{r}\gamma^{2j+2}_{lr}\nonumber\\
&=&\frac{\gamma^{-\frac{1}{2}}_{l}\gamma^{-\frac{1}{2}}_{r}}{2k+2}\sum^{k}_{j=0}(2j+2)^{2}\binom{2k+2}{k+j+2}\left[\sum^{q-j}_{p=0}\frac{(-1)^{p}(2q-p+1)!}{p!(q-p+j+2)!(q-p-j)!}\right]\gamma^{-j}_{l}\gamma^{-j}_{r}\gamma^{2j+2}_{lr}.\nonumber\\
\end{eqnarray}

Denote by
\begin{equation}
H_{1}(q,j) = \sum^{q-j}_{p=0}\frac{(-1)^{p}(2q-p)!}{p!(q-p+j+1)!(q-p-j)!},
\label{eqn:H1}
\end{equation}
\begin{equation}
H_{2}(q,j) = \sum^{q-j}_{p=0}\frac{(-1)^{p}(2q-p+1)!}{p!(q-p+j+2)!(q-p-j)!}.
\label{eqn:H2}
\end{equation}
Then 
\begin{equation}
\langle (\frac{x}{2\sqrt{\gamma_{l}}})^{2k}, U^{\gamma_{r}}_{2q}  \rangle_{lr} = \frac{1}{2k+1}\sum^{k}_{j=0}(2j+1)^{2}\binom{2k+1}{k+j+1}H_{1}(q,j) \gamma^{-j}_{l}\gamma^{-j}_{r}\gamma^{2j+1}_{lr},
\end{equation}
\begin{equation}
\langle (\frac{x}{2\sqrt{\gamma_{l}}})^{2k+1}, U^{\gamma_{r}}_{2q+1} \rangle_{lr} = \frac{\gamma^{-\frac{1}{2}}_{l}\gamma^{-\frac{1}{2}}_{r}}{2k+2}\sum^{k}_{j=0}(2j+2)^{2}\binom{2k+2}{k+j+2} H_{2}(q,j)\gamma^{-j}_{l}\gamma^{-j}_{r}\gamma^{2j+2}_{lr}.
\end{equation}

It follows from (\ref{eqn:H1}-\ref{eqn:H2}) that
\begin{equation}
H_{1}(q,j) = \frac{(2q)!}{(q-j)!(q+j+1)!} = {}_2F_1\left(\begin{array}{c}-(q-j),-(q+j+1)\\-2q\end{array};1\right),
\end{equation}
\begin{equation}
H_{2}(q,j) = \frac{(2q+1)!}{(q-j)!(q+j+1)!} = {}_2F_1\left(\begin{array}{c}-(q-j),-(q+j+2)\\-(2q+1)\end{array};1\right),
\end{equation}
where $_2 F_1$ is a hypergeometric function. See \cite{cite_key5} for the definition of hypergeometric functions. Below let $(x)_{n} = x(x+1)\cdots(x+n-1)$ denote the rising factorial.  By the Chu-Vandermonde identity (see e.g. \cite{cite_key5}), it follows that

\begin{equation}
H_{1}(q,j) = \frac{(2q)!}{(q-j)!(q+j+1)!}  \ \frac{(j-q+1)_{q-j}}{(-2q)_{q-j}} = \left\{
     \begin{array}{lr}
       0 &  0 \leq j < q\\
       \frac{1}{2q+1} & j=q
     \end{array}
   \right.
\end{equation}
\begin{equation}
H_{2}(q,j) = \frac{(2q+1)!}{(q-j)!(q+j+2)!}  \ \frac{(j-q+1)_{q-j}}{(-2q-1)_{q-j}} = \left\{
     \begin{array}{lr}
       0 &  0 \leq j < q\\
       \frac{1}{2q+2} & j=q
     \end{array}
   \right.
\end{equation}
Therefore, for $ k = 0, 1, \cdots, q-1, $ we get that  $\langle (\frac{x}{2\sqrt{\gamma_{l}}})^{2k}, U^{\gamma_{r}}_{2q} \rangle_{lr} =0$ and also $ \ \langle (\frac{x}{2\sqrt{\gamma_{l}}})^{2k+1}, U^{\gamma_{r}}_{2q+1} \rangle_{lr} = 0$. With $ k=q$ we obtain
\begin{eqnarray}
\langle (\frac{x}{2\sqrt{\gamma_{l}}})^{2k}, U^{\gamma_{r}}_{2k} \rangle_{lr} &=& \frac{1}{2k+1}\sum^{k}_{j=0}(2j+1)^{2}\binom{2k+1}{k+j+1}H_{1}(k,j) \gamma^{-j}_{l}\gamma^{-j}_{r}\gamma^{2j+1}_{lr}\nonumber\\
&=&\frac{(2k+1)^{2}}{2k+1} \binom{2k+1}{2k+1}H_{1}(k,k)\gamma^{-k}_{l}\gamma^{-k}_{r}\gamma^{2k+1}_{lr}\nonumber\\
&=&\frac{\gamma^{2k+1}_{lr}}{\gamma^{k}_{l}\gamma^{k}_{r}}
\end{eqnarray}
and
\begin{eqnarray}
\langle (\frac{x}{2\sqrt{\gamma_{l}}})^{2k+1}, U^{\gamma_{r}}_{2k+1} \rangle_{lr} &=& \frac{\gamma^{-\frac{1}{2}}_{l}\gamma^{-\frac{1}{2}}_{r}}{2k+2}\sum^{k}_{j=0}(2j+2)^{2}\binom{2k+2}{k+j+2}H_{2}(k,j) \gamma^{-j}_{l}\gamma^{-j}_{r}\gamma^{2j+2}_{lr}\nonumber\\
&=&\gamma^{-\frac{1}{2}}_{l}\gamma^{-\frac{1}{2}}_{r}\frac{(2k+1)^{2}}{2k+1} \binom{2k+1}{2k+1}H_{2}(k,k)\gamma^{-k}_{l}\gamma^{-k}_{r}\gamma^{2k+1}_{lr}\nonumber\\
&=&\frac{\gamma^{2k+2}_{lr}}{\sqrt{\gamma_{l}\gamma_{r}}^{2k+1}}.
\end{eqnarray}

Thus, for $ k < q$,
\begin{equation}
\langle U^{\gamma_{l}}_{2k}, U^{\gamma_{r}}_{2q} \rangle_{lr} = 0, \ \ \langle U^{\gamma_{l}}_{2k+1}, U^{\gamma_{r}}_{2q+1} \rangle_{lr} = 0,
\end{equation}
and for $ k=q$
\begin{equation}
\langle U^{\gamma_{l}}_{2k}, U^{\gamma_{r}}_{2k} \rangle_{lr} = \langle (\frac{x}{2\sqrt{\gamma_{l}}})^{2k}, U^{\gamma_{r}}_{2k} \rangle_{lr} = \frac{\gamma^{2k+1}_{lr}}{\gamma^{k}_{l}\gamma^{k}_{r}},
\end{equation}
\begin{equation}
\langle U^{\gamma_{l}}_{2k+1}, U^{\gamma_{r}}_{2k+1} \rangle_{lr} = \langle (\frac{x}{2\sqrt{\gamma_{l}}})^{2k+1}, U^{\gamma_{r}}_{2k+1} \rangle_{lr} = \frac{\gamma^{2k+2}_{lr}}{\sqrt{\gamma_{l}\gamma_{r}}^{2k+1}}.
\end{equation}
This completes the proof of proposition \ref{prop_diag}, which is the diagonalization part of lemma \ref{lem:diag}.
\end{proof}

\section{Appendix 3}
Below we prove Proposition \ref{lem:exten} formulated in Section 3.

\begin{proof}
First it will be argued by approximation that $ \langle \cdot, \cdot \rangle_{lr}$ can be extended to the class of functions $ \mathcal{H}_{\frac{3}{2} + \epsilon}$, and then the bilinear form will be explicitly computed. It will be sufficient to approximate $f,g$ below by truncated polynomials with rational coefficients in $ \mathcal{H}_{\frac{3}{2} + \epsilon}$, because of the estimate \ref{ineq-MS} . Recall that functions of the Schwartz class  are dense in $\mathcal{H}_{s}$, so after a triangle inequality argument it is in fact sufficient to suppose that $f,g \in \mathcal{S}(\mathbb{R})$. Let $ h \in \mathcal{C}^{\infty}_{c}$ be a function so that $ h(x) = 1$ for $ x \in [-3, 3],  \ h(x) = 0$ for $ x \notin [-4,4]$ and is smoothly interpolated in between.  Note that with overwhelming probability, the eigenvalues of the submatrices concentrate in the support of $\mu_{sc}$. As a consequence we may suppose that $f,g$ are supported in $[-3,3]$. 
\indent
We give a density argument. It is sufficient to argue that $ || hf - hp_{j}||_{\frac{3}{2} + \epsilon}$ and $ || hg - hq_{j}||_{\frac{3}{2} + \epsilon}$ converge to $0$ as $j \to \infty$, where $\{p_{j}\}, \{q_{j}\}$ are appropriately chosen sequences of polynomials with rational coefficients. Note that $hf = f$ and $hg = g$. We now focus on estimating $ || f - hp_{j}||_{\frac{3}{2} + \epsilon}$. Since $f$ is a Schwartz function, we have that $f \in \mathcal{H}_{2}$. We note that 
 \begin{equation}
\int^{\infty}_{-\infty} |\widehat{f}(t)|^{2}\left(1 + |t|\right)^{3+\epsilon} dt  \leq \int^{\infty}_{-\infty} |\widehat{f}(t)|^{2}\left(1 + |t|\right)^{4} dt, 
\end{equation}
so it will be sufficient to approximate $f$ in the larger $||\cdot ||_{2}$ norm. Also, since 
\begin{equation}
||f||^{2}_{2} = \int^{\infty}_{-\infty} |\widehat{f}(t)|^{2}\left(1 + |t|\right)^{4} dt \leq Const \left[ \int^{\infty}_{-\infty} |\widehat{f}(t)|^{2}dt + \int^{\infty}_{-\infty} t^{4}|\widehat{f}(t)|^{2}dt \right],
\end{equation}
we only need to approximate the two terms on the right hand side. Consider polynomials $\{p_{j}\}$ with rational coefficients so that $\sup_{-4 \leq x \leq 4} | f''(x) - p_{j}(x)| \to 0$ as $ j \to \infty$. Then denote by $ \tilde{p}_{j}(x) = \int^{x}_{-4} p_{j}(t)dt$, and  $\tilde{ \tilde{p}}_{j}(x) = \int^{x}_{-4} \tilde{p}_{j}(t)dt$.  As a consequence of Parseval's theorem, it will be sufficient to show that

\begin{equation}
 || f - h\tilde{\tilde{p}}_{j} ||_{L^{2}([-4,4])} \to 0 \  \text{and } \ || f'' - (h\tilde{\tilde{p}}_{j})'' ||_{L^{2}([-4,4])} \to 0,  \ \text{as} \ j \to \infty.
\end{equation} 
 
 But observe that 
\begin{equation}
 || f'' - (h\tilde{\tilde{p}}_{j})''  ||_{L^{2}([-4,4])} \ \leq \ || f'' - hp_{j}  ||_{L^{2}([-4,4])} +  || h''\tilde{\tilde{p}}_{j} + 2h'\tilde{p}_{j}||_{L^{2}([-4,4])}. 
\end{equation}

The first term on the right hand side converges to $0$ because of the uniform approximation. Noting that $h'(x) = 0$ and $ h''(x) = 0$ on $(-3,3)$, and also that $ \tilde{p}_{j}$ and $ \tilde{\tilde{p}}_{j}$ converge to $0$ uniformly on $[-4,-3) \cup (3,4]$, it follows that the second term on the right hand side converges to $0$ as well. Finally we observe that
\begin{eqnarray}
|| f - h\tilde{\tilde{p}}_{j} ||^{2}_{L^{2}([-4,4])} \nonumber &=& \int^{4}_{-4} |f(x) - h(x)\tilde{\tilde{p}}_{j}(x)|^{2}dx\\ 
&\leq& \int^{4}_{-4} h^{2}(x) \left| \int^{x}_{-4}\int^{t}_{-4} [ f''(u) - p_{j}(u)]dudt\right|^{2}dx \nonumber\\
&\leq& Const \cdot \left(\sup_{-4\leq u \leq 4} \left|   f''(u) - p_{j}(u)   \right|\right)^{2}
\end{eqnarray}
It follows that $ || f - h\tilde{\tilde{p}}_{j} ||^{2}_{L^{2}([-4,4])} \to 0$ because of the uniform approximation. This completes the approximation argument, so we now turn toward computing the bilinear form.

\indent
Setting
\begin{equation}
f_{k} = \frac{1}{2\pi \gamma_{l}} \int^{2\sqrt{\gamma_{l}}}_{-2\sqrt{\gamma_{l}}}f(x)U^{\gamma_{l}}_{k}(x)\sqrt{4\gamma_{l}-x^{2}}dx, \ \ g_{k} = \frac{1}{2\pi \gamma_{r}} \int^{2\sqrt{\gamma_{r}}}_{-2\sqrt{\gamma_{r}}}g(y)U^{\gamma_{r}}_{k}(y)\sqrt{4\gamma_{r}-y^{2}}dy,
\end{equation}
it follows that
\begin{eqnarray}
\lefteqn{\langle f, g \rangle_{lr}}\nonumber\\
  &=&\langle \ \sum^{\infty}_{k=0}f_{k}U^{\gamma_{l}}_{k}(x), \ \sum^{\infty}_{p=0} g_{p} U^{\gamma_{r}}_{p}(x) \ \rangle_{lr}\nonumber\\
&=& \sum^{\infty}_{k=0}\sum^{\infty}_{p=0}f_{k}g_{p}\langle U^{\gamma_{l}}_{k}, U^{\gamma_{r}}_{p} \rangle_{lr}\nonumber\\
&=&\sum^{\infty}_{k=0}f_{k}g_{k}\frac{\gamma^{k+1}_{lr}}{\gamma^{k/2}_{l}\gamma^{k/2}_{r}}\nonumber\\
&=& \frac{1}{4\pi^{2}\gamma_{l}\gamma_{r}} \int^{2\sqrt{\gamma_{l}}}_{-2\sqrt{\gamma_{l}}} \int^{2\sqrt{\gamma_{r}}}_{-2\sqrt{\gamma_{r}}} f(x)g(y)\left[ \sum^{\infty}_{k=0}  U^{\gamma_{l}}_{k}(x)U^{\gamma_{r}}_{k}(y)\frac{\gamma^{k+1}_{lr}}{\gamma^{k/2}_{l}\gamma^{k/2}_{r}} \right] \sqrt{4\gamma_{l}-x^{2}}\sqrt{4\gamma_{r}-y^{2}}dydx.\nonumber
\end{eqnarray}
It also follows, using (\ref{cheby-lim-as}),  that a.s.
\begin{eqnarray}
\lefteqn{\lim_{n \rightarrow \infty} \frac{1}{n} \Tr\left\{ P^{(l)} f(M^{(l)})\cdot P^{(l,r)}\cdot g(M^{(r)})P^{(r)} \right\}}\nonumber\\
&=& \frac{1}{4\pi^{2}\gamma_{l}\gamma_{r}} \int^{2\sqrt{\gamma_{l}}}_{-2\sqrt{\gamma_{l}}} \int^{2\sqrt{\gamma_{r}}}_{-2\sqrt{\gamma_{r}}} f(x)g(y)\left[ \sum^{\infty}_{k=0}  U^{\gamma_{l}}_{k}(x)U^{\gamma_{r}}_{k}(y)\frac{\gamma^{k+1}_{lr}}{\gamma^{k/2}_{l}\gamma^{k/2}_{r}} \right] \sqrt{4\gamma_{l}-x^{2}}\sqrt{4\gamma_{r}-y^{2}}dydx.\nonumber\\
\end{eqnarray}
Proposition \ref{lem:exten} follows.
\end{proof}


\bibliographystyle{plain}

\begin{thebibliography}{10}

\bibitem{cite_key1} G.W. Anderson, A. Guionnet, O. Zeitouni.    \emph{An Introduction to Random Matrices.} Cambridge University Press, 2010.

\bibitem{Bandmodel} G.~W. Anderson and O.~Zeitouni.
\emph{A CLT for a Band Matrix Model.}
Probab. Theory Relat. Fields, 134:283--338, 2006.

\bibitem{cite_key5} G. Andrews, R. Askey, R. Roy. \emph{Special Functions.} Cambridge University Press, 1999.

\bibitem{B.spectral} Z.~D. Bai, X.~Wang, and W.~Zhou.
\emph{CLT for Linear Spectral Statistics of Wigner Matrices.}
Electronic Journal of Probability, 14(83):2391--2417, 2009.

\bibitem{BG}  G. Ben Arous  and A. Guionnet. \emph{Wigner Matrices}, 
in {\em Oxford Handbook on Random Matrix Theory.} edited by 
Akemann G., Baik J. and Di Francesco P., Oxford University Press, New York, 2011.

\bibitem{BGS} O. Bohigas,  M.J. Giannoni,  C. Schmit.
\emph{Characterization of Chaotic Quantum Spectra and Universality of Level Fluctuation Laws.} 
Phys. Rev. Lett. 52 (1), 1--4, 1984. 


\bibitem{cite_key2} A. Borodin, \emph{CLT for spectra of submatrices of Wigner random matrices.} 
Mosc. Math. J., 14(1): 1, 29--38, 2014.		

\bibitem{B2} A. Borodin, \emph{CLT for spectra of submatrices of Wigner random matrices II. Stochastic evolution.}
Random Matrix Theory, Interacting Particles, and Integrable Systems, 57--69, 2014.


\bibitem{Chow} G.P Chow 
\emph{Analysis and Control of Dynamic Economic Systems.}
Wiley, New York,1976.


\bibitem{Cek} V. Cekanavicius, \emph{Approximation Methods in Probability Theory.}
Springer, 2016.


\bibitem{ER} A. Edelman, N.R. Rao. 
\emph{Random matrix theory.}
Acta Numerica. 14: 233-297, 2005.


\bibitem{FK} F. Franchini, V. Kravtsov.
\emph{Horizon in random matrix theory, the Hawking radiation, and flow of cold atoms.}
Phys. Rev. Lett. 103 (16): 166401, 2009.

\bibitem{Johansson1} K. Johansson.
\emph{On Fluctuations of Eigenvalues of Random Hermitian Matrices.}
Duke Mathematical Journal, 91(1):151--204, 1998.

\bibitem{Joh} K. Johansson.
\emph{Random Growth and Random Matrices.},
Proceedings of the 2000 European Congress of Mathematics, Progress in Mathematics, 201, 445--456, 2001.


\bibitem{Jon} D. Jonsson, \emph{Some limit theorems for the eigenvalues of a sample covariance matrix.}
J. Mult. Anal. 12, 1-38, 1982.


\bibitem{John} I.M. Johnstone.
\emph{High dimensional statistical inference and random matrices.}
Proceedings of the International Congress of Mathematicians, Madrid, Spain, 2006.


\bibitem{Har} J. Harnad (Ed.) \emph{Random Matrices, Random Processes and Integrable Systems.}
CRM Series in Mathematical Physics, Springer, 2011.

\bibitem{Keat} J. Keating. \emph{The Riemann zeta-function and quantum chaology.}
Proc. Internat. School of Phys. Enrico Fermi. CXIX: 145-185, 1993.

\bibitem{LS} L. Li, A. Soshnikov, \emph{Central Limit Theorem for Linear Statistics of Eigenvalues of Band Random Matrices.}
Random Matrices: Theory and Applications, 2(4), 1350009, 50 pages, 2013.


\bibitem{LodS} A. Lodhia, N. J. Simm, \emph {Mesoscopic linear statistics of Wigner matrices.}
available at arXiv:1503.03533 [math.PR].

\bibitem{cite_key3} A. Lytova, L. Pastur. 
\emph{ Central limit theorem for linear eigenvalue statistics of random matrices with independent entries.} Annals of Probability 37(5): 1778--1840, 2009.

\bibitem{lp} A.Lytova, L. Pastur. \emph{Fluctuations of Matrix Elements of Regular Functions of Gaussian Random Matrices.} 
J. Stat. Phys., 134, 147--159, 2009.


\bibitem{Meh} M.L. Mehta.
\newblock {\em Random Matrices.}
\newblock Elsevier Publishing, 2004.


\bibitem{ORS} S, O'Rourke, D. Renfrew, A. Soshnikov. 
\emph{On Fluctuations of Matrix Entries of Regular Functions of Wigner Matrices with Non-identically Distributed Entries.} 
Journal of Theoretical Probability  26(3): 750--780, 2013.


\bibitem{PRS} A, Pizzo, D. Renfrew, A. Soshnikov. \emph{Fluctuations of Matrix Entries of Regular Functions of Wigner Matrices.} 
Annales de l'Institut Henri Poincare (B) Probabilites et Statistiques, 49(1): 64--94, 2013.


\bibitem{Rom} D. Romik, \emph{The surprising mathematics of longest Increasing subsequences.}
Cambridge University Press, 2015.

\bibitem{SB} D. Sanchez D, M. Buettiker.  
\emph{Magnetic-field asymmetry of nonlinear mesoscopic transport.} Phys. Rev. Lett. 93 (10): 106802, 2004.


\bibitem{cite_key4} M. Shcherbina.
\emph{Central Limit Theorem for linear eigenvalue statistics of the Wigner and sample covariance random matrices.}
Journal of Mathematical Physics, Analysis, Geometry,  7(2):176--192, 2011.


\bibitem{shch1} M. Shcherbina, B. Tirozzi.
\emph{Central limit theorem for fluctuations of linear eigenvalue statistics of large random graphs.} J. Math. Phys. 51, 023523, 20, 2010.


\bibitem{shch2} M. Shcherbina. \emph{On fluctuations of eigenvalues of random band matrices.} J. Stat. Phys. 161, 1, 73--90, 2015.


\bibitem{SinSo} Y. Sinai, A. Soshnikov, \emph{Central limit theorem for traces of large random symmetric matrices with independent matrix elements.}
Bol. Soc. Brasil. Mat. (N.S.) 29, 1-24, 1998.

\bibitem{Som} H. Sompolinsky, A. Crisanti, H. Sommers, H. \emph{Chaos in Random Neural Networks.}
Physical Review Letters. 61 (3): 259-262, 1988.


\bibitem{tao} T. Tao.
\newblock {\em Topics in Random Matrix Theory.}
\newblock American Mathematical Society, 2012.

\bibitem{Wig1} E.~P. Wigner.
\emph{On the statistical distribution of the widths and the spacings of nuclear resonance levels.}
Proc. Cambridge Phil. Soc., 47, 790--798, 1951.

\bibitem{Wig2} E.~P. Wigner.
\emph{On the distribution of the roots of certain symmetric matrices.}
Ann. Math., 67: 325--327, 1958.







\end{thebibliography}

\end{document}